\documentclass[12pt]{article}

\usepackage{amsmath}
\usepackage{amsthm}

\usepackage{amssymb}

\usepackage{caption, subcaption}%

\usepackage{xparse} 
\usepackage{stmaryrd} 

\usepackage{mathrsfs}

\usepackage{multirow}

\usepackage{geometry}
\usepackage{bm}

\usepackage{enumitem}

\usepackage{color}
\usepackage{graphicx}

\allowdisplaybreaks[3]
\def\jump#1{\llbracket #1 \rrbracket }

\newtheorem{theorem}{Theorem}[section]

\newtheorem{lemma}{Lemma}[section]

\newtheorem{remark}{Remark}[section]

\providecommand{\keywords}[1]
{
	\small	
	\textbf{\textit{Keywords:}} #1
}

\title{{\large  \bf  Uniformly High-Order Structure-Preserving Discontinuous Galerkin Methods for Euler Equations with Gravitation: Positivity and Well-Balancedness}} 

\author{{Kailiang Wu}\thanks{Department of Mathematics, The Ohio State University, Columbus, OH 43210, USA ({\tt wu.3423@osu.edu}).}{~~~and~~Yulong Xing}\thanks{Department of Mathematics, The Ohio State University, Columbus, OH 43210, USA ({\tt xing.205@osu.edu}). The work of Y. Xing is partially supported by the NSF grant DMS-1753581.}}

\date{May 13, 2020}

\begin{document}

	\maketitle

\vspace{-5mm}

\begin{abstract}
	This paper presents a class of novel high-order accurate discontinuous Galerkin (DG) schemes for the compressible Euler equations under gravitational fields. A notable feature of these schemes is that they are well-balanced for a general hydrostatic equilibrium state, and at the same time, provably preserve the positivity of density and pressure.  
	In order to achieve the well-balanced and positivity-preserving properties simultaneously, a novel DG spatial discretization is carefully designed with suitable source term reformulation and a properly modified Harten-Lax-van Leer-contact (HLLC) flux. 
	Based on some technical decompositions as well as several key properties of the admissible states and HLLC flux, rigorous positivity-preserving analyses are carried out. It is proven that the resulting well-balanced DG schemes, coupled with strong stability preserving time discretizations, satisfy a weak positivity property, which implies that one can apply a simple existing limiter to effectively enforce the positivity-preserving property, without losing high-order accuracy and conservation. 
	The proposed methods and analyses are applicable to the Euler system with {\em general} equation of state. 
	Extensive one- and two-dimensional numerical tests demonstrate the desired properties of these schemes, including the exact preservation of the equilibrium state, the ability to capture small perturbation of such state, the robustness for solving problems involving low density and/or low pressure, and good resolution for smooth and discontinuous solutions.
	
	\vspace{6mm}
\noindent	

\end{abstract}

\newpage

	\section{Introduction}
In this paper, we present highly accurate and robust numerical methods for 
the compressible Euler equations with gravitation, which has wide application in astrophysics and atmospheric science. 
In the $d$-dimensional case, this model can be written as the following nonlinear system of balance laws
\begin{equation}\label{eq:dD}
{\bf U}_t + {\bm \nabla} \cdot {\bf F} ( {\bf U} ) = {\bf S} ( {\bf U}, {\bf x} ),
\end{equation}
with
\begin{equation}\label{eq:dD1}
{\bf U} = \begin{pmatrix}
\rho
\\
{\bf m}
\\
E
\end{pmatrix},
\quad  
{\bf F} ( {\bf U} )= \begin{pmatrix}
\rho {\bf u}
\\
\rho {\bf u} \otimes {\bf u} + p {\bf I}_d
\\
(E+p) {\bf u}
\end{pmatrix},
\quad 
{\bf S}({\bf U},{\bf x}) = \begin{pmatrix}
0
\\
-\rho {\bm \nabla} \phi 
\\
-  {\bf m} \cdot {\bm \nabla} \phi
\end{pmatrix}.
\end{equation}
Here ${\bf m}=\rho {\bf u}$ denotes the momentum vector;  
$\rho,$ ${\bf u}$, and $p$ denote the fluid density, velocity and pressure, respectively; 
${\bf I}_d$ is the identity matrix of size $d$; $E=\frac12 \rho \| {\bf u} \|^2 + \rho e$ is the total non-gravitational energy with $e$ denoting the specific internal energy. The source terms at the right hand side of \eqref{eq:dD} represent the effect of the gravitational field, and $\phi({\bf x})$ is the static gravitational potential. 
An additional thermodynamic equation relating state variables, the so-called equation of state (EOS), is needed to close the system \eqref{eq:dD1}. 
A general EOS can be written as 
$e={\mathcal E}(\rho,p).$ 
For ideal gases it is given by
\begin{equation}\label{eq:IEOS}
p=(\gamma-1)\rho e=(\gamma - 1) \left( E - \frac{ \| {\bf m} \|^2 }{2\rho} \right),
\end{equation}
where the constant $\gamma>1$ denotes the adiabatic index. 
Although we will mainly focus on the ideal EOS for better legibility, the methods and analyses presented in this paper are readily extensible to general EOS 
as shown in Appendix \ref{app:gEOS}.

The equations \eqref{eq:dD} with \eqref{eq:IEOS} form a hyperbolic system of balance laws and admit (nontrivial) hydrostatic equilibrium solutions, in which the gravitational source term is exactly balanced by the flux gradient, with two well-known examples being the isothermal and polytropic equilibria.
The astrophysical and atmospheric applications often involve nearly equilibrium flows, which are small perturbation of the hydrostatic equilibrium states. 
Standard numerical methods may not balance the contribution of the flux and gravitational source terms, and generate large numerical error,
especially for a long-time simulation, e.g., in modeling star and galaxy formation. 
To address the issue, one may need to conduct the simulation on a very refined mesh, which can be  time-consuming especially for the multidimensional problems. 
To save the computational cost, well-balanced methods, which preserve exactly the discrete version of these steady-state solutions
up to machine accuracy, are designed to effectively capture these nearly equilibrium flows well on relatively coarse meshes. 
Study of well-balanced methods has attracted much attention over the past few decades.
Most of them were proposed for the shallow water equations over a non-flat bottom topology,
another prototype example of hyperbolic balance laws; see, e.g., 
\cite{BV1994,GL1996,L1998b,X2002,ABBKP2004,XS2005,XZS2010,XS2014} and the references therein.  
In recent years, well-balanced numerical methods for the Euler equations \eqref{eq:dD} with gravitation have been designed
within several different frameworks, including but not limited to the finite volume methods \cite{LB1998,BKLL2004,KM2014,CK2015,LiXingWBFV2016,KM2016,KlingenbergSISC2019,GROSHEINTZLAVAL2019324}, 
gas-kinetic schemes \cite{XLC2010,LXL2011}, 
finite difference methods \cite{XS2013,GM2016,LXCMA2018}, and finite element discontinuous Galerkin (DG) methods \cite{LiXingWBDG2016,CZ2017,LXJCP2018,Veiga2019}. 
Recently, comparison between high-order DG method and well-balanced DG methods was carried out in \cite{Veiga2019}.


Besides maintaining the hydrostatic equilibrium states, another numerical challenge for the system \eqref{eq:dD} is
to preserve the positivity of density and pressure. 
Such positivity property is not only necessary for the physical nature of the solution, but also 
crucial for the robustness of numerical computations. 
In fact, when negative density or/and pressure are produced, 
numerical instability can develop and cause the breakdown of numerical simulations. 
However, most high-order accurate schemes for the Euler equations with gravity are generally not positivity-preserving, and thus may suffer from a risk of failure when simulating problems with low density, low pressure and/or strong discontinuity. 
In recent years, high-order bound-preserving numerical schemes have been actively studied for hyperbolic systems. 
Most of them are built upon two types of limiting approaches: a simple scaling limiter \cite{zhang2010b} for the reconstructed or evolved solution polynomials in
finite volume/DG methods; see, e.g., \cite{zhang2010,zhang2010b,XZS2010,ZHANG2017301,Wu2017a,WuShu2019,wu2020provably},  
or a flux-correction limiter \cite{Xu2014,Hu2013,WuTang2015}. 
For more developments and applications, we refer to the recent review \cite{Shu2018} and the references therein.  
Based on the simple scaling limiter, high-order positivity-preserving DG schemes 
were constructed for the Euler equations without source term \cite{zhang2010b,zhang2012maximum} and with source terms including 
the gravitational source term \cite{zhang2011}. The bound-preserving framework was also extended in \cite{Wu2017} to the general relativistic Euler equations 
under strong gravitational fields. 

The main objective of this paper is to develop a class of uniformly high-order DG methods, 
which are well-balanced and at the same time provably positivity-preserving for the Euler equations with gravitation. 
Most of the existing methods possess only one of these two properties.
A recent work to satisfy both properties was studied in \cite{Thomann2019}, based on a new approximate Riemann solver using relaxation approach. 
The accuracy of the schemes in \cite{Thomann2019} was limited to second-order, yet its extension to higher-order is challenging.   
The framework established in this paper would be the first one, to our best knowledge, that achieves this goal with arbitrarily high-order accurate 
schemes. The efforts in this paper are summarized as follows. 	
\begin{enumerate}[leftmargin=*]
	\item One key novelty of this work is to devise novel high-order well-balanced DG schemes, 
	with suitable source term treatments and proper well-balanced numerical fluxes, so that  
	the desired positivity-preserving property is also accommodated in the discretization at the same time.  
	\item Our source term discretization is motivated by \cite{XS2013}, 
	where the gravitational source is first reformulated into an equivalent special form using the corresponding hydrostatic equilibrium solution. For the well-balancedness, 
	the reformulation can be made based on either the cell-centered solution values (cf.~\cite{LiXingWBDG2016}) or the cell average of the solution (cf.~\cite{LiXingWBFV2016}). 
	Our analysis indicates that the latter is advantageous for establishing the positivity-preserving property under a milder CFL condition; see Remark \ref{rem:cellave} for details. Besides, for the theoretical positivity-preserving considerations,  we also
	observe that the source term in the energy equation should be 
	discretized in a same fashion as in the momentum equations, which is not required for the well-balancedness consideration. 	
	\item 
	The Lax-Friedrichs (LF) flux is employed in \cite{LiXingWBDG2016,LiXingWBFV2016} and properly modified to be well-balanced 
	for two special equilibria (isothermal and polytropic equilibria), separately. 
	However, the modification in the polytropic case  
	makes it challenging, if not impossible, to prove the positivity-preserving property.
	In this paper, we will follow \cite{CK2015} and consider a Harten-Lax-van Leer-contact (HLLC) numerical flux, 
	which exactly resolves stationary contacts and has some significant advantages over the LF flux in the present study. 
	We will show in our framework that the HLLC flux can be properly modified, in a {\em unified} way, to be well-balanced 
	with our discrete source terms for an {\em arbitrary} hydrostatic equilibrium. 
	Moreover, it will be shown that such modification also retains the positivity-preserving property of 
	the HLLC flux and does not affect the high-order accuracy. 
	\item Based on some technical decompositions as well as several key properties of the admissible states and HLLC flux, 
	we will rigorously prove that the resulting well-balanced DG schemes satisfy a weak positivity property, which implies that 
	a simple existing limiter \cite{zhang2010b,wang2012robust} can effectively enforce the positivity-preserving property without losing high-order accuracy and conservation. The well-balanced modification of the numerical flux and discretization of source terms lead to additional difficulties in the positivity-preserving analyses, which are more complicated than 
	the analyses for the standard DG methods in \cite{zhang2010b,zhang2011}.
\end{enumerate}
It is also worth noting that, in the context of shallow water equations,
several positivity-preserving well-balanced schemes have been developed in the literature \cite{kurganov2007second,XZS2010,XS2011}. 
In that context, the positivity refers to the non-negativity of the water height. 
In the Euler equations \eqref{eq:dD}, the density is the analogue of the water height and is evolved only in the continuity equation, 
which makes it relatively easy to ensure its positivity. 
However, it is much more difficult to guarantee the positivity of pressure, since it depends nonlinearly on all the conservative variables $\{\rho,{\bf m},E\}$, as shown in \eqref{eq:IEOS}.  
More specifically, the pressure (internal energy) is computed by subtracting the kinetic energy ${ \| {\bf m} \|^2 }/{(2\rho)}$ from the total energy $E$.
For high Mach flows or very cold flows, when the numerical errors in $E$ and ${ \| {\bf m} \|^2 }/{(2\rho)}$ are large enough, negative pressure can be produced easily. 
Since the conservative quantities $\{\rho,{\bf m},E\}$ are evolved according to their own conservation laws which are seemingly unrelated, the positivity of pressure is not easy to guarantee numerically. 
In theory, it is indeed a challenge to make an a priori judgment on whether a numerical scheme is always positivity-preserving under all circumstances or not. 
For these reasons, seeking positivity-preserving well-balanced schemes for the Euler equations \eqref{eq:dD} with gravitation is quite nontrivial and cannot directly follow any existing frameworks on shallow water equations.  

The rest of this paper is organized as follows. In Section \ref{sec:setup}, we will introduce the stationary hydrostatic solutions 
of \eqref{eq:dD} and present several useful properties of the admissible state set and the HLLC flux. 
We first construct the positivity-preserving well-balanced DG schemes for the one-dimensional system in 
Section \ref{sec:method1D}, and then extend them to the multidimensional cases in Section \ref{sec:method2D}. 
We conduct numerical tests to verify the properties and effectiveness of the proposed schemes 
in Section \ref{sec:examples}, before concluding the paper in Section \ref{sec:conclusion}. 
The extensions of the proposed methods and analyses to general EOS are presented in 
Appendix \ref{app:gEOS}.  
For completeness of this work and comparison purpose,  
we also discuss in Appendix \ref{app:DGLF} the positivity of the well-balanced DG schemes with a modified LF flux for the isothermal case. 

\section{Auxiliary results} \label{sec:setup}
This section introduces the stationary hydrostatic solutions 
of \eqref{eq:dD} and presents several useful properties of the admissible state set and the HLLC flux. 


\subsection{Stationary hydrostatic solutions}\label{sec:steady}

Under the time-independent gravitation potential, the system \eqref{eq:dD} admit zero-velocity
stationary hydrostatic solutions of the form
\begin{equation}\label{eq:equi1}
\rho = \rho({\bf x}), \quad {\bf u}={\bf 0}, \quad {\bm \nabla} p = - \rho {\bm \nabla} \phi.
\end{equation}
Two important special equilibria arising in the applications are 
the polytropic \cite{KM2014} and isothermal \cite{XS2013} hydrostatic states.
For an isothermal hydrostatic state, we have $T({\bf x})\equiv T_0$, where $T$ denotes the temperature. For an ideal gas, it is given by 
\begin{equation*}
\rho = \rho_0 \exp \left( -\frac{\phi}{RT_0} \right),\qquad {\bf u}={\bf 0}, \qquad p = p_0 \exp \left( -\frac{\phi}{RT_0} \right),
\end{equation*}
where $R$ is the gas constant; $p_0$, $\rho_0$, and $T_0$ are positive constants satisfying 
$p_0=\rho_0 R T_0$.
A polytropic equilibrium is characterized by 
$p = K_0 p^\gamma,$ 
which leads to the form of 
\begin{equation*}
\rho = \left( \frac{\gamma-1}{ K_0 \gamma} ( C-\phi )   \right)^{\frac{1}{\gamma-1}}, \quad {\bf u}={\bf 0}, \quad  p = \frac{1}{ K_0^{\frac{1}{\gamma-1} } }   \left( \frac{\gamma-1}{\gamma} ( C-\phi )   \right)^{\frac{\gamma}{\gamma-1}},
\end{equation*}
where $K_0$ and $C$ are both constant. 

\subsection{Properties of admissible states}
In physics, the density $\rho$ and the pressure $p$ are both positive, which is equivalent to the description that 
the conservative vector ${\bf U}$ should stay in the set of physically admissible states, defined by
\begin{equation}\label{eq:DefG}
{G} := \left\{   {\bf U} = (\rho,{\bf m},E)^\top:~ \rho > 0,~
{\mathcal G}(  {\bf U}  ) :=  E-  \frac{\|{\bf m}\|^2}{2 \rho}   > 0 \right\},
\end{equation}
where ${\mathcal G}({\bf U})$ is a concave function of ${\bf U}$ if $\rho \ge 0$. 
It is easy to show that the admissible state set $G$ satisfies the following properties, which will be useful in our positivity-preserving analysis.  

\begin{lemma}[Convexity]\label{lem1}
	The set $G$ is a convex set. Moreover, $\lambda {\bf U}_1 + (1-\lambda) {\bf U}_0  \in G$ for any ${\bf U}_1 \in G, {\bf U}_0 \in \overline G$ and $\lambda \in (0,1]$, where 
	$\overline G$ is the closure of $G$.
\end{lemma}
This property can be verified by definition and Jensen’s inequality; see \cite{zhang2010b}. 

\begin{lemma}[Scale invariance]  \label{lem2}
	If ${\bf U} \in G$, for any $\lambda > 0$, it holds $\lambda {\bf U} \in G$. 
\end{lemma}
The proof is straightforward.
Combining Lemmas \ref{lem1} and \ref{lem2}, we immediately obtain the following stronger property. 
\begin{lemma}\label{lem3}
	For any $\lambda_1 >0$, $\lambda_0 \ge 0$, ${\bf U}_1 \in G$ and $ {\bf U}_0 \in \overline G$, we have 
	$\widehat {\bf U}:= \lambda_1 {\bf U}_1 + \lambda_0 {\bf U}_0 \in  G$.
\end{lemma}

\begin{proof}
	Let $\lambda : = \frac{\lambda_1}{\lambda_1 + \lambda_0}  \in (0,1] $. It follows from 
	Lemma \ref{lem1} that 
	$\lambda {\bf U}_1 + (1-\lambda) {\bf U}_0  \in G$. 
	Thus, we have $\widehat {\bf U} = (\lambda_1 + \lambda_0) (\lambda {\bf U}_1 + (1-\lambda) {\bf U}_0) \in G $, according to 
	Lemma \ref{lem2}. 
\end{proof}

\begin{lemma}\label{lem:UcontrolS}
	For any $\lambda \ge0$, $\delta \in \mathbb R$, ${\bf U}=(\rho,{\bf m},E)^\top \in G$, and ${\bf a} \in \mathbb R^d$, if $|\delta| \frac{ \| {\bf a} \| } { \sqrt{2 e} } \le \lambda$, then   
	$$
	\widehat {\bf U}:=\lambda {\bf U} + \delta \big(0,\rho {\bf a}, {\bf m} \cdot {\bf a} \big)^\top
	\in \overline G.
	$$
\end{lemma}

\begin{proof}
	If $\lambda =0$, it then follows from $|\delta| { \| {\bf a} \| } /{ \sqrt{2 e} } \le \lambda$ that $\delta =0$ or ${\bf a}={\bf 0}$, which implies 
	$ \widehat {\bf U} = {\bf 0} \in \overline {G} $. If $\lambda >0$, the first component of $\widehat {\bf U}$ equals $\lambda \rho >0$, and $\widehat{\bf U} = ( \lambda \rho, \lambda {\bf m}
	+ \delta \rho {\bf a}, \lambda E + \delta {\bf m} \cdot {\bf a} )^\top$ satisfies 
	\begin{align*}
	{\mathcal G}(\widehat{\bf U})  = \lambda E + \delta {\bf m} \cdot {\bf a} - 
	\frac{ \| \lambda {\bf m} 
		+ \delta \rho {\bf a} \|^2 }{ 2 \lambda \rho } 
	= \rho e \left( 1 + |\delta| \frac{ \| {\bf a} \| }{ \lambda \sqrt{ 2 e } }  \right) 
	\left( \lambda - |\delta| \frac{ \| {\bf a} \| }{ \sqrt{ 2 e } }  \right) \ge 0,
	\end{align*} 
	where the last inequality follows from the condition $|\delta| { \| {\bf a} \| }/ { \sqrt{2 e} } \le \lambda$. 
	Therefore, $\widehat {\bf U} \in \overline G$. 
\end{proof}

\begin{lemma}\label{lem:LFflux}
	For any ${\bf U}\in G$ and any unit vector ${\bf n} \in \mathbb R^d$, we have 
	$
	{\bf U} - \lambda {\bf F} ( {\bf U} ) \cdot {\bf n} \in G,
	$
	for any $\lambda \in \mathbb R$ satisfying 
	$ |\lambda| \alpha_{\bf n} ({\bf U}  ) \le 1$, where $\alpha_{\bf n} ({\bf U}  ) := |{\bf u}\cdot {\bf n}|+\sqrt{\gamma p/\rho}.$
\end{lemma}
The proof of Lemma \ref{lem:LFflux} can be found in, for example, \cite{zhang2010b,ZHANG2017301}.

%

\subsection{Properties of HLLC flux in one dimension}
In this subsection, we introduce several important properties of the HLLC numerical flux, 
whose properly modified version will be a key ingredient of our numerical schemes presented later. 
For notational convenience, we here focus on the properties of the HLLC flux in the
one-dimensional (1D) case ($d=1$), while the multidimensional extensions will be discussed in Section \ref{sec:HLLC2D}. 

In the 1D case, 
the HLLC flux (see, for example, \cite{BattenSISC-HLLCPP,toro2013riemann}) is defined by 
\begin{equation}\label{HLLC}
{\bf F}^{hllc} ({\bf U}_L,{\bf U}_R) = \begin{cases}
{\bf F}({\bf U}_L), \quad & \mbox{if~~} 0 \le S_L,
\\
{\bf F}_{*L}, \quad & \mbox{if~~} S_L \le 0 \le S_*,
\\
{\bf F}_{*R}, \quad & \mbox{if~~} S_* \le 0 \le S_R,
\\
{\bf F} ( {\bf U}_R ), \quad & \mbox{if~~} 0 \ge S_R,
\end{cases}	
\end{equation}
where $S_L$ and $S_R$ are the estimated (left and right)
fastest signal velocities arising from the solution of the Riemann problem, and 
the middle wave speed $S_*$ and fluxes are given by 
\begin{align*}
&S_* = \frac{ p_R-p_L +\rho_L u_L ( S_L - u_L ) - \rho_R u_R ( S_R - u_R ) }
{ \rho_L (S_L - u_L) -\rho_R ( S_R - u_R ) }, \qquad
{\bf F}_{*i} = {\bf F}_i+ S_i ( {\bf U}_{*i} - {\bf U}_i ), ~~~ i=L,\,R,
\end{align*} 
with the intermediate states given by
\begin{equation}\label{US}
{\bf U}_{*i} = \rho_i \left( \frac{S_i-u_i}{S_i-S_*} \right) 
\begin{pmatrix}
1
\\
S_*
\\
\frac{E_i}{\rho_i} + ( S_* - u_i ) \left( S_* + \frac{ p_i }{ \rho_i ( S_i - u_i ) }  \right)
\end{pmatrix}.
\end{equation}
With $\alpha_\pm = u \pm \sqrt{\gamma p/\rho} $, the following estimates of $S_L$ and $S_R$ are used in our computation.
\begin{equation}\label{SLSR}
S_L = \min \{ \alpha_- ( {\bf U}_L ) , \alpha_- ( {\bf U}_R )   \}, \qquad 
S_R = \max \{ \alpha_+ ( {\bf U}_L ) , \alpha_+ ( {\bf U}_R )   \}.
\end{equation}

The HLLC flux possesses two important properties, namely the contact property (see, e.g., \cite{CK2015}) and the positivity \cite{BattenSISC-HLLCPP}, as outlined below.

\begin{lemma}\label{lem:HLLCcontact}
	For any two states ${\bf U}_L =( \rho_L, 0, p/(\gamma-1) )^\top $ 
	and ${\bf U}_R =( \rho_R, 0, p/(\gamma-1) )^\top $, the HLLC flux \eqref{HLLC} satisfies 
	$$ 
	{\bf F}^{hllc} ({\bf U}_L,{\bf U}_R) = ( 0, p, 0 )^\top.
	$$
\end{lemma}
The proof is straightforward. The importance of this property for the well-balancedness was observed and used in \cite{CK2015}.

\begin{lemma}\label{lem:HLLCpp}
	For any two admissible states ${\bf U}_L \in G$ and ${\bf U}_R \in G$, the intermediate states defined in \eqref{US} satisfy
	$$
	{\bf U}_{*L} \in G, \quad  {\bf U}_{*R} \in G.
	$$ 
\end{lemma}
The proof of this property for the Euler equations can be found in \cite[Section 5.3]{BattenSISC-HLLCPP}. 
As a direct consequence of Lemma \ref{lem:HLLCpp}, we have the following conclusions, 
which are relevant to the positivity of 
the HLLC scheme for the 1D Euler equations without gravitation. 

\begin{lemma}\label{HLLC_std_euler:old00}
	For any two admissible states ${\bf U}_0 , {\bf U}_1 \in G$, 
	one has 
	\begin{align}\label{eq:www11}
	&{\bf U}_\lambda^{(1)} : =  {\bf U}_1 - \lambda \left(   {\bf F} ({\bf U}_1)  - {\bf F}^{hllc} ({\bf U}_0,{\bf U}_1) \right) \in G, 
	\\ \label{eq:www22}
	& {\bf U}_\lambda^{(0)} : =  {\bf U}_0 - \lambda \left(  
	{\bf F}^{hllc} ({\bf U}_0,{\bf U}_1) - 
	{\bf F} ({\bf U}_0)   \right) \in G, 
	\end{align}
	if $\lambda > 0$ and satisfies 
	\begin{equation}\label{HLLCcfl22}
	\lambda  \max_{ {\bf U} \in \{  {\bf U}_0 , {\bf U}_1 \}  } \alpha_{\max} ({\bf U}  ) \le 1,
	\end{equation}	
	where 
	$$
	\alpha_{\max} ({\bf U}  ) := |u|+\sqrt{\gamma p/\rho} = \max\{ |	\alpha_- ({\bf U}  ) |, |	\alpha_+ ({\bf U}  ) |  \}.
	$$
\end{lemma}
\begin{proof}
	Let $
	S_1 :=  S_L( {\bf U}_0 , {\bf U}_1 ),
	$ which satisfies $\lambda | S_1|\le 1$. 
	According to the definition of the HLLC flux, we derive that 
	\begin{align*}
	{\bf U}_\lambda^{(1)}   = \int_0^{ \lambda \max \{  S_1, 0 \} 
	} 	{\mathcal R} ( x/\lambda, {\bf U}_0, {\bf U}_1 ) {\rm d} x
	+ ( 1 - \lambda \max \{  S_1, 0 \}  ) {\bf U}_1, 
	\end{align*}
	where ${\mathcal R} (x/t, {\bf U}_L, {\bf U}_R ) $ denotes the approximate HLLC solution to  the Riemann problem between the states 
	${\bf U}_L$ and ${\bf U}_R$, i.e., 
	$$
	{\mathcal R} (x/t, {\bf U}_L, {\bf U}_R ) = \begin{cases}
	{\bf U}_L, \quad & \mbox{if~~} \frac{x}{t} \le S_L,
	\\
	{\bf U}_{*L}, \quad & \mbox{if~~} S_L \le \frac{x}{t} \le S_*,
	\\
	{\bf U}_{*R}, \quad & \mbox{if~~} S_* \le \frac{x}{t} \le S_R,
	\\
	{\bf U}_R , \quad & \mbox{if~~} \frac{x}{t} \ge S_R.
	\end{cases}	
	$$
	Thanks to Lemma \ref{lem:HLLCpp}, we have 
	${\mathcal R} (x/t, {\bf U}_0, {\bf U}_1 ) \in G,$ for all $x \in \mathbb R$ and $t>0$. 
	The convexity of $G$ leads to
	${\bf U}_\lambda^{(1)} \in G $ under the condition \eqref{HLLCcfl22}. Similar argument yields ${\bf U}_\lambda^{(0)} \in G $. 
\end{proof}

\begin{lemma}\label{HLLC_std_euler:old}
	For any three admissible states ${\bf U}_L , {\bf U}_M, {\bf U}_R \in G$, 
	one has 
	$$
	{\bf U}_\lambda : =  {\bf U}_M - \lambda \left(   {\bf F}^{hllc} ({\bf U}_M,{\bf U}_R)  - {\bf F}^{hllc} ({\bf U}_L,{\bf U}_M) \right) \in G, 
	$$
	if $\lambda > 0$ satisfies 
	\begin{equation}\label{HLLCcfl}
	\lambda  \max_{ {\bf U} \in \{  {\bf U}_L , {\bf U}_M, {\bf U}_R \}  } \alpha_{\max} ({\bf U}  ) \le \frac1{2}.
	\end{equation}
\end{lemma}

\begin{proof}
	Under the condition \eqref{HLLCcfl}, applying Lemma \ref{HLLC_std_euler:old00} leads to
	$$  {\bf U}_M - 2\lambda \left(   {\bf F} ({\bf U}_M )  - {\bf F}^{hllc} ({\bf U}_L,{\bf U}_M) \right) \in G, \qquad 
	{\bf U}_M - 2\lambda \left(  
	{\bf F}^{hllc} ({\bf U}_M,{\bf U}_R) - 
	{\bf F} ({\bf U}_M)   \right) \in G.
	$$
	Taking average of the above two terms and using the convexity of $G$ yield 
	${\bf U}_\lambda \in G$. 
\end{proof}

As generalization of Lemmas \ref{HLLC_std_euler:old00} and \ref{HLLC_std_euler:old}, the following results discuss the positivity of 
a properly modified HLLC flux, used in the construction of well-balanced methods in Section \ref{sec:method1D}. 

\begin{lemma} \label{HLLC_std_euler00}
	For any parameters $\zeta_1,~\zeta_2,~\zeta_3,~\zeta_4 \in \mathbb R^+$ and  
	any two admissible states ${\bf U}_0 , {\bf U}_1 \in G$, if $\lambda > 0$ and satisfies \eqref{HLLCcfl22}, we have 
	\begin{align}\label{eq:www11gg}
	& \zeta_2 {\bf U}_1 - \lambda \left(   {\bf F} (\zeta_2 {\bf U}_1)  - {\bf F}^{hllc} ( \zeta_1 {\bf U}_0,\zeta_2 {\bf U}_1) \right) \in G, 
	\\ \label{eq:www22gg}
	&  \zeta_3 {\bf U}_0 - \lambda \left(  
	{\bf F}^{hllc} (\zeta_3 {\bf U}_0, \zeta_4{\bf U}_1) - 
	{\bf F} ( \zeta_3 {\bf U}_0)   \right) \in G.
	\end{align}
\end{lemma}
%
\noindent
This follows from Lemmas \ref{HLLC_std_euler:old00} and \ref{lem2}, and noting 
$$ \max_{ {\bf U} \in \{ \zeta_1  {\bf U}_0 , \zeta_2 {\bf U}_1\}  } \alpha_{\max} ({\bf U}  ) 
= \max_{ {\bf U} \in \{  {\bf U}_0 , {\bf U}_1 \}  } \alpha_{\max} ({\bf U}  ).$$

\begin{lemma} \label{HLLC_std_euler}
	For any parameters $\zeta_1,~\zeta_2,~\zeta_3 \in \mathbb R^+$ and 
	any admissible states ${\bf U}_L , {\bf U}_M, {\bf U}_R \in G$, if $\lambda > 0$ satisfies \eqref{HLLCcfl}, we have 
	$$
	\zeta _2 {\bf U}_M - \lambda \left(  {\bf F}^{hllc} ( \zeta _2  {\bf U}_M , \zeta _3  {\bf U}_R)  - {\bf F}^{hllc} ( \zeta _1 {\bf U}_L, \zeta _2 {\bf U}_M)   \right) \in G.
	$$
\end{lemma}
%
\noindent
The proof directly follows from Lemma \ref{HLLC_std_euler:old} by noting that 
$\zeta_1 {\bf U}_L , \zeta_2 {\bf U}_M, \zeta_3{\bf U}_R \in G$ (due to Lemma \ref{lem2})  and that
$ \max_{ {\bf U} \in \{ \zeta_1  {\bf U}_L , \zeta_2 {\bf U}_M, \zeta_3 {\bf U}_R \}  } \alpha_{\max} ({\bf U}  ) 
= \max_{ {\bf U} \in \{  {\bf U}_L , {\bf U}_M, {\bf U}_R \}  } \alpha_{\max} ({\bf U}  ).$

\section{Positivity-preserving well-balanced DG methods in one dimension} \label{sec:method1D}
In one spatial dimension, the Euler equations \eqref{eq:dD} take the form of
\begin{equation}\label{eq:1D}
{\bf U}_t +  ( {\bf F} ( {\bf U} ) )_x = {\bf S} ( {\bf U}, x ),
\end{equation}
with
\begin{equation}\label{eq:1D1}
{\bf U} = \begin{pmatrix}
\rho
\\
m
\\
E
\end{pmatrix},
\quad  
{\bf F} ( {\bf U} )= \begin{pmatrix}
\rho u
\\
\rho u^2 + p 
\\
(E+p) u
\end{pmatrix},
\quad 
{\bf S}({\bf U},x) = \begin{pmatrix}
0
\\
-\rho  \phi_x 
\\
-  m  \phi_x
\end{pmatrix}.
\end{equation}

\subsection{Well-balanced DG discretization}\label{sec:1DWB} 
Assume that the spatial domain $\Omega$ is divided into cells $\{I_j=(x_{j-\frac12},x_{j+1/2})\}$, and the mesh size is denoted by $h_j = x_{j+1/2}-x_{j-1/2}$, with $h=\max_j \{h_j\}$. 
The center of each cell is $x_j = ( x_{j-1/2} + x_{j+1/2} )/2$. Denote the DG numerical solutions as ${\bf U}_h(x,t)$, and for each $t \in (0,T_f]$, each component of ${\bf U}_h$ belongs to the finite dimensional space of discontinuous piecewise polynomial functions, $\mathbb V_h^k$, defined by 
$$
\mathbb V_h^k = \left\{ u(x) \in L^2(\Omega):~ 
u(x) \big|_{I_j} \in \mathbb P^k(I_j),~ \forall  j  \right\},
$$
where $\mathbb P^k(I_j)$ denotes the space of polynomials of degree up to $k$ in cell $I_j$. 
Then the semi-discrete DG methods for \eqref{eq:1D} are given as follows: 
for any test function $v \in \mathbb V_h^k$, ${\bf U}_h$ is computed by 
\begin{equation}\label{eq:1DDG}
\int_{I_j} ({\bf U}_h)_t v {\rm d} x - \int_{I_j} {\bf F}({\bf U}_h) v_x {\rm d} x 
+ \widehat{\bf F}_{j+\frac12} v( x_{j+\frac12}^- ) - \widehat{\bf F}_{j-\frac12} v( x_{j-\frac12}^+ ) 
= \int_{I_j} {\bf S} v {\rm d} x, 
\end{equation}
where $\widehat{\bf F}_{j+1/2}$ denotes the numerical flux at $x_{j+1/2}$. The notations 
$x_{j+1/2}^-$ and $x_{j+1/2}^+$ indicate the associated limits at $x_{j+1/2}$ 
taken from the left and right sides, respectively, with ${\bf U}_{j+1/2}^\pm := {\bf U}_h ( x_{j+1/2}^\pm )$.
For notional convenience, the $t$ dependence of all quantities is suppressed hereafter. 

Now, we  
construct the well-balanced DG methods which preserve a general equilibrium state \eqref{eq:equi1}. Assume that the target stationary hydrostatic solutions to be preserved are explicitly known and are denoted by $\{\rho^e(x), p^e(x), u^e(x)=0\}$. 
This yields 
\begin{equation}\label{eq:1DSsol}
( p^e(x) )_x = - \rho^e(x) \phi_x,\qquad u^e(x)=0.
\end{equation}
Let $\rho^e_h(x)$ and $p^e_h(x)$ denote the projections of 
$
\rho^e(x) 
$ and $p^e(x)$ onto the space $\mathbb V_h^k$, respectively.

To render the DG methods \eqref{eq:1DDG} well-balanced, we consider the modified HLLC numerical flux
\begin{equation}\label{eq:1Dflux}
\widehat{\bf F}_{j+\frac12} = 
{\bf F}^{hllc} \left( \frac{ p^{e,\star}_{j+\frac12}  }{ p_h^e( x_{j+\frac12}^- ) }   {\bf U}_{j+\frac12}^-, \frac{ p^{e,\star}_{j+\frac12}  }{ p_h^e( x_{j+\frac12}^+ ) }   {\bf U}_{j+\frac12}^+ \right),
\end{equation}
where $p^{e,\star}_{j+\frac12}$ is a suitable approximation to the equilibrium pressure at $x_{j+\frac12}$.  Here we define it as
\begin{equation}\label{eq:opt1}
p^{e,\star}_{j+\frac12} =  \frac12 \left( p_h^e( x_{j+\frac12}^- ) 
+ p_h^e( x_{j+\frac12}^+ ) \right),
\end{equation}
and other choices of $p^{e,\star}_{j+\frac12}$, including $\min \big\{ p_h^e( x_{j+\frac12}^- ), p_h^e( x_{j+\frac12}^+ ) \big\}$ and $\max \big\{ p_h^e( x_{j+\frac12}^- ), p_h^e( x_{j+\frac12}^+ ) \big\}$, also work. 
This modification does not affect the accuracy, provided that $\rho^e(x)$ and $p^e(x)$ are smooth. The element integral $\int_{I_j} {\bf F}({\bf U}_h) v_x {\rm d} x $ in \eqref{eq:1DDG} is approximated by 
the standard quadrature rule
\begin{equation}\label{eq:elemflux}
\int_{I_j} {\bf F}({\bf U}_h) v_x {\rm d} x \approx h_j \sum_{\mu=1}^N \omega_\mu  {\bf F}\big({\bf U}_h( x_j^{(\mu)} ) \big) v_x ( x_j^{(\mu)} ),
\end{equation}
where $\{ x_j^{(\mu)}, \, \omega_\mu \}_{1\le \mu \le N}$ denote the $N$-point Gauss quadrature nodes and weights in $I_j$. 

Next we consider the discretization of the integrals of the source terms in \eqref{eq:1DDG} to achieve the well-balanced property. 
Let ${\bf S} =: ( 0, S^{[2]}, S^{[3]} )^\top$. 
Following the techniques in \cite{XS2013,LiXingWBDG2016,LiXingWBFV2016}, we reformulate and decompose the integral of the source term in the momentum equation as 
\begin{align}\nonumber
&	 \int_{I_j} S^{[2]} v {\rm d} x  = \int_{I_j}  -\rho  \phi_x v {\rm d} x = \int_{I_j} \frac{\rho }{\rho^e} p_x^e v {\rm d} x 
= \int_{I_j} \left( \frac{\rho }{\rho^e} - 
\frac{ \overline \rho_j }{ \overline \rho^e_j}
+ \frac{ \overline \rho_j }{ \overline \rho^e_j}
\right) p_x^e v {\rm d}x 
\\ \label{eq:Sreformulation}
& 
\quad  = \int_{I_j} \left( \frac{\rho }{\rho^e} - 
\frac{ \overline \rho_j }{ \overline \rho^e_j}
\right) p_x^e v {\rm d}x + \frac{ \overline \rho_j }{ \overline \rho^e_j} \left( 
p^e ( x_{j+\frac12}^- ) v ( x_{j+\frac12}^- )  - p^e ( x_{j-\frac12}^+ )  v ( x_{j-\frac12}^+ ) - \int_{I_j} p^e v_x {\rm d} x
\right) ,
\end{align}
where \eqref{eq:1DSsol} has been used in the second identity, and the notation $\overline{(\cdot)}_j$ denotes the cell average of the associated quantity over $I_j$.  
We then approximate it by 
\begin{align} \label{eq:1Ds2}
& \int_{I_j} S^{[2]} v {\rm d} x  \approx h_j \sum_{ \mu=1}^N \omega_\mu \left( \frac{\rho_h ( x_j^{(\mu)} ) }{\rho^e_h  ( x_j^{(\mu)} ) } - 
\frac{ \overline {(\rho_h)}_j }{ \overline {(\rho^e_h)}_j  }
\right) (p_h^e)_x (   x_j^{(\mu)}  ) v  (  x_j^{(\mu)}  ) 
\\ \nonumber
& \qquad  + \frac{ \overline {(\rho_h)}_j }{ \overline {(\rho^e_h)}_j  } \left( 
p^{e,\star}_{j+\frac12}  v ( x_{j+\frac12}^- )  -  p^{e,\star}_{j-\frac12} v ( x_{j-\frac12}^+ ) - h_j \sum_{\mu=1}^N \omega_\mu p^e_h  ( x_j^{(\mu)} ) v_x  ( x_j^{(\mu)} ) 
\right)
=: \big\langle S^{[2]}, v \big\rangle_j.
\end{align}
Similarly, we approximate the integral of the source term in the energy equation by 
\begin{align} \label{eq:1Ds3}
& \int_{I_j} S^{[3]} v {\rm d} x  \approx h_j \sum_{ \mu=1}^N \omega_\mu \left( \frac{ m_h ( x_j^{(\mu)} ) }{\rho^e_h  ( x_j^{(\mu)} ) } - 
\frac{ \overline {(m_h)}_j }{ \overline {(\rho^e_h)}_j  }
\right) (p_h^e)_x (   x_j^{(\mu)}  ) v  (  x_j^{(\mu)}  ) 
\\ \nonumber
& \qquad  + \frac{ \overline {(m_h)}_j }{ \overline {(\rho^e_h)}_j  } \left( 
p^{e,\star}_{j+\frac12}  v ( x_{j+\frac12}^- )  - p^{e,\star}_{j-\frac12}  v ( x_{j-\frac12}^+ ) - h_j \sum_{\mu=1}^N \omega_\mu p^e_h  ( x_j^{(\mu)} ) v_x  ( x_j^{(\mu)} ) 
\right) =: \big\langle S^{[3]}, v \big\rangle_j,
\end{align}
Combining these leads to the well-balanced DG methods of the form
\begin{equation}\label{eq:1DDGweak}
\begin{aligned}
\int_{I_j} ({\bf U}_h)_t v {\rm d} x & = 
h_j \sum_{\mu=1}^N \omega_\mu  {\bf F}\big({\bf U}_h( x_j^{(\mu)} ) \big) v_x ( x_j^{(\mu)} ) 
- \left( \widehat{\bf F}_{j+\frac12} v( x_{j+\frac12}^- ) - \widehat{\bf F}_{j-\frac12} v( x_{j-\frac12}^+ ) \right)
\\
& \quad 
+ \Big( 0, \big\langle S^{[2]}, v \big\rangle_j, \big\langle S^{[3]}, v \big\rangle_j  \Big)^\top, \qquad \forall v \in \mathbb V_h^k.
\end{aligned}
\end{equation}

\begin{remark} \label{rem:WBPPconsideration}
	We here choose the modified HLLC flux \eqref{eq:1Dflux}, instead of the modified LF fluxes as in \cite{LiXingWBDG2016}, due to the following two considerations. 
	First, the HLLC flux satisfies the contact property (Lemma \ref{lem:HLLCcontact}), 
	which provides a unified modification approach to make the HLLC flux well-balanced for an arbitrary hydrostatic equilibrium; 
	whereas the modifications of the LF flux \cite{LiXingWBDG2016} have to be done separately for different types of equilibria. 
	Secondly, we will show that our modified HLLC flux \eqref{eq:1Dflux} also meets the positivity-preserving requirements, 
	whereas the modification to the LF fluxes in the polytropic equilibrium case may lose the positivity-preserving property. 
	We can prove the positivity of the well-balanced DG methods with the modified LF fluxes, only when isothermal equilibria are considered (see the Appendix \ref{app:DGLF}). 
\end{remark}
\begin{remark} \label{rem:WBPPconsideration1}
	Here, we approximate the integral $\int_{I_j} S^{[3]} v dx$ in \eqref{eq:1Ds3} in a way consistent with the term $\int_{I_j} S^{[2]} v dx$, 
	while in \cite{LiXingWBDG2016} $\int_{I_j} S^{[3]} v dx$ was approximated by the standard quadrature rule.  
	For the well-balancedness only, either approach is fine, and the standard one is even simpler. 
	However, our analysis will indicate that it is important to use a 
	``consistent'' approach for the purpose to accommodate the theoretical positivity-preserving property at the same time.  
\end{remark}

\begin{theorem}\label{thm:1DWB}
	For the 1D Euler equations \eqref{eq:1D} with gravitation, the semi-discrete DG schemes \eqref{eq:1DDGweak} are well-balanced for a general known stationary hydrostatic solution \eqref{eq:1DSsol}. 
\end{theorem}

\begin{proof}
	At the equilibrium state \eqref{eq:1DSsol}, we have 
	$
	\rho_h  = \rho^e_h, ~ u_h=u_h^e=0,~ E_h = \frac{p^e_h}{\gamma - 1},
	$
	which leads to
	\begin{align*}
	&\frac{ p^{e,\star}_{j+\frac12}  }{ p_h^e( x_{j+\frac12}^\pm ) }   {\bf U}_{j+\frac12}^\pm
	= \left( \rho^e_h( x_{j+\frac12}^\pm ) \frac{ p^{e,\star}_{j+\frac12}  }{ p_h^e( x_{j+\frac12}^\pm ) },~0,~ \frac{ p^{e,\star}_{j+\frac12} }{\gamma - 1} \right)^\top.
	\end{align*}
	Thanks to the contact property (Lemma \ref{lem:HLLCcontact}), 
	the modified HLLC numerical flux \eqref{eq:1Dflux} reduces to 
	\begin{equation}\label{WKLproof1124}
	\widehat{\bf F}_{j+\frac12} = 
	\big( 0,~p^{e,\star}_{j+\frac12},~0 \big)^\top.
	\end{equation}
	It is easy to observe that the well-balanced property holds for the mass and energy equations of \eqref{eq:1DDGweak}, as the first and third components of both
	the flux and source term approximations become zero.
	For the momentum
	equation, thanks to ${\rho_h ( x_j^{(\mu)} ) }/{\rho^e_h  ( x_j^{(\mu)} ) } =
	{ \overline {(\rho_h)}_j }/{ \overline {(\rho^e_h)}_j  } = 1$, we have  
	$$
	\big\langle S^{[2]}, v \big\rangle_j = 
	p^{e,\star}_{j+\frac12}  v ( x_{j+\frac12}^- )  -  p^{e,\star}_{j-\frac12} v ( x_{j-\frac12}^+ ) - h_j \sum_{\mu=1}^N \omega_\mu p^e_h  ( x_j^{(\mu)} ) v_x  ( x_j^{(\mu)} ).
	$$
	Let $F_2$ denote the second component of ${\bf F}$. 
	Since $u_h=0$, the flux term $F_2({\bf U}_h( x_j^{(\mu)} ))$ reduces to $p^e_h  ( x_j^{(\mu)} )$. 
	This, together with \eqref{WKLproof1124}, imply 
	\begin{align*}
	& h_j \sum_{\mu=1}^N \omega_\mu  F_2 \big({\bf U}_h( x_j^{(\mu)} ) \big) v_x ( x_j^{(\mu)} ) 
	- \left( \widehat F_{2,j+\frac12} v( x_{j+\frac12}^- ) - \widehat F_{2,j-\frac12} v( x_{j-\frac12}^+ ) \right)
	\\
	& \quad = h_j \sum_{\mu=1}^N \omega_\mu p^e_h  ( x_j^{(\mu)} ) v_x  ( x_j^{(\mu)} ) - \left( p^{e,\star}_{j+\frac12}  v ( x_{j+\frac12}^- )  -  p^{e,\star}_{j-\frac12} v ( x_{j-\frac12}^+ ) \right),
	\end{align*}
	which is exactly equal to $-\langle S^{[2]}, v \rangle_j$. 
	Therefore, the flux and source term approximations balance each other, which
	leads to the well-balanced property of our methods \eqref{eq:1DDGweak}.
\end{proof}

The weak form \eqref{eq:1DDGweak}
can be rewritten in the ODE form as 
\begin{equation}\label{semi_discrete}
\frac{{\rm d} {\bf U}_h (t)} {{\rm d} t} = {\bf L} ( {\bf U}_h ),  
\end{equation}
after choosing a suitable basis of $\mathbb V_h^k$ and representing 
${\bf U}_h$ as a linear combination of the basis functions; see \cite{CockburnShu1989} for details. 
The semi-discrete DG schemes \eqref{semi_discrete} can be further discretized in time by some explicit strong-stability-preserving (SSP) Runge-Kutta (RK) methods \cite{GottliebShuTadmor2001}. 
For example, with $\Delta t$ being the time step size, the third-order accurate SSP RK method is given by 
\begin{equation}\label{SSP1}
\begin{aligned}
&{\bf U}^{(1)}_h = {\bf U}_h^n + \Delta t {\bf L} ( {\bf U}_h^n ),
\\
&{\bf U}^{(2)}_h = \frac34 {\bf U}_h^n + \frac14 \Big(  {\bf U}^{(1)}_h + \Delta t {\bf L} ( {\bf U}^{(1)}_h ) \Big),
\\
&{\bf U}_h^{n+1} = \frac13 {\bf U}_h^n + \frac23 \Big(  {\bf U}^{(2)}_h + \Delta t {\bf L} ( {\bf U}^{(2)}_h ) \Big).
\end{aligned}
\end{equation}

\subsection{Positivity of first-order well-balanced DG scheme}
In this and the next subsections, we shall analyze the positivity of the well-balanced DG schemes \eqref{eq:1DDGweak}. The well-balanced modification of the numerical flux and discretization of source terms lead to additional difficulties in the positivity-preserving analyses, which are more complicated than 
the analyses for the standard DG methods. 

Denote the cell average of ${\bf U}_h$ over $I_j$ by 
$$
\overline {\bf U}_j(t) = \frac{1}{h_j} \int_{ I_j } {\bf U}_h (x,t) {\rm d} x.
$$
Taking the test function $v=1$ in \eqref{eq:1DDGweak}, one can obtain the semi-discrete 
evolution equations satisfied by the cell average as 
\begin{equation}\label{eq:AVEevolve}
\frac{{\rm d} \overline {\bf U}_j (t)} { {\rm d} t} = {\bf L}_j ( {\bf U}_h ) := -\frac{1}{h_j} \left( 
\widehat{\bf F}_{j+\frac12}  - \widehat{\bf F}_{j-\frac12}
\right) + \overline{\bf S}_j,
\end{equation}
where $\overline{\bf S}_j= \big( 0, \overline S_j^{[2]}, \overline S_j^{[3]} \big)^\top$ with 
$\overline S_j^{[\ell]} := \frac1{h_j} \left\langle S^{[\ell]}, 1 \right\rangle_j,$ $\ell = 2,~ 3$. 

When the polynomial degree $k=0$, 
we have $ {\bf U}_h (x,t) \equiv \overline {\bf U}_j(t)$ for all $x \in I_j$, and
the above DG methods \eqref{eq:AVEevolve} reduce to the 
corresponding first-order scheme with 
\begin{equation}\label{eq:1Dflux1}
\widehat{\bf F}_{j+\frac12} = 
{\bf F}^{hllc} \left( \frac{ p^{e,\star}_{j+\frac12}  }{ \overline p_j^e } \overline  {\bf U}_j, ~\frac{ p^{e,\star}_{j+\frac12}  }{ \overline p_{j+1}^e }   \overline {\bf U}_{j+1} \right). 
\end{equation}
We start by showing the positivity property of the homogeneous case.

\begin{lemma}\label{lem:1DF}
	If the DG polynomial degree $k=0$ and $\overline {\bf U}_{j} \in G$ for all $j$, we have
	\begin{equation}\label{key1112}
	\overline {\bf U}_j -\frac{\Delta t}{h_j} \left( \widehat{\bf F}_{j+\frac12} - \widehat{\bf F}_{j-\frac12}   \right) \in G, \quad \forall j,
	\end{equation}
	under the CFL-type condition 
	\begin{equation}\label{eq:CFL1}
	\frac{\Delta t}{h_j} \left( \frac
	{ p^{e,\star}_{j+\frac12} + p^{e,\star}_{j-\frac12}  } { \overline p_j^e  }  \max_{ {\bf U} \in \{  \overline {\bf U}_{j-1} , \overline {\bf U}_j, \overline {\bf U}_{j+1} \}  } \alpha_{\max} ({\bf U}  ) \right)  \le \frac12. 
	\end{equation}
\end{lemma}

\begin{proof}
	Using \eqref{eq:1Dflux1}, we have 
	\begin{align*}
	& \overline {\bf U}_j -\frac{\Delta t}{h_j} \left( \widehat{\bf F}_{j+\frac12} - \widehat{\bf F}_{j-\frac12}   \right)  = 
	\overline {\bf U}_j - \frac{\Delta t}{h_j} 
	\left[
	{\bf F}^{hllc} \left( \frac{ p^{e,\star}_{j+\frac12}  }{ \overline p_{j}^e }  \overline {\bf U}_j,  \frac{ p^{e,\star}_{j+\frac12}  }{ \overline p_{j+1}^e }  \overline {\bf U}_{j+1} \right)
	- {\bf F}^{hllc} \left( \frac{ p^{e,\star}_{j-\frac12}  }{ \overline p_{j-1}^e }  \overline  {\bf U}_{j-1}, \frac{ p^{e,\star}_{j-\frac12}  }{ \overline p_{j}^e }  \overline {\bf U}_{j} \right)
	\right]. 
	\end{align*}		
	Note that, the well-balanced modification leads to 
	$$\frac{ p^{e,\star}_{j+\frac12}  }{ \overline p_{j}^e }  \overline {\bf U}_j \neq \frac{ p^{e,\star}_{j-\frac12}  }{ \overline p_{j}^e }  \overline {\bf U}_{j},$$
	so that the positivity of the standard HLLC scheme cannot be used directly. To address this issue, we make the following decomposition
	\begin{align*}
	\begin{split}	
	&\overline {\bf U}_j -\frac{\Delta t}{h_j} \left( \widehat{\bf F}_{j+\frac12} - \widehat{\bf F}_{j-\frac12}   \right)  
	=	\beta_{j}
	\left( {\bf W }_1 + {\bf W}_2   \right),
	\end{split} 
	\end{align*}
	with $\beta_{j}:= \frac{ \overline p_{j}^e  }
	{ p^{e,\star}_{j+\frac12} + p^{e,\star}_{j-\frac12}  } >0$, and 
	\begin{align*}
	& {\bf W}_1 = 		\frac{ p^{e,\star}_{j+\frac12}  }
	{ \overline p_{j}^e  }
	\overline {\bf U}_j -  \frac{\Delta t}{ \beta_{j}  h_j} 
	\left[
	{\bf F}^{hllc} \left( \frac{ p^{e,\star}_{j+\frac12}  }{ \overline p_{j}^e }  \overline  {\bf U}_j, \frac{ p^{e,\star}_{j+\frac12}  }{ \overline p_{j+1}^e }  \overline {\bf U}_{j+1} \right)
	- {\bf F}^{hllc} \left( \frac{ p^{e,\star}_{j-\frac12}  }{ \overline p_{j}^e }  \overline  {\bf U}_{j}, \frac{ p^{e,\star}_{j+\frac12}  }{ \overline p_{j}^e }  \overline  {\bf U}_{j} \right)
	\right],
	\\
	& {\bf W}_2 = \frac{ p^{e,\star}_{j-\frac12}  }
	{ \overline p_{j}^e   }
	\overline {\bf U}_j
	- \frac{\Delta t} { \beta_{j}   h_j} 
	\left[
	{\bf F}^{hllc} \left( \frac{ p^{e,\star}_{j-\frac12}  }{ \overline p_{j}^e }  \overline {\bf U}_{j}, \frac{ p^{e,\star}_{j+\frac12}  }{ \overline p_{j}^e }  \overline {\bf U}_{j} \right)
	- {\bf F}^{hllc} \left( \frac{ p^{e,\star}_{j-\frac12}  }{ \overline p_{j-1}^e }  \overline {\bf U}_{j-1}, \frac{ p^{e,\star}_{j-\frac12}  }{ \overline p_{j}^e  } \overline  {\bf U}_{j} \right)
	\right].
	\end{align*}
	Applying Lemma \ref{HLLC_std_euler} leads to ${\bf W}_1, {\bf W}_2 \in G$ under the condition \eqref{eq:CFL1}. We can conclude \eqref{key1112} by using 
	Lemma \ref{lem3}, and this completes the proof.
\end{proof}

%
%

For all $j$, we define $\overline e_j := \frac{1}{\overline \rho_j} \Big( \overline E_j - 
\frac{ \overline m_j^2 }{ 2  \overline \rho_j } \Big)$ and $\widehat \alpha_j := \widehat \alpha_j^{F}+ \widehat \alpha_j^{S}$ with 
\begin{align*}
&
\widehat \alpha_j^{F} :=  2\frac
{ p^{e,\star}_{j+\frac12} + p^{e,\star}_{j-\frac12}  } { \overline p_{j}^e  }  \max_{ {\bf U} \in \{    \overline {\bf U}_{j-1} , \overline {\bf U}_j, \overline {\bf U}_{j+1} \}  } \alpha_{\max} ({\bf U}  ),\qquad \widehat \alpha_j^{S}:=  \frac{ \left| p^{e,\star}_{j+\frac12} - p^{e,\star}_{j-\frac12} \right|  }{ \overline \rho_j^e \sqrt{2 \overline e_j }   }.
\end{align*}

\begin{theorem}\label{thm:1D1st}	
	If the DG polynomial degree $k=0$ and $\overline {\bf U}_{j} \in G$ for all $j$, we have
	\begin{equation}\label{eq:1Dorder1}
	\overline {\bf U}_j + \Delta t {\bf L}_j ( {\bf U}_h  ) \in G, \quad \forall j,	
	\end{equation}
	under the CFL-type condition 
	\begin{equation}\label{eq:CFL1F}	
	\widehat \alpha_j \Delta t \le h_j.
	\end{equation}	
\end{theorem}

\begin{proof}
	When $k=0$, one has $\overline \rho_j^e= \frac1{h_j}  \int_{I_j} \rho_h^e(x)dx>0 $, $\overline p_j^e= \frac 1 {h_j}  \int_{I_j} p_h^e(x)dx>0 $ and 
	\begin{equation}\label{proof1S}
	\overline {\bf S}_j =   \frac{ \left| p^{e,\star}_{j+\frac12} - p^{e,\star}_{j-\frac12} \right|  }{ h_j \overline \rho_j^e  } 
	\Big( 0,
	\overline \rho_j, 
	\overline m_j \Big)^\top.
	\end{equation}
	If $\big| p^{e,\star}_{j+\frac12} - p^{e,\star}_{j-\frac12} \big|=0$, we have $\overline {\bf S}_j={\bf 0}$, and 
	$\overline {\bf U}_j + \Delta t {\bf L}_j ( {\bf U}_h  ) = \overline {\bf U}_j -\frac{\Delta t}{h_j} \big( \widehat{\bf F}_{j+\frac12} - \widehat{\bf F}_{j-\frac12}   \big) \in G,
	$ according to Lemma \ref{lem:1DF}. Otherwise, decompose the scheme as 
	\begin{equation}\label{eq:proofwkl143}
	\overline {\bf U}_j + \Delta t {\bf L}_j ( {\bf U}_h  ) = \overline {\bf U}_j -\frac{\Delta t}{h_j} \left( \widehat{\bf F}_{j+\frac12} - \widehat{\bf F}_{j-\frac12}   \right) + \Delta t \overline {\bf S}_j
	= \frac{ \widehat \alpha_j^F }{\widehat \alpha_j} 
	{\bf W}_F+  \frac{ 1 }{\widehat \alpha_j}  {\bf W}_S,
	\end{equation}
	where \vspace*{-3mm}
	\begin{align*}
	&{\bf W}_F:=\overline {\bf U}_j - \frac{\Delta t \widehat \alpha_j }{ h_j \widehat \alpha_j^F } \left( \widehat{\bf F}_{j+\frac12} - \widehat{\bf F}_{j-\frac12}   \right), 
	\\
	&{\bf W}_S := \widehat \alpha_j^S \overline {\bf U}_j +  { \widehat \alpha_j \Delta t } \overline {\bf S}_j = 
	\widehat \alpha_j^S \overline {\bf U}_j + \Delta t { \widehat \alpha_j   } 
	\widehat \alpha_j^S \frac{\sqrt{2 \overline e_j}}{h_j} \Big( 0,
	\overline \rho_j, 
	\overline m_j \Big)^\top.
	\end{align*}
	The condition \eqref{eq:CFL1F} implies
	$
	\left| \Delta t { \widehat \alpha_j   } 
	\widehat \alpha_j^S \frac{\sqrt{2e_j}}{h_j} \right| \frac{1}{ \sqrt{2 \overline e_j} } \le \widehat \alpha_j^S, 
	$
	which leads to, based on Lemma \ref{lem:UcontrolS}, that 
	${\bf W}_S \in \overline G$. With the aid of Lemma \ref{lem:1DF}, we obtain 
	${\bf W}_F \in G$ under the condition \eqref{eq:CFL1F}. 
	Finally, the combination of \eqref{eq:proofwkl143} and Lemma \ref{lem3} yields 
	\eqref{eq:1Dorder1}. 
\end{proof}

Theorem \ref{thm:1D1st} indicates that the first-order ($k = 0$)
well-balanced DG method \eqref{eq:1DDGweak}, coupled with a forward Euler time discretization, is 
positivity-preserving under the CFL-type condition \eqref{eq:CFL1F}.


\subsection{Positivity-preserving high-order well-balanced DG schemes}

When the polynomial degree $k \ge 1$, the
high-order well-balanced DG schemes \eqref{eq:1DDGweak} are not positivity-preserving in general. Fortunately, a weak positivity property can be proven for the schemes \eqref{eq:1DDGweak}; see Theorem \ref{thm:1Dhigh}. 
As we will see, such weak positivity is crucial and implies that 
a simple limiter can enforce the positivity-preserving property without losing conservation and high-order accuracy.  

\subsubsection{Theoretical positivity-preserving analysis}

Let $\{ \widehat{x}_{j}^{(\nu)} \}_{1\le \nu \le L}$ be the Gauss-Lobatto nodes transformed into 
the interval $I_j$, and $\{\widehat \omega_\nu\}_{1\le \nu \le L}$ be the associated quadrature weights satisfying $\sum_{\nu=1}^L \widehat \omega_\nu = 1$ and $\widehat \omega_1 = \widehat \omega_L =\frac{1}{L(L-1)}$, 
with $L\ge (k+3)/2$ to ensure that the algebraic precision of the corresponding quadrature
rule is at least $k$. For each cell $I_j$, we define the point set
\begin{equation}\label{eq:1DpointSet}
\mathbb S_j := \{ \widehat x_{j}^{(\nu)} \}_{\nu=1}^L \cup \{ x_{j}^{(\mu)} \}_{\mu=1}^N,
\end{equation}
and define $\widetilde \alpha_j$ as
\begin{align}\nonumber   
& \widetilde \alpha_j := \widetilde \alpha_j^F + \widetilde \alpha_j^S + \overline \alpha_j^S,
\qquad \widetilde \alpha_j^F := 2~ {\rm max} \left\{ 
\frac{ p^{e,\star}_{j + \frac12}   } {   p_h^e( x_{j + \frac12}^-  )   } , 
\frac{ p^{e,\star}_{j - \frac12}   } {   p_h^e( x_{j - \frac12}^+  )   } \right\}
\max_{ {\bf U} \in \{  {\bf U}_{j-\frac12}^\pm , {\bf U}_{j+\frac12}^\pm \}  } \alpha_{\max} ({\bf U}  )  ,
\\ \label{eq:tildealphajs}
& \widetilde \alpha_j^S := \widehat \omega_1 h_j \max_{1\le \mu \le N} \left\{  \frac{ \left| \big(p_h^e\big)_x (   x_j^{(\mu)}  ) \right|
}{ \rho_h^e( x_j^{(\mu)} ) \sqrt{2 e_h( x_j^{(\mu)} ) } } \right\},
\qquad  \overline \alpha_j^S := \widehat \omega_1 \frac{  \left| \jump{p_h^e}_{j+\frac12 } + \jump{p_h^e}_{j-\frac12 }   \right|  }{ 2 \overline \rho_j^e \sqrt{2 \overline e_j } },
\end{align}
with $\jump{p_h^e}_{j+\frac12 }:=p_h^e( x_{j+\frac12}^+ ) -p_h^e( x_{j+\frac12}^- ) $, 
where $\widetilde \alpha_j^S + \overline \alpha_j^S = {\mathcal O}(h_j) $ 
and $ {\rm max} \bigg\{ 
\frac{ p^{e,\star}_{j + \frac12}   } {   p_h^e( x_{j + \frac12}^-  )   } , 
\frac{ p^{e,\star}_{j - \frac12}   } {   p_h^e( x_{j - \frac12}^+  )   } \bigg\} = 1 + {\mathcal O}(h^{k+1})$ for smooth $p^e(x)$. 
Then we have the following sufficient condition for the high-order scheme \eqref{semi_discrete} to be positivity-preserving. 

\begin{theorem}\label{thm:1Dhigh}
	Assume that the projected stationary hydrostatic solutions satisfy
	\begin{equation}\label{PPcondition:depe}
	\rho^e_h( x )>0, \quad p^e_h(x)>0, \quad     \forall x \in \mathbb S_j,~~\forall j,
	\end{equation}
	and the numerical solution ${\bf U}_h$ satisfies 
	\begin{equation}\label{PPcondition}
	{\bf U}_h ( x ) \in G,  \quad     \forall x \in \mathbb S_j,~~\forall j,
	\end{equation} 
	then we have the weak positivity property 
	\begin{equation}\label{1DPPhigh}
	\overline {\bf U}_j + \Delta t {\bf L}_j ( {\bf U}_h  ) \in G,~~\forall j,
	\end{equation}
	under the CFL-type condition 
	\begin{equation}\label{CFL2high}
	\widetilde \alpha_j \Delta t \le \widehat \omega_1 h_j. 
	\end{equation}
\end{theorem}

\begin{proof}
	The exactness of the $L$-point Gauss-Lobatto quadrature rule for polynomials of degree up to $k$ implies 
	$$
	\overline {\bf U}_j = \frac{1}{h_j} \int_{I_j} {\bf U}_h( x ) {\rm d} x = 
	\sum_{\nu=1}^L \widehat \omega_\nu {\bf U}_h( \widehat x_j^{(\nu)} ),
	$$
	with $\widehat x_j^{(1)}=x_{j-\frac12}$, $\widehat x_j^{(L)}=x_{j+\frac12}$ and $\widehat \omega_1 = \widehat \omega_L$. We consider, for an arbitrary parameter $\eta \in (0,1]$, the following decomposition 
	\begin{align} \nonumber
	&  \overline {\bf U}_j + \Delta t {\bf L}_j ( {\bf U}_h  ) = 
	\eta 
	\overline {\bf U}_j  -\frac{\Delta t}{h_j} \left( 
	\widehat{\bf F}_{j+\frac12}  - \widehat{\bf F}_{j-\frac12}
	\right)
	+ (1-\eta)  \overline {\bf U}_j  + \Delta t \overline {\bf S}_j
	\\  \nonumber
	& \quad = \eta
	\sum_{\nu=1}^L \widehat \omega_\nu {\bf U}_h( \widehat x_j^{(\nu)} ) 
	-\frac{\Delta t}{h_j} \left( 
	\widehat{\bf F}_{j+\frac12}  - \widehat{\bf F}_{j-\frac12}
	\right) + (1-\eta)  \overline {\bf U}_j  + \Delta t \overline {\bf S}_j
	\\ \nonumber
	& \quad =    
	\left[ \eta \sum_{\nu=2}^{L-1} \widehat \omega_\nu {\bf U}_h( \widehat x_j^{(\nu)} ) \right] +
	\left[ \eta \widehat \omega_1  \big( {\bf U}_{j-\frac12}^+ +  {\bf U}_{j+\frac12}^- \big) 
	-\frac{\Delta t}{h_j} \left( 
	\widehat{\bf F}_{j+\frac12}  - \widehat{\bf F}_{j-\frac12}
	\right) \right]
	+
	\left[
	(1-\eta) \overline {\bf U}_j  + \Delta t \overline {\bf S}_j
	\right]
	\\ \label{eq:wklproof3132}
	& \quad =: {\bf W}_1 + {\bf W}_2 + {\bf W}_3,
	\end{align}
	where ${\bf W}_1 \in G \cup \{ {\bf 0}\} \subset \overline G$ according to Lemma \ref{lem3}. 
	The parameter $\eta$ could be simply taken as $1/2$, but this will lead to a restrictive condition for $\Delta t$. 
	In the following we would like to determine a suitable parameter $\eta$ in $(0,1]$ such that ${\bf W}_2 \in G$ and ${\bf W}_3 \in \overline G$.

Let us first consider
${\bf W}_2$ and reformulate it as follows 
\begin{align}
\nonumber 
~~~ & {\bf W}_2  = \eta \widehat \omega_1   {\bf U}_{j-\frac12}^+ +  \eta \widehat \omega_1  {\bf U}_{j+\frac12}^- 
\\ \nonumber
& \quad 
- \frac{\Delta t}{h_j} 
\left[
{\bf F}^{hllc} \left( \frac{ p^{e,\star}_{j+\frac12}  }{ p_h^e( x_{j+\frac12}^- ) }  {\bf U}_{j+\frac12}^-, \frac{ p^{e,\star}_{j+\frac12}  }{ p_h^e( x_{j+\frac12}^+ ) }   {\bf U}_{j+\frac12}^+ \right)
- {\bf F}^{hllc} \left( \frac{ p^{e,\star}_{j-\frac12}  }{ p_h^e( x_{j-\frac12}^- ) }  {\bf U}_{j-\frac12}^-, \frac{ p^{e,\star}_{j-\frac12}  }{ p_h^e( x_{j-\frac12}^+ ) }   {\bf U}_{j-\frac12}^+ \right)
\right]
\\ \nonumber
&  = 
\eta \widehat \omega_1   {\bf U}_{j+\frac12}^-
- \frac{\Delta t}{h_j} 
\left[
{\bf F}^{hllc} \left( \frac{ p^{e,\star}_{j+\frac12}  }{ p_h^e( x_{j+\frac12}^- ) }  {\bf U}_{j+\frac12}^-, \frac{ p^{e,\star}_{j+\frac12}  }{ p_h^e( x_{j+\frac12}^+ ) }   {\bf U}_{j+\frac12}^+ \right)
- {\bf F}^{hllc} \left( \frac{ p^{e,\star}_{j-\frac12}  }{ p_h^e( x_{j-\frac12}^+ ) }   {\bf U}_{j-\frac12}^+ , \frac{ p^{e,\star}_{j+\frac12}  }{ p_h^e( x_{j+\frac12}^- ) }  {\bf U}_{j+\frac12}^- \right)
\right]
\\ \nonumber
& + 
\eta \widehat \omega_1 
{\bf U}_{j-\frac12}^+
- \frac{\Delta t}{h_j} 
\left[
{\bf F}^{hllc} \left( \frac{ p^{e,\star}_{j-\frac12}  }{ p_h^e( x_{j-\frac12}^+ ) }   {\bf U}_{j-\frac12}^+ , \frac{ p^{e,\star}_{j+\frac12}  }{ p_h^e( x_{j+\frac12}^- ) }  {\bf U}_{j+\frac12}^- \right)
- {\bf F}^{hllc} \left( \frac{ p^{e,\star}_{j-\frac12}  }{ p_h^e( x_{j-\frac12}^- ) }  {\bf U}_{j-\frac12}^-, \frac{ p^{e,\star}_{j-\frac12}  }{ p_h^e( x_{j-\frac12}^+ ) }   {\bf U}_{j-\frac12}^+ \right)
\right]
\\ \label{eq:wklproof3321}
&  =: \eta \widehat \omega_1 \frac{ p_h^e( x_{j+\frac12}^- ) } { p^{e,\star}_{j+\frac12}  }
{\bf W}_2^+ + \eta \widehat \omega_1 \frac{ p_h^e( x_{j-\frac12}^+ ) } { p^{e,\star}_{j-\frac12}  } {\bf W}_2^-,
\end{align}
where  
\begin{align*}
\begin{split}
{\bf W}_2^+ &= \frac{ p^{e,\star}_{j+\frac12}  }{ p_h^e( x_{j+\frac12}^- ) } {\bf U}_{j+\frac12}^-
- \frac{\Delta t  { p^{e,\star}_{j+\frac12}  }}{ \eta \widehat \omega_1 h_j { p_h^e( x_{j+\frac12}^- ) }    } 
\\
& \quad \times 
\left[
{\bf F}^{hllc} \left( \frac{ p^{e,\star}_{j+\frac12}  }{ p_h^e( x_{j+\frac12}^- ) }  {\bf U}_{j+\frac12}^-, \frac{ p^{e,\star}_{j+\frac12}  }{ p_h^e( x_{j+\frac12}^+ ) }   {\bf U}_{j+\frac12}^+ \right)
- {\bf F}^{hllc} \left( \frac{ p^{e,\star}_{j-\frac12}  }{ p_h^e( x_{j-\frac12}^+ ) }   {\bf U}_{j-\frac12}^+ , \frac{ p^{e,\star}_{j+\frac12}  }{ p_h^e( x_{j+\frac12}^- ) }  {\bf U}_{j+\frac12}^- \right)
\right],
\end{split}
\\
\begin{split}
{\bf W}_2^- & = \frac{ p^{e,\star}_{j-\frac12}  }{ p_h^e( x_{j-\frac12}^+ ) }  {\bf U}_{j-\frac12}^+
- \frac{\Delta t  { p^{e,\star}_{j-\frac12}  }}{ \eta \widehat \omega_1 h_j { p_h^e( x_{j-\frac12}^+ ) }  } 
\\
& \quad \times 
\left[
{\bf F}^{hllc} \left( \frac{ p^{e,\star}_{j-\frac12}  }{ p_h^e( x_{j-\frac12}^+ ) }   {\bf U}_{j-\frac12}^+ , \frac{ p^{e,\star}_{j+\frac12}  }{ p_h^e( x_{j+\frac12}^- ) }  {\bf U}_{j+\frac12}^- \right)
- {\bf F}^{hllc} \left( \frac{ p^{e,\star}_{j-\frac12}  }{ p_h^e( x_{j-\frac12}^- ) }  {\bf U}_{j-\frac12}^-, \frac{ p^{e,\star}_{j-\frac12}  }{ p_h^e( x_{j-\frac12}^+ ) }   {\bf U}_{j-\frac12}^+ \right)
\right].
\end{split}
\end{align*}
Thanks to Lemma \ref{HLLC_std_euler}, we have ${\bf W}_2^\pm  \in G$, if 
$$
\frac{  \Delta t p^{e,\star}_{j\pm \frac12}   } { \eta \widehat w_1 h_j   p_h^e( x_{j\pm \frac12}^\mp  )   }   \max_{ {\bf U} \in \{  {\bf U}_{j-\frac12}^-, {\bf U}_{j-\frac12}^+ , {\bf U}_{j+\frac12}^-, {\bf U}_{j+\frac12}^+ \}  } \alpha_{\max} ({\bf U}  ) \le \frac12,
$$
or equivalently
\begin{equation}\label{eq:PP1}
\Delta t  \widetilde \alpha_j^F \le \eta \widehat w_1 h_j.
\end{equation}
By applying Lemma \ref{lem3} on \eqref{eq:wklproof3321}, we obtain ${\bf W}_2\in G$ under the condition \eqref{eq:PP1}. 

Next, the term $ {\bf W}_3$ is analyzed. 
Note that, for an arbitrary parameter $\lambda \in [0,1]$,  we have
\begin{align*}
(1-\eta) \overline {m}_j + \Delta t \overline {S}_j^{[2]} & = (1-\eta)
\overline {m}_j + \Delta t \sum_{\mu=1}^N \omega_\mu \left( \frac{ \rho_h( x_j^{(\mu)} ) }{ \rho_h^e( x_j^{(\mu)} ) } 
- \frac{ \overline{\rho}_j  }{ \overline\rho^e_j } \right) (p_h^e)_x (   x_j^{(\mu)}  ) 
+ \frac{\Delta t}{h_j} \frac{ \overline{\rho}_j  }{ \overline\rho^e_j } 
\left( p^{e,\star}_{j+\frac12} - p^{e,\star}_{j-\frac12}  
\right)
\\
\begin{split}
& = (1-\eta) \left[ (1- \lambda ) \overline {m}_j + \lambda \sum_{\mu=1}^N \omega_\mu m_h( x_j^{(\mu)} ) \right] + \Delta t \sum_{\mu=1}^N \omega_\mu \frac{ \rho_h( x_j^{(\mu)} ) }{ \rho_h^e( x_j^{(\mu)} ) } (p_h^e)_x (   x_j^{(\mu)}  )
\\
& \quad + \frac{\Delta t}{h_j} \frac{ \overline{\rho}_j  }{ \overline\rho^e_j } 
\left( p^{e,\star}_{j+\frac12} - p^{e,\star}_{j-\frac12}  - \int_{I_j}
(p_h^e)_x {\rm d} x
\right)
\end{split}
\\
\begin{split}
&=  (1-\eta) \lambda \sum_{\mu=1}^N \omega_\mu m_h( x_j^{(\mu)} ) + \Delta t \sum_{\mu=1}^N \omega_\mu \frac{ \rho_h( x_j^{(\mu)} ) }{ \rho_h^e( x_j^{(\mu)} ) } (p_h^e)_x (   x_j^{(\mu)}  )
\\
& \quad + (1-\eta)(1 -\lambda) \overline {m}_j + \frac{\Delta t}{h_j} \frac{ \overline{\rho}_j  }{ \overline\rho^e_j } \frac12 \left(
\jump{p_h^e}_{j+\frac12 } + \jump{p_h^e}_{j-\frac12 }  
\right),
\end{split}
\end{align*}
and similarly, 
\begin{align*}
(1-\eta) \overline {E}_j + \Delta t \overline {S}_j^{[3]} & = 
(1-\eta) \lambda \sum_{\mu=1}^N \omega_\mu E_h( x_j^{(\mu)} ) + \Delta t \sum_{\mu=1}^N \omega_\mu \frac{ m_h( x_j^{(\mu)} ) }{ \rho_h^e( x_j^{(\mu)} ) } (p_h^e)_x (   x_j^{(\mu)}  )
\\
& \quad + (1-\eta)(1 -\lambda) \overline {E}_j + \frac{\Delta t}{h_j} \frac{ \overline{m}_j  }{ \overline\rho^e_j } \frac12 \left(
\jump{p_h^e}_{j+\frac12 } + \jump{p_h^e}_{j-\frac12 }  
\right).
\end{align*}
Therefore, we have
\begin{align}\label{eq:PP22}
&{\bf W}_3  = 
\sum_{\mu=1}^N \omega_\mu {\bf W}_3^{(\mu)} 
+ \overline {\bf W}_3,
\\
\label{eq:W3mu}
&{\bf W}_3^{(\mu)}   := 
(1-\eta) \lambda {\bf U}_h( x_j^{(\mu)} ) +  \Delta t \frac{ (p_h^e)_x (   x_j^{(\mu)}  )
}{ \rho_h^e( x_j^{(\mu)} ) } 
\Big(0,~\rho_h( x_j^{(\mu)} ),~m_h( x_j^{(\mu)} ) \Big)^\top, 
\\ \label{eq:barW3}
&  \overline {\bf W}_3 := (1-\eta) (1 -\lambda) \overline {\bf U}_j + \Delta t \frac{  \jump{p_h^e}_{j+\frac12 } + \jump{p_h^e}_{j-\frac12 }    }{ 2 h_j \overline\rho^e_j } 
\Big(0,~\overline{\rho}_j,~\overline{m}_j \Big)^\top.
\end{align}
Thanks to Lemma \ref{lem:UcontrolS}, we have $\overline {\bf W}_3 \in G$ and  $ {\bf W}_3^{(\mu)}   \in \overline G$ for all 
$\mu$, if
$$
\Delta t \max_{1\le \mu \le N} \left\{  \frac{ \left| \big(p_h^e\big)_x (   x_j^{(\mu)}  ) \right|
}{ \rho_h^e( x_j^{(\mu)} ) \sqrt{2 e_h( x_j^{(\mu)} ) } } \right\} \le  (1-\eta) \lambda,
\qquad   \Delta t  \frac{  \left| \jump{p_h^e}_{j+\frac12 } + \jump{p_h^e}_{j-\frac12 }   \right|  }{ 2 h_j  \overline\rho^e_j \sqrt{2 \overline e_j } } \le (1-\eta) (1 -\lambda),
$$
or equivalently
\begin{equation}\label{eq:PP66}
\Delta t \widetilde \alpha_j^S \le \widehat \omega_1 h_j (1-\eta) \lambda, \qquad  
\Delta t \overline \alpha_j^S \le \widehat \omega_1 h_j (1-\eta) (1 -\lambda) .
\end{equation}
By applying Lemma \ref{lem3} on \eqref{eq:PP22}, we obtain ${\bf W}_3 \in \overline G$ under the condition \eqref{eq:PP66}. 

Combining these results, we conclude that if $\Delta t$ satisfies 
\begin{equation}\label{eq:Cond11}
\Delta t \in  \Omega_{\eta,\lambda}^{(j)} := 
\Big \{ \tau \in \mathbb R^+:~
\tau \widetilde \alpha_j^F \le \eta \widehat w_1 h_j, ~ 
\tau  \widetilde \alpha_j^S \le \widehat \omega_1 h_j (1-\eta) \lambda, ~
\tau \overline \alpha_j^S \le \widehat \omega_1 h_j (1-\eta) (1 -\lambda)  \Big \}, 
\end{equation}
then 
$${\bf W}_1 \in \overline G,~ {\bf W}_2 \in G,~ {\bf W}_3 \in \overline G,$$ 
which implies \eqref{1DPPhigh}, i.e., $\overline {\bf U}_j + \Delta t {\bf L}_j ( {\bf U}_h  ) = \sum_{i=1}^3 {\bf W}_i \in G$,
following Lemma \ref{lem3}. 
Since the two parameters $\eta$ and $\lambda$ can be chosen arbitrarily in this proof, 
we would like to specify the ``best'' $\eta$ and $\lambda$ that maximize $\sup \Omega_{\eta,\lambda}^{(j)}=: g(\eta,\lambda)$. 
Solving such an optimization problem gives 
$$
\max_{ \eta \in (0,1],\lambda \in [0,1] } g(\eta,\lambda) = g(\eta_*,\lambda_*) = 
\frac{\widehat \omega_1 h_j}{  \widetilde \alpha_j^F + \widetilde \alpha_j^S + \overline \alpha_j^S} = \frac{\widehat \omega_1 h_j}{ \widetilde \alpha_j },
$$
which is reached at  
$\eta_* = { \widetilde \alpha_j^F }/{ \widetilde  \alpha_j },$
$\lambda_* = \frac{ \widetilde \alpha^S }{ \widetilde \alpha^S + \overline \alpha^S }. $
Therefore the condition \eqref{eq:Cond11} reduces to 
$$\Delta t \le g(\eta_*,\lambda_*),$$
which is equivalent to \eqref{CFL2high}. This finishes the proof. 
\end{proof}

Theorem \ref{thm:1Dhigh} gives a sufficient condition for the proposed high-order well-balanced
DG schemes \eqref{semi_discrete} to ensure that the cell-averages $\overline {\bf U}_j$ in $G$, when combined with the forward Euler time
discretization. Since any high-order SSP-RK time discretization can be written as a convex combination of
the forward Euler method, the same conclusion also holds when SSP-RK time discretization is used. 

\begin{remark}\label{rem:cellave}
The well-balanced source term reformulation \eqref{eq:Sreformulation} involves the cell average $\{\overline \rho_j, \overline \rho^e_j\}$, 
instead of the midpoint values $\{\rho (x_j),  \rho^e (x_j) \}$ used in \cite{LiXingWBDG2016}, which also works for the purpose of well-balanced property.
However, in the latter case, the vector $\overline {\bf W}_3$ in  \eqref{eq:barW3} would become 
$$
\overline {\bf W}_3 := (1-\eta) (1 -\lambda) \overline {\bf U}_j + \Delta t \frac{  \jump{p_h^e}_{j+\frac12 } + \jump{p_h^e}_{j-\frac12 }    }{ 2 h_j \rho^e_h(x_j) } 
\left(0,
{\rho}_h( x_j ),
{m}_h(x_j) \right)^\top,
$$
and more restrictive condition on $\Delta t$ is required to ensure $\overline {\bf W}_3 \in 
\overline G$, because in general
${\rho}_h( x_j )$ and ${m}_h(x_j)$ are not necessarily components of $\overline {\bf U}_j$. 
\end{remark}

\subsubsection{Positivity-preserving limiter}
A simple positivity-preserving limiter (cf.~\cite{zhang2010b,wang2012robust}) can be applied to enforce the 
condition \eqref{PPcondition}. Denote 
\begin{align*}
\overline{\mathbb G}_h^k:= \left\{ {\bf u} \in [{\mathbb V}_h^k]^3:~~\frac{1}{h_j} \int_{I_j} {\bf u} (x) {\rm d}x \in G,~ \forall j \right\},
\quad {\mathbb G}_h^k:= \left\{ {\bf u} \in [{\mathbb V}_h^k]^3:~~{\bf u}\big|_{I_j} (x) \in G,~\forall x \in \mathbb S_j,\forall j \right\},
\end{align*}
where $\mathbb S_j$ is defined in \eqref{eq:1DpointSet}. 
For any $ {\bf U}_h \in  \overline{\mathbb G}_h^k$ with ${\bf U}_h \big|_{I_j} =: {\bf U}_j (x) $, we define 
the positivity-preserving limiting operator ${\bf \Pi}_h: \overline{\mathbb G}_h^k \to {\mathbb G}_h^k$ as 
\begin{equation}\label{eq:limiter}
{\mathbf \Pi}_h {\bf U}_h 
\big|_{I_j} = \theta_j^{(2)} ( \widehat {\bf U}_j(x) -\overline { \bf  U }_j ) + \overline { \bf  U }_j, \quad  \forall j,
\end{equation}
with 
$\theta_j^{(2)} = \min\bigg\{1, \frac{ {\mathcal G} ( \overline {\bf U}_j ) - \epsilon_2 }{ {\mathcal G} ( \overline {\bf U}_j ) - \min_{ x \in {\mathbb S}_j } {\mathcal G} \big( \widehat {\bf U}_j (x) \big) }\bigg\}$, ${\mathcal G}({\bf U})$ defined in \eqref{eq:DefG}, 
$\widehat{\bf U}_j( x ) := ( \widehat \rho_j( x ), {\bf m}_j(x), E_j(x)  )^\top$, and 
\begin{equation}\label{eq:limitertheta}
\widehat \rho_j(x) = \theta_j^{(1)} ( \rho_j( x) - \overline \rho_j ) + \overline  \rho_j,\quad \theta_j^{(1)} = \min\left\{1, \frac{\overline \rho_j - \epsilon_1}{ \bar \rho_j + \min_{ x \in {\mathbb S}_j} \rho_j(x ) }\right\}.
\end{equation}
Here $\epsilon_1$ and $\epsilon_2$ are two sufficiently small positive numbers, introduced to avoid the effect of the round-error.
In the computation, one can take $\epsilon_1= \min\{10^{-13},\overline \rho_j \}$ and $\epsilon_2=\min\{ 10^{-13},{\mathcal G}(\overline{\bf U}_j)\}$. 
Note that the positivity-preserving limiter keeps the mass conservation
$ \int_{I_j}  {\mathbf \Pi}_h  ( {\bf u} ) {\rm d} x =   \int_{I_j}  {\bf u} {\rm d} x,~\forall {\bf u} \in \overline{\mathbb G}_h^k $ 
and does not destroy the high-order accuracy; see \cite{zhang2010,zhang2010b,ZHANG2017301} for details. 

Define the initial numerical solutions as ${\bf U}_h^0 (x) := {\mathbf \Pi}_h {\bf P}_h {\bf U} (x,0) $. 
For the well-balanced DG schemes \eqref{semi_discrete} coupled with an SSP-RK method, if the positivity-preserving limiter \eqref{eq:limiter} is used at each RK stage, 
the resulting fully discrete DG methods are positivity-preserving, namely the numerical solutions ${\bf U}^n_h$ always satisfy \eqref{PPcondition}, 
i.e., ${\bf U}^n_h\in{\mathbb G}_h^k$. 	
For example, when the third-order method \eqref{SSP1} is adopted, the proposed high-order positivity-preserving well-balanced DG schemes of the form
\begin{equation}\label{SSP2}
\begin{aligned}
&{\bf U}^{(1)}_h = {\mathbf \Pi}_h \left[ {\bf U}_h^n + \Delta t {\bf L} (   {\bf U}_h^n )\right],
\\
&{\bf U}^{(2)}_h = {\mathbf \Pi}_h \left[ \frac34  {\bf U}_h^n + \frac14 \Big(  {\bf U}^{(1)}_h + \Delta t {\bf L} (  {\bf U}^{(1)}_h ) \Big) \right],
\\
&{\bf U}_h^{n+1} = {\mathbf \Pi}_h \left[ \frac13  {\bf U}_h^n + \frac23 \Big(  {\bf U}^{(2)}_h + \Delta t {\bf L} (  {\bf U}^{(2)}_h ) \Big) \right],
\end{aligned}
\end{equation}
are positivity-preserving under the CFL-type condition \eqref{CFL2high}. 

\begin{remark}
If the projected stationary hydrostatic solutions $\rho^e_h$ and $p^e_h$ do not satisfy the condition \eqref{PPcondition:depe} in Theorem \ref{thm:1Dhigh}, 
we can redefine $\rho^e_h,p^e_h \in \mathbb V_h^k$ as 
\begin{equation}\label{eq:hydrodstaticPP}
\left( \rho^e_h(x), 0, \frac{ p^e_h (x) }{\gamma -1} \right)^\top := {\mathbf \Pi}_h {\bf P}_h  \left( \rho^e(x), 0, \frac{ p^e(x) }{\gamma -1} \right)^\top,
\end{equation}
where ${\mathbf P}_h$ denotes the $L^2$--projection onto the space $[\mathbb V_h^k]^3$. 
One can verify that $\rho^e_h$ and $p^e_h$ defined by \eqref{eq:hydrodstaticPP} always satisfy \eqref{PPcondition:depe}. In practice, 
if the exact stationary hydrostatic solutions $\rho^e$ and $p^e$ do not involve low density or low pressure, the operator ${\mathbf \Pi}_h$ in 
\eqref{eq:hydrodstaticPP} would not be turned on. We remark that the positivity-preserving DG schemes also retain the well-balanced property, if 
\eqref{eq:hydrodstaticPP} is used.
\end{remark}

\begin{remark}\label{rem:CFL}
Note that the CFL constraint \eqref{CFL2high} is sufficient rather than necessary for preserving positivity. Also, for a Runge-Kutta
time discretization, to enforce the CFL condition rigorously, we need to obtain an accurate
estimation of $\widetilde \alpha_j$ for all the stages of Runge-Kutta based only on the numerical solution at 
time level $n$, which is very difficult in most of test examples. An efficient implementation (cf.~\cite{xing2013positivity})
may be, if a preliminary calculation to the next time step produces negative density or pressure, we restart
the computation from the time step $n$ with half of the time step size. 
Our numerical tests demonstrate that the proposed methods always work robustly with a CFL number slightly smaller than 
$\widehat \omega_1$ and the restart is yet never encountered. 
\end{remark}

\section{Positivity-preserving well-balanced DG methods in multiple  dimensions}\label{sec:method2D}
In this section, we extend the proposed 1D positivity-preserving well-balanced DG methods to the multidimensional cases. 
For the sake of clarity, we shall focus on the two-dimensional (2D) case with $d=2$ in the remainder of this section, and the extension of our numerical methods and 
analyses to the three-dimensional case ($d=3$) follows similar lines.
%

\subsection{Well-balanced DG discretization}
Assume that the 2D spatial domain $\Omega$ is partitioned into a mesh ${\mathcal T}_h$, which may be unstructured 
and consist of 
polygonal cells. Throughout this section, the lower-case $k$ is used to denote the DG polynomial degree, while the 
capital $K$ always represents a cell in ${\mathcal T}_h$.
Denote the DG numerical solutions as ${\bf U}_h({\bf x},t)$, and for any $t \in (0,T_f]$, each component of ${\bf U}_h$ belongs to the finite dimensional space of discontinuous piecewise polynomial functions, $\mathbb V_h^k$, defined by 
$$
\mathbb V_h^k = \left\{ u({\bf x}) \in L^2(\Omega):~ 
u({\bf x}) \big|_{K} \in \mathbb P^k(K),~ \forall  K \in {\mathcal T}_h  \right\},
$$
where ${\mathbb P}^{k} (K)$ is the space of polynomials of total degree up to $k$ in cell $K$. 
The semi-discrete DG methods for \eqref{eq:dD} are given as follows: 
for any test function $v \in \mathbb V_h^k$, ${\bf U}_h$ is computed by 
\begin{equation}\label{eq:2DDG}
\int_{K} ( {\bf U}_h )_t   v    {\rm d} {\bf x} - \int_{K}    
{\bf F}  ( {\bf U}_h ) \cdot {\bm \nabla} v  {\rm d} {\bf x}  
+ \sum_{ {\mathscr E} \in \partial K }  \int_{ {\mathscr E}  }  
\widehat{\bf F}_{{\bf n}_{{\mathscr E},K}} v^{{\tt int}(K)} {\rm d}s
= \int_{K}    
{\bf S}    v  {\rm d} {\bf x}, \qquad \forall v \in \mathbb V_h^k,
\end{equation}
where $\partial K$ denotes the boundary of the cell $K$, $\widehat{\bf F}_{{\bf n}_{{\mathscr E},K}}$ denotes the numerical flux on 
edge ${\mathscr E}$, 
${\bf n}_{{\mathscr E},K}$ is the outward unit normal to the edge ${\mathscr E}$ of $K$, and 
the superscripts ``${\tt int}(K)$'' or ``${\tt ext}(K)$'' indicate that the associated limit of $v (\bf x)$ at the cell interfaces is taken from the interior or the exterior of $K$. 

Assume that the target stationary hydrostatic solutions to be preserved are explicitly known and are denoted by $\{\rho^e({\bf x}), p^e({\bf x}), u^e({\bf x})=0\}$. 
Let $\rho^e_h({\bf x})$ and $p^e_h({\bf x})$ be the projections of 
$
\rho^e({\bf x}) 
$ and $p^e({\bf x})$ onto the space $\mathbb V_h^k$, respectively. The design of our multidimensional well-balanced DG methods is similar to the 1D case. 
More specifically, it is based on the well-balanced numerical flux and source term approximation given as follows. 

\subsubsection{The modified HLLC numerical fluxes} \label{sec:HLLC2D}
For any unit vector ${\bf n}\in \mathbb R^d$, 
let ${\bf F}^{hllc}\left( {\bf U}_L, {\bf U}_R ; {\bf n}
\right)$ denote the standard 2D HLLC numerical flux in the
direction ${\bf n}$ for the Euler equations. 
Details of the standard multidimensional HLLC flux can be found in \cite{BattenSISC-HLLCPP}. 
Analogous to the 1D HLLC flux, 
the 2D HLLC flux satisfies the following properties, whose proofs are similar to the 1D case and omitted.

\begin{lemma}\label{lem:HLLCcontact2D}
For any two states ${\bf U}_L =( \rho_L, 0, 0, p/(\gamma-1) )^\top $ 
and ${\bf U}_R =( \rho_R, 0, 0, p/(\gamma-1) )^\top $, the 2D HLLC flux  satisfies 
$$ 
{\bf F}^{hllc}\left( {\bf U}_L, {\bf U}_R ; {\bf n}
\right) = ( 0, p {\bf n}^\top, 0 )^\top.
$$
\end{lemma}

\begin{lemma} \label{HLLC_std_euler2D00}
For any parameters $\zeta_1,~\zeta_2 \in \mathbb R^+$ and  
any two admissible states ${\bf U}_0 , {\bf U}_1 \in G$, one has
$$\zeta_1 {\bf U}_0 - \lambda \left[
{\bf F}^{hllc} (\zeta_1 {\bf U}_0, \zeta_2{\bf U}_1; {\bf n}) - 
{\bf F} ( \zeta_1 {\bf U}_0) \cdot {\bf n}   \right]   \in G.
$$
if $\lambda > 0$ and satisfies 
\begin{equation*}
\lambda  \max_{ {\bf U} \in \{  {\bf U}_0 , {\bf U}_1 \}  } \alpha_{\bf n} ({\bf U}  ) \le 1,
\qquad \text{ with }~~ \alpha_{\bf n} ({\bf U}  ) := |{\bf u}\cdot {\bf n}|+\sqrt{\gamma p/\rho}.
\end{equation*}
\end{lemma}


Based on the above properties, our well-balanced numerical fluxes are chosen as the modified HLLC flux
\begin{equation}\label{2DWBflux}
\widehat{\bf F}_{{\bf n}_{{\mathscr E},K}} = {\bf F}^{hllc} \left(
\frac{  p^{e,\star}_h }{ p^{e,{\tt int}(K)}_h }
{\bf U}_h^{{\tt int}(K)},~\frac{ p^{e,\star}_h }{ p^{e,{\tt ext}(K)}_h } {\bf U}_h^{{\tt ext}(K)};~ {\bf n}_{{\mathscr E},K}
\right),
\end{equation}
with $p^{e,\star}_h := \frac12 \big( p^{e,{\tt int}(K)}_h + p^{e,{\tt ext}(K)}_h \big)$. 
Using the $N$-point Gauss quadrature with $N=k+1$, we obtain the following approximation to the edge 
integral of numerical flux in \eqref{eq:2DDG}
\begin{equation}\label{eq:2DWBflux_edge}
\int_{ {\mathscr E}  }  \widehat{\bf F}_{{\bf n}_{{\mathscr E},K}} v^{{\tt int}(K)} {\rm d}s \approx 
|{\mathscr E} | \sum_{\mu=1}^N \omega_\mu 
\widehat{\bf F}_{{\bf n}_{{\mathscr E},K}} ( {\bf x}_{\mathscr E}^{(\mu)} ) v^{{\tt int}(K)} ({\bf x}_{\mathscr E}^{(\mu)}),
\end{equation}
where $|{\mathscr E}|$ is the length of the edge ${\mathscr E}$, 
$\{ {\bf x}_{\mathscr E}^{(\mu)}, ~\omega_\mu \}_{1\le \mu \le N}$ denote the set of 1D $N$-point Gauss quadrature nodes and weights on the edge ${\mathscr E}$.

\subsubsection{Source term approximations}
Let ${\bf S} =: ( 0, {\bf S}^{[2]}, S^{[3]} )^\top$ with ${\bf S}^{[2]} :=-\rho  {\bm \nabla} \phi$. 
We decompose the integral of the source terms in the momentum equations as 
\begin{align*}
& \int_{K} {\bf S}^{[2]} v {\rm d} {\bf x}  = \int_{I_j}  -\rho  {\bm \nabla} \phi v {\rm d} {\bf x} = \int_{K} \frac{\rho }{\rho^e} {\bm \nabla} p^e v {\rm d} {\bf x}
= \int_{K} \left( \frac{\rho }{\rho^e} - 
\frac{ \overline \rho_K }{ \overline \rho^e_K}
+ \frac{ \overline \rho_K }{ \overline \rho^e_K } 
\right) {\bm \nabla} p^e v {\rm d} {\bf x}
\\ 
& \qquad = \int_{K} \left( \frac{\rho }{\rho^e} - 
\frac{ \overline \rho_K }{ \overline \rho^e_K }
\right) {\bm \nabla} p^e v {\rm d}x + \frac{ \overline \rho_K }{ \overline \rho^e_K } \left( 
\sum_{ {\mathscr E} \in \partial K }  \int_{ {\mathscr E}  } p^{e} v^{ {\tt int}(K) } {\bf n}_{ {\mathscr E}, K } {\rm d} s - \int_{K}  p^e {\bm \nabla} v {\rm d} {\bf x}
\right) ,
\end{align*}
where ${\bm \nabla} p^e = -\rho^e {\bm \nabla} \phi $ has been used in the second identity, and the notation $\overline{(\cdot)}_K$ denotes the cell average of the associated quantity over the cell $K$.
This source term is then approximated by 
\begin{align} \label{eq:2Ds2}
& \int_{K} {\bf S}^{[2]} v {\rm d} x  \approx |K| \sum_{ q=1}^Q \varpi_q \left( \frac{\rho_h ( {\bf x}_K^{(q)} ) }{\rho^e_h  ( {\bf x}_K^{(q)} ) } - 
\frac{ \overline {(\rho_h)}_K }{ \overline {(\rho^e_h)}_K  }
\right) {\bm \nabla} p_h^e (   {\bf x}_K^{(q)}  ) v  (  {\bf x}_K^{(q)}  ) 
\\ \nonumber
& + \frac{ \overline {(\rho_h)}_K }{ \overline {(\rho^e_h)}_K  } \left[
\sum_{ {\mathscr E} \in \partial K }  \left( |{\mathscr E}| \sum_{\mu=1}^N \omega_\mu p^{e,\star}_h ( {\bf x}_{\mathscr E}^{(\mu)} ) v^{ {\tt int}(K) } ( {\bf x}_{\mathscr E}^{(\mu)} ) {\bf n}_{ {\mathscr E}, K } \right)
- |K| \sum_{q=1}^Q \varpi_q p^e_h  ( {\bf x}_K^{(q)} ) {\bm \nabla} v  ( {\bf x}_K^{(q)} ) 
\right] =: \big\langle {\bf S}^{[2]}, v \big\rangle_K,
\end{align}
where $|K|$ is the area of the cell $K$, 
$\{ {\bf x}_K^{(q)}, ~ \varpi_K^{(q)}\}_{1\le q \le Q}$ denote a set of 2D quadrature nodes and weights in $K$. 
Similarly, we approximate the integral of the source term in the energy equation by 
\begin{align} \label{eq:2Ds3}
& \int_{K} S^{[3]} v {\rm d} x  \approx |K| \sum_{ q=1}^Q \varpi_q \left( \frac{ {\bf m}_h ( {\bf x}_K^{(q)} ) }{\rho^e_h  ( {\bf x}_K^{(q)} ) } - 
\frac{ \overline {( {\bf m}_h)}_K }{ \overline {(\rho^e_h)}_K  }
\right) \cdot  {\bm \nabla} p_h^e (   {\bf x}_K^{(q)}  ) v  (  {\bf x}_K^{(q)}  ) 
\\ \nonumber
&  + \frac{ \overline {( {\bf m}_h)}_K }{ \overline {(\rho^e_h)}_K  }  \left[
\sum_{ {\mathscr E} \in \partial K }  \left( |{\mathscr E}| \sum_{\mu=1}^N \omega_\mu p^{e,\star}_h ( {\bf x}_{\mathscr E}^{(\mu)} ) v^{ {\tt int}(K) } ( {\bf x}_{\mathscr E}^{(\mu)} ) {\bf n}_{ {\mathscr E}, K } \right)
- |K| \sum_{q=1}^Q \varpi_q p^e_h  ( {\bf x}_K^{(q)} ) {\bm \nabla} v  ( {\bf x}_K^{(q)} ) 
\right] =: \big\langle S^{[3]}, v \big\rangle_K.
\end{align}

\subsubsection{Well-balanced DG methods}
The element integral $\int_{K}    
{\bf F}  ( {\bf U}_h ) \cdot {\bm \nabla} v  {\rm d} {\bf x}$ should be approximated by 
the same 2D quadrature set
\begin{equation}\label{eq:2Dfluxelem}
\int_{K}    
{\bf F}  ( {\bf U}_h ) \cdot {\bm \nabla} v  {\rm d} {\bf x} \approx 
|K| \sum_{ q=1}^Q \varpi_q  {\bf F}  \big( {\bf U}_h(  {\bf x}_K^{(q)}  )   \big) \cdot {\bm \nabla} v  (  {\bf x}_K^{(q)}  ).
\end{equation}
Substituting the approximations \eqref{2DWBflux}--\eqref{eq:2Dfluxelem} into 
\eqref{eq:2DDG} gives the following DG formulation
\begin{equation}\label{eq:2DDGweak}
\begin{aligned}
\int_{K} ({\bf U}_h)_t v {\rm d} x & = 
|K| \sum_{ q=1}^Q \varpi_q  {\bf F}  \big( {\bf U}_h(  {\bf x}_K^{(q)}  )   \big) \cdot {\bm \nabla} v  (  {\bf x}_K^{(q)}  ) + \Big( 0, \big\langle {\bf S}^{[2]}, v \big\rangle_j, \big\langle S^{[3]}, v \big\rangle_j  \Big)^\top
\\
& \quad 
- \sum_{ {\mathscr E} \in \partial K } \left(  |{\mathscr E} | \sum_{\mu=1}^N \omega_\mu 
\widehat{\bf F}_{{\bf n}_{{\mathscr E},K}} ( {\bf x}_{\mathscr E}^{(\mu)} ) v^{{\tt int}(K)} ({\bf x}_{\mathscr E}^{(\mu)}) \right),  \qquad \forall v \in \mathbb V_h^k.
\end{aligned}
\end{equation}

\begin{theorem}\label{thm:WB2D}
For the 2D Euler equations \eqref{eq:dD} with gravitation, the semi-discrete DG schemes \eqref{eq:2DDGweak} are well-balanced for a general known stationary hydrostatic solution  \eqref{eq:equi1}. 
\end{theorem}

The proof is similar to that of Theorem \ref{thm:1DWB} and is thus omitted.

\subsection{Positivity of first-order well-balanced DG scheme}
Denote the cell average of ${\bf U}_h({\bf x},t)$ over $K$ by $\overline {\bf U}_K(t)$,
and take the test function $v=1$ in \eqref{eq:2DDGweak}. We obtain the semi-discrete 
evolution equations satisfied by the cell average as 
\begin{equation}\label{eq:AVEevolve2D}
\frac{{\rm d} \overline {\bf U}_K (t)} { {\rm d} t} = {\bf L}_K ( {\bf U}_h ) := -\frac{1}{|K|} \sum_{ {\mathscr E} \in \partial K } \left(  |{\mathscr E} | \sum_{\mu=1}^N \omega_\mu 
\widehat{\bf F}_{{\bf n}_{{\mathscr E},K}} ( {\bf x}_{\mathscr E}^{(\mu)} ) \right) + \overline{\bf S}_K,
\end{equation}
where $\overline{\bf S}_K= \big( 0, \overline {\bf S}_K^{[2]}, \overline S_K^{[3]} \big)^\top$ with 
$\overline {\bf S}_K^{[\ell]} := \frac1{|K|} \left\langle {\bf S}^{[\ell]}, 1 \right\rangle_K$ for $\ell=2~,3$.

We start with showing the positivity of the first-order ($k=0$) well-balanced DG scheme \eqref{eq:2DDGweak}. For each $K\in {\mathcal T}_h$, 
let $K_{\mathscr E}$ denote the adjacent cell that shares the edge ${\mathscr E}$ with $K$, and define 
\begin{align*}
\widehat \alpha_K^F := \max\left\{
\max_{ {\mathscr E} \in \partial K } \alpha_{{\bf n}_{{\mathscr E},K} } ( \overline {\bf U}_{K} ), 
\max_{ {\mathscr E} \in \partial K } \alpha_{ {\bf n}_{{\mathscr E},K} } ( \overline {\bf U}_{K_{\mathscr E}} )
\right\}, \qquad \widehat \alpha_K^S := 
\frac{ \left\| \sum_{ {\mathscr E} \in \partial K } |{\mathscr E}| p^{e,\star}_{ {\mathscr E}, K } {\bf n}_{ {\mathscr E}, K }  \right\| } { |K| \overline \rho_K^e \sqrt{ 2 \overline e_K } },
\end{align*}
where $p^{e,\star}_{ {\mathscr E}, K }:= ( \overline p_K^e + \overline p^e_{ K_{\mathscr E} }  )/2$. 

\begin{theorem}\label{thm:2D1st}
If the DG polynomial degree $k=0$ and $\overline {\bf U}_{K} \in G$ for all $K\in {\mathcal T}_h$, we have 
\begin{equation}\label{eq:2Dorder1}
\overline {\bf U}_K + \Delta t {\bf L}_K ( {\bf U}_h  ) \in G, \quad \forall K \in {\mathcal T}_h,
\end{equation}
under the CFL-type condition 
\begin{equation}\label{eq:CFL1F2D}	
\Delta t \left(  2 \frac{\widehat \alpha_K^F}{|K|} \sum_{ {\mathscr E} \in \partial K } 
|{\mathscr E}| \frac{  p^{e,\star}_{ {\mathscr E}, K } }{  \overline p_K^e} 
+ \widehat \alpha_K^S
\right)  \le 1.
\end{equation}
\end{theorem}

\begin{proof}
Note that, for $k=0$, 
$ {\bf U}_h ({\bf x},t) \equiv \overline {\bf U}_K(t)$ for all ${\bf x}\in K$. We have 
\begin{align*}
\overline {\bf U}_K + \Delta t {\bf L}_K ( {\bf U}_h  ) & = \overline {\bf U}_K
-\frac{\Delta t}{|K|} \sum_{ {\mathscr E} \in \partial K }   |{\mathscr E} | 
{\bf F}^{hllc} \left(
\frac{  p^{e,\star}_{ {\mathscr E}, K } }{  \overline p_K^e}
\overline {\bf U}_K,\frac{  p^{e,\star}_{ {\mathscr E}, K } }{  \overline p^e_{ K_{\mathscr E} } } \overline{\bf U}_{ K_{\mathscr E} };{\bf n}_{{\mathscr E},K}
\right)
+ \Delta t \overline{\bf S}_K
\\
& =  \overline {\bf U}_K -\frac{\Delta t}{|K|} \sum_{ {\mathscr E} \in \partial K }  \left(  |{\mathscr E} | \frac{  p^{e,\star}_{ {\mathscr E}, K } }{  \overline p_K^e} {\bf F}  (\overline {\bf U}_K) \cdot {\bf n}_{{\mathscr E},K} \right) 
+  \Delta t \overline{\bf S}_K
\\
& \quad + \frac{\Delta t}{|K|} \sum_{ {\mathscr E} \in \partial K }   |{\mathscr E} | \left[  {\bf F}  \left(
\frac{  p^{e,\star}_{ {\mathscr E}, K } }{  \overline p_K^e}
\overline {\bf U}_K \right) \cdot {\bf n}_{{\mathscr E},K} - 
{\bf F}^{hllc} \left(
\frac{  p^{e,\star}_{ {\mathscr E}, K } }{  \overline p_K^e}
\overline {\bf U}_K,\frac{  p^{e,\star}_{ {\mathscr E}, K } }{  \overline p^e_{ K_{\mathscr E} } } \overline{\bf U}_{ K_{\mathscr E} };{\bf n}_{{\mathscr E},K}
\right) \right], 
\end{align*}
where the homogeneous property ${\bf F}(a{\bf U}) = a {\bf F}({\bf U})$ for any $a \in \mathbb R^+$ has been used. 
We further split $\overline {\bf U}_K + \Delta t {\bf L}_K ( {\bf U}_h  )$ into four parts as
\begin{equation}\label{key111}
\overline {\bf U}_K + \Delta t {\bf L}_K ( {\bf U}_h  ) = 
{\bf W}_1 + {\bf W}_2 + {\bf W}_3 + {\bf W}_4,
\end{equation} 
with  
\begin{align*}
& {\bf W}_1 := \left[ 1- 	 \Delta t   \left(  2 \frac{ \widehat \alpha_K^F}{|K|} \sum_{ {\mathscr E} \in \partial K } 
|{\mathscr E}| \frac{  p^{e,\star}_{ {\mathscr E}, K } }{  \overline p_K^e} 
+ \widehat \alpha_K^S
\right)  \right] \overline {\bf U}_K,
\\
& {\bf W}_2 := \frac{\Delta t}{|K|} \sum_{ {\mathscr E} \in \partial K } |{\mathscr E} | \widehat \alpha_K^F \frac{  p^{e,\star}_{ {\mathscr E}, K } }{  \overline p_K^e}  \left(
\overline {\bf U}_K - \frac{1}{ \widehat \alpha_K^F }
{\bf F}   (\overline {\bf U}_K) \cdot {\bf n}_{{\mathscr E},K} \right), \qquad 
{\bf W}_3:=  \Delta t  \left( \widehat \alpha_K^S  \overline {\bf U}_K + 
\overline {\bf S}_K \right),
\\
& {\bf W}_4 := \frac{\Delta t}{|K|} \sum_{ {\mathscr E} \in \partial K }   |{\mathscr E} | \widehat \alpha_K^F \left\{ 
\frac{  p^{e,\star}_{ {\mathscr E}, K } }{  \overline p_K^e}
\overline {\bf U}_K - \frac{1}{ \alpha_K^F } \left[ 
{\bf F}^{hllc} \left(
\frac{  p^{e,\star}_{ {\mathscr E}, K } }{  \overline p_K^e}
\overline {\bf U}_K,\frac{  p^{e,\star}_{ {\mathscr E}, K } }{  \overline p^e_{ K_{\mathscr E} } } \overline{\bf U}_{ K_{\mathscr E} };{\bf n}_{{\mathscr E},K}
\right) 
-  {\bf F}  \left(
\frac{  p^{e,\star}_{ {\mathscr E}, K } }{  \overline p_K^e}
\overline {\bf U}_K \right) \cdot {\bf n}_{{\mathscr E},K} 
\right] 
\right\}.
\end{align*} 
By using Lemma \ref{lem2}, it is easy to observe that ${\bf W}_1 \in \overline G$ under the condition \eqref{eq:CFL1F2D}. Lemma \ref{lem:LFflux} leads to
$ \overline {\bf U}_K - \frac{1}{ \widehat \alpha_K^F }
{\bf F}   (\overline {\bf U}_K) \cdot {\bf n}_{{\mathscr E},K} \in G$, which implies 
${\bf W}_2 \in G$ with the aid of Lemma \ref{lem3}. Note that 
\begin{equation*}
\widehat \alpha_K^S  \overline {\bf U}_K + \overline {\bf S}_K 
=  \widehat \alpha_K^S  \overline {\bf U}_K + \frac{1}{ |K| \overline \rho_K^e } 
\big( 
0,~ 
\overline \rho_K  {\bf a},~ 
\overline {\bf m}_K \cdot {\bf a} \big)^\top, 
\qquad \quad {\bf a}:= \sum_{ {\mathscr E} \in \partial K } |{\mathscr E}| p_{ {\mathscr E}, K }^{e,\star} {\bf n}_{ {\mathscr E}, K },
\end{equation*}
and $ \frac{1}{ |K| \overline \rho_K^e } \frac{ \| {\bf a}\| }{  \sqrt{ 2 \overline e_K } } 
= \widehat \alpha_K^S $. This yields $ \widehat \alpha_K^S  \overline {\bf U}_K + \overline {\bf S}_K \in \overline G$ by Lemma \ref{lem:UcontrolS}. 
Thus ${\bf W}_3 \in \overline G$. 
Sequentially using Lemma \ref{HLLC_std_euler2D00} and Lemma \ref{lem3} yields ${\bf W}_4\in G$.  Because ${\bf W}_1,{\bf W}_3 \in \overline G$ and ${\bf W}_2,{\bf W}_4 \in G$, we conclude from \eqref{key111} that $\overline {\bf U}_K + \Delta t {\bf L}_K ( {\bf U}_h  ) \in G$, which completes the proof. 
\end{proof}

Theorem \ref{thm:2D1st} indicates that the first-order ($k = 0$)
well-balanced DG method \eqref{eq:2DDGweak}, coupled with the forward Euler time
discretization, is positivity-preserving under the CFL-type condition \eqref{eq:CFL1F2D}.

\subsection{Positivity-preserving high-order well-balanced DG schemes}

When the DG polynomial degree $k \ge 1$, the
high-order well-balanced DG schemes \eqref{eq:2DDGweak} are not positivity-preserving in general. 
Similar to the 1D case, 
we can prove that our schemes satisfy a weak positivity property, which is crucial and implies that 
a simple limiter can enforce the positivity-preserving property without losing conservation and high-order accuracy. 

\subsubsection{Theoretical positivity-preserving analysis}
Assume that there exists a special 2D quadrature on each cell $K \in {\mathcal T}_h$ satisfying:
\begin{enumerate}[label=(\roman*)]
\item The quadrature rule has positive weights and is
exact for integrals of polynomials of degree up to $k$ on the cell $K$; 
\item The set of the quadrature points, denoted by $ {\mathbb S}_K^{(1)}$, must include all the Gauss quadrature points $ {\bf x}_{\mathscr E}^{(\mu)}$, $\mu=1,\dots,N$, 
on all the edges ${\mathscr E} \in \partial K$. 
\end{enumerate}
In other words, we would like to have a special quadrature such that 
\begin{equation}\label{eq:decomposition}
\frac{1}{|K|}\int_K u({\bf x}) {\rm d}{\bf x}  = 
\sum_{{\mathscr E} \in \partial K} \sum_{\mu=1}^N \widehat \varpi_{{\mathscr E}}^{(\mu)} u ( {\bf x}_{{\mathscr E}}^{(\mu)} ) + \sum_{ q=1 }^{\widetilde Q} \widetilde \varpi_q   
u ( \widetilde {\bf x}_K^{(q)} ), \quad \forall 
u \in {\mathbb P}^{k}(K) ,
\end{equation}
where $\{\widetilde {\bf x}_K^{(q)}\}$ are the other (possible) quadrature points in $K$, and the quadrature weights 
$\widehat \varpi_{{\mathscr E}}^{(\mu)}$ and $\widetilde \varpi_q $ are positive.
For rectangular cells, this quadrature was constructed in \cite{zhang2010,zhang2010b} by tensor products of Gauss quadrature and Gauss--Lobatto quadrature. 
For triangular cells and more general polygons, see \cite{zhang2012maximum,du2018positivity} 
for how to construct such quadrature. 
We remark that this special quadrature is only used in the proof and the positivity-preserving limiter presented later, and will not be used to evaluate any integral 
in the numerical implementation.
With this, we can define the point set 
\begin{align} \label{2DPPset}
\mathbb S_K & :=   {\mathbb S}_K^{(1)} \cup  {\mathbb S}_K^{(2)}
=
\big\{ {\bf x}_{{\mathscr E}}^{(\mu)}: {\mathscr E} \in \partial K, 1\le \mu \le  N \big\} \cup \big\{ \widetilde {\bf x}_{K}^{(q)}:  1\le q \le  \widetilde Q \big\}  \cup  \big\{  {\bf x}_{K}^{(q)}:  1\le q \le   Q \big\}, 
\end{align}
where $ {\mathbb S}_K^{(2)}:=\{{\bf x}_{K}^{(q)}\}_{1\le q \le Q}$ are the 2D quadrature points involved in the approximations \eqref{eq:2Ds2}--\eqref{eq:2Dfluxelem}. 

For convenience we will frequently use the following shorten notations
\begin{alignat*}{3}
& {\bf U}^{{\tt int}(K)}_{{\mathscr E},\mu }:=
{\bf U}_h^{{\tt int}(K)} ( {\bf x}_{\mathscr E}^{(\mu )} ), \qquad 
&{\bf U}^{{\rm ext}(K)}_{{\mathscr E},\mu }:=
{\bf U}_h^{{\tt ext}(K)} ( {\bf x}_{\mathscr E}^{(\mu )} ), \qquad
& p^{e,\star }_{{\mathscr E},\mu }:=
p_h^{e,\star } ( {\bf x}_{\mathscr E}^{(\mu )} ),
\\
& p^{e,{\tt int}(K)}_{{\mathscr E},\mu }:= 
p_h^{e,{\tt int}(K)} ( {\bf x}_{\mathscr E}^{(\mu )} ), \qquad 
&p^{e,{\rm ext}(K)}_{{\mathscr E},\mu }:=
p_h^{e,{\tt ext}(K)} ( {\bf x}_{\mathscr E}^{(\mu )} ), \qquad 
&\jump{p_h^e ( {\bf x}_{\mathscr E}^{(\mu)} ) }
:= p^{e,{\tt ext}(K)}_{{\mathscr E},\mu } - p^{e,{\tt int}(K)}_{{\mathscr E},\mu }.
\end{alignat*}

\begin{theorem}\label{thm:2Dhigh}
Assume that the projected stationary hydrostatic solution satisfies 
\begin{equation}\label{2DPPcondition:depe}
\rho^e_h( {\bf x} )>0, \quad p^e_h( {\bf x})>0, \quad     \forall {\bf x} \in \mathbb S_K,~~\forall K \in {\mathcal T}_h,
\end{equation}
and the numerical solution ${\bf U}_h$ satisfies 
\begin{equation}\label{2DPPcondition}
{\bf U}_h ( {\bf x} ) \in G,  \quad     \forall {\bf x} \in \mathbb S_K,~~\forall K \in {\mathcal T}_h,
\end{equation} 
then we have
\begin{equation}\label{2DPPhigh}
\overline {\bf U}_K + \Delta t {\bf L}_K ( {\bf U}_h  ) \in G,~~\forall K \in {\mathcal T}_h,
\end{equation}
under the CFL-type condition 
\begin{equation}\label{CFL2high2D}
\Delta t \left( 
\widetilde \alpha_K^F \frac{ 2  |{\mathscr E}| p^{e,\star }_{{\mathscr E},\mu } }
{ |K| p^{e,{\tt int}(K)}_{{\mathscr E},\mu } }
+ \widetilde \alpha_K^S \frac{ \widehat \varpi_{{\mathscr E}}^{(\mu)} }{ \omega_\mu }
\right) \le \frac{ \widehat \varpi_{{\mathscr E}}^{(\mu)} }{ \omega_\mu }, 
\qquad   1\le \mu \le N,~\forall{\mathscr E} \in \partial K,~\forall K \in {\mathcal T}_h,
\end{equation}
where 
\begin{align*}
&\widetilde \alpha_K^F := \max\Big\{
\max_{ {\mathscr E} \in \partial K, 1\le \mu \le N } \alpha_{{\bf n}_{{\mathscr E},K} } ( {\bf U}^{{\tt int}(K)}_{{\mathscr E},\mu } ), 
\max_{ {\mathscr E} \in \partial K, 1\le \mu \le N } \alpha_{ {\bf n}_{{\mathscr E},K} } ( {\bf U}^{{\tt ext}(K)}_{{\mathscr E},\mu } )
\Big\}, \qquad \widetilde \alpha_K^{S} = \widetilde \alpha_K^{S,1} + \widetilde \alpha_K^{S,2},
\\
&
\widetilde \alpha_K^{S,1} :=\max_{1\le q \le Q} \left\{  
\frac{ \left \|  {\bm \nabla} p_h^e (   {\bf x}_K^{(q)}  ) \right \| }  { 
	\rho_h^e ( {\bf x}_K^{(q )} )
	\sqrt{ 
		2 e_h (   {\bf x}_K^{(q)}  ) } } 
\right\}, 
\quad 
\widetilde \alpha_K^{S,2} :=
\frac{ \left\| \sum_{ {\mathscr E} \in \partial K }  \left( |{\mathscr E}| \sum_{\mu=1}^N \omega_\mu 
	\jump{p_h^e ( {\bf x}_{\mathscr E}^{(\mu)} ) }  
	\right){\bf n}_{ {\mathscr E}, K } \right\| }{ 
	2  |K| \overline \rho_K^e
	\sqrt{ 2 \overline e_K } }.
\end{align*}
\end{theorem}

\begin{proof}
For the modified HLLC flux, applying Lemmas \ref{HLLC_std_euler2D00} and \ref{lem3} yields
{\small \begin{equation}\label{eq:2DproofW1}
	{\bf W}_1 := \frac{ \Delta t  }{ |K| } \widetilde \alpha_K^F \sum_{ {\mathscr E} \in \partial K } 
	|{\mathscr E}| \sum_{ \mu=1}^N \omega_\mu  
	\left(\frac{ p^{e,\star }_{{\mathscr E},\mu } }{ p^{e,{\tt int}(K)}_{{\mathscr E},\mu } } 
	{\bf U}^{{\tt int}(K)}_{{\mathscr E},\mu } \right. 
	\left.  
	- \frac{1}{ \widetilde \alpha_K^F } 
	\left[ \widehat{\bf F}_{{\bf n}_{{\mathscr E},K}} ( {\bf x}_{\mathscr E}^{(\mu)} )  - {\bf F} \left( \frac{ p^{e,\star }_{{\mathscr E},\mu } }{ p^{e,{\tt int}(K)}_{{\mathscr E},\mu } } 
	{\bf U}^{{\tt int}(K)}_{{\mathscr E},\mu }  \right) \cdot {\bf n}_{{\mathscr E},K}
	\right]  \right) \in G.  
	\end{equation}}
Using the formulas of ${\bf W}_1$ and ${\bf L}_K ( {\bf U}_h  )$ in 
\eqref{eq:2DproofW1} and 
\eqref{eq:AVEevolve2D}, respectively, we deduce that 
\begin{align}\nonumber
& \overline {\bf U}_K + \Delta t {\bf L}_K ( {\bf U}_h  ) - {\bf W}_1 - \Delta t 
\overline {\bf S}_K
\\  \nonumber
& = \overline {\bf U}_K  - 
\frac{ \Delta t  }{ |K| } \widetilde \alpha_K^F  \sum_{ {\mathscr E} \in \partial K } 
\left[  |{\mathscr E}| \sum_{ \mu=1}^N \omega_\mu   
\left(
\frac{ p^{e,\star }_{{\mathscr E},\mu } }{ p^{e,{\tt int}(K)}_{{\mathscr E},\mu } } 
{\bf U}^{{\tt int}(K)}_{{\mathscr E},\mu } 
+ \frac{1} { \widetilde \alpha_K^F }
{\bf F} \left( \frac{ p^{e,\star }_{{\mathscr E},\mu } }{ p^{e,{\tt int}(K)}_{{\mathscr E},\mu } } 
{\bf U}^{{\tt int}(K)}_{{\mathscr E},\mu }  \right) \cdot {\bf n}_{{\mathscr E},K}
\right)
\right]
\\ \nonumber
& = \overline {\bf U}_K  - 
2 \frac{ \Delta t  }{ |K| } \widetilde \alpha_K^F  \sum_{ {\mathscr E} \in \partial K } 
\left[  |{\mathscr E}| \sum_{ \mu=1}^N \omega_\mu   
\left(
\frac{ p^{e,\star }_{{\mathscr E},\mu } }{ p^{e,{\tt int}(K)}_{{\mathscr E},\mu } } 
{\bf U}^{{\tt int}(K)}_{{\mathscr E},\mu } 
\right)
\right]
\\ \label{eq:2Dsplit1}
& \quad + 
\frac{ \Delta t  }{ |K| } \widetilde \alpha_K^F  \sum_{ {\mathscr E} \in \partial K } 
\left[  |{\mathscr E}| \sum_{ \mu=1}^N \omega_\mu   
\left(
\frac{ p^{e,\star }_{{\mathscr E},\mu } }{ p^{e,{\tt int}(K)}_{{\mathscr E},\mu } } 
{\bf U}^{{\tt int}(K)}_{{\mathscr E},\mu } 
- \frac{1} { \widetilde \alpha_K^F }
\frac{ p^{e,\star }_{{\mathscr E},\mu } }{ p^{e,{\tt int}(K)}_{{\mathscr E},\mu } }  {\bf F} \left( 
{\bf U}^{{\tt int}(K)}_{{\mathscr E},\mu }  \right) \cdot {\bf n}_{{\mathscr E},K}
\right)
\right],
\end{align}
where the homogeneous property ${\bf F}(a{\bf U}) = a {\bf F}({\bf U})$ for any $a \in \mathbb R^+$ is used. 
Applying Lemmas \ref{lem:LFflux} and \ref{lem3} implies that 
\begin{equation}\label{eq:2DproofW2}
{\bf W}_2 := \frac{ \Delta t  }{ |K| } \widetilde \alpha_K^F \sum_{ {\mathscr E} \in \partial K } 
\left( |{\mathscr E}| \sum_{ \mu=1}^N \omega_\mu 
\frac{ p^{e,\star }_{{\mathscr E},\mu } }{ p^{e,{\tt int}(K)}_{{\mathscr E},\mu } }
\left({\bf U}^{{\tt int}(K)}_{{\mathscr E},\mu } 
- \frac{1} { \widetilde \alpha_K^F } {\bf F} \left( {\bf U}^{{\tt int}(K)}_{{\mathscr E},\mu }  \right) \cdot {\bf n}_{{\mathscr E},K} \right)  \right) \in G.
\end{equation}
Based on equation \eqref{eq:2Dsplit1} and the definition of ${\bf W}_2$, we rewrite 
$\overline {\bf U}_K + \Delta t {\bf L}_K ( {\bf U}_h  )$ as 
\begin{equation}\label{eq:2Dsplit2}
\overline {\bf U}_K + \Delta t {\bf L}_K ( {\bf U}_h  ) = {\bf W}_1 + {\bf W}_2 + {\bf W}_3,
\end{equation}
with 
$$
{\bf W}_3 := \overline {\bf U}_K  - 
2 \frac{ \Delta t  }{ |K| } \widetilde \alpha_K^F  \sum_{ {\mathscr E} \in \partial K } 
\left[  |{\mathscr E}| \sum_{ \mu=1}^N \omega_\mu   
\left(
\frac{ p^{e,\star }_{{\mathscr E},\mu } }{ p^{e,{\tt int}(K)}_{{\mathscr E},\mu } } 
{\bf U}^{{\tt int}(K)}_{{\mathscr E},\mu } 
\right)
\right] + \Delta t \overline {\bf S}_K. 
$$
Recall that $\overline{\bf S}_K= \big( 0, \overline {\bf S}_K^{[2]}, \overline S_K^{[3]} \big)^\top$ with 
$\overline {\bf S}_K^{[\ell]} = \frac1{|K|} \left\langle {\bf S}^{[\ell]}, 1 \right\rangle_K$, $\ell=2,3$. We can reformulate $\overline {\bf S}_K^{[2]}$ as 
\begin{align*}
\overline {\bf S}_K^{[2]} & = 
\sum_{ q=1}^Q \varpi_q \left( \frac{\rho_h ( {\bf x}_K^{(q)} ) }{\rho^e_h  ( {\bf x}_K^{(q)} ) } - 
\frac{ \overline \rho_K }{ \overline \rho_K^e  }
\right) {\bm \nabla} p_h^e (   {\bf x}_K^{(q)}  ) 
+ \frac{ \overline \rho_K }{ \overline \rho_K^e  } \left[ \frac1{|K|}
\sum_{ {\mathscr E} \in \partial K }  \left( |{\mathscr E}| \sum_{\mu=1}^N \omega_\mu p^{e,\star}_h ( {\bf x}_{\mathscr E}^{(\mu)} )  {\bf n}_{ {\mathscr E}, K } \right)
\right]
\\
& = \sum_{ q=1}^Q \varpi_q  \frac{\rho_h ( {\bf x}_K^{(q)} ) }{\rho^e_h  ( {\bf x}_K^{(q)} ) } 
{\bm \nabla} p_h^e (   {\bf x}_K^{(q)}  ) 
+ \frac{ \overline \rho_K }{ \overline \rho_K^e |K| } \left[
\sum_{ {\mathscr E} \in \partial K }  \left( |{\mathscr E}| \sum_{\mu=1}^N \omega_\mu p^{e,\star}_h ( {\bf x}_{\mathscr E}^{(\mu)} )  {\bf n}_{ {\mathscr E}, K } \right)
-\int_K  {\bm \nabla} p_h^e ({\bf x}) {\rm d} {\bf x}
\right]
\\
& = \sum_{ q=1}^Q \varpi_q  \frac{\rho_h ( {\bf x}_K^{(q)} ) }{\rho^e_h  ( {\bf x}_K^{(q)} ) } 
{\bm \nabla} p_h^e (   {\bf x}_K^{(q)}  ) 
+ \frac{ \overline \rho_K }{ \overline \rho_K^e  |K| } 
\sum_{ {\mathscr E} \in \partial K }  \left( |{\mathscr E}| \sum_{\mu=1}^N \omega_\mu p^{e,\star}_h ( {\bf x}_{\mathscr E}^{(\mu)} )  - \int_{ {\mathscr E}  } 
p_h^e {\rm d} s
\right){\bf n}_{ {\mathscr E}, K }
\\
& = \sum_{ q=1}^Q \varpi_q  \frac{\rho_h ( {\bf x}_K^{(q)} ) }{\rho^e_h  ( {\bf x}_K^{(q)} ) } 
{\bm \nabla} p_h^e (   {\bf x}_K^{(q)}  ) 
+ \frac{ \overline \rho_K }{ 2 \overline \rho_K^e |K| }
{\bf a},
\end{align*}
with ${\bf a}:= \sum_{ {\mathscr E} \in \partial K }  \left( |{\mathscr E}| \sum_{\mu=1}^N \omega_\mu 
\jump{p_h^e ( {\bf x}_{\mathscr E}^{(\mu)} ) }  
\right){\bf n}_{ {\mathscr E}, K }$, 
where we used the divergence theorem and   
the exactness of the quadrature rules for polynomials of degree up to $k$. Similarly, $\overline S_K^{[3]}$ can be written as 
\begin{equation*}
\overline S_K^{[3]}  = \sum_{ q=1}^Q \varpi_q \left( \frac{ {\bf m}_h ( {\bf x}_K^{(q)} ) }{\rho^e_h  ( {\bf x}_K^{(q)} ) } \right) \cdot
{\bm \nabla} p_h^e (   {\bf x}_K^{(q)}  ) 
+ \frac{ \overline {\bf m}_K }{ 2 \overline \rho_K^e |K| } 
\cdot 
{\bf a}.
\end{equation*}
Therefore, $ \widetilde \alpha_K^S  
\overline {\bf U}_K + \overline {\bf S}_K= ( \widetilde \alpha_K^{S,1} + \widetilde \alpha_K^{S,2} ) 
\overline {\bf U}_K + \overline {\bf S}_K$ can be reformulated as 
$$
\sum_{q=1}^Q \varpi_q \left[
\widetilde \alpha_K^{S,1} {\bf U}_h ( {\bf x}_K^{(q)} ) + \frac{1}{ \rho_h^e ( {\bf x}_K^{(q )} )}
\begin{pmatrix}
0\\
\rho_h ( {\bf x}_K^{(q)} )  {\bm \nabla} p_h^e (   {\bf x}_K^{(q)}  )
\\
{\bf m}_h ( {\bf x}_K^{(q)} ) \cdot  {\bm \nabla} p_h^e (   {\bf x}_K^{(q)}  )
\end{pmatrix}
\right] + \left[ \widetilde \alpha_K^{S,2} \overline {\bf U}_K 
+ \frac{1}{ 2 \overline \rho_K^e |K| }
\begin{pmatrix}
0\\
\overline \rho_K {\bf a}
\\
\overline {\bf m}_h \cdot {\bf a}  
\end{pmatrix}
\right].
$$
Since
$$
\frac{1}{ \rho_h^e ( {\bf x}_K^{(q )} ) }  \frac{ \left \|  {\bm \nabla} p_h^e (   {\bf x}_K^{(q)}  ) \right \| }  { \sqrt{ 
		2 e_h (   {\bf x}_K^{(q)}  ) } } \le  \widetilde \alpha_K^{S,1}, \qquad 
\frac{1}{ 2 \overline \rho_K^e |K| } \frac{ \| {\bf a} \| }{ \sqrt{ 2 \overline e_K } } = \widetilde \alpha_K^{S,2},
$$ 
we conclude that
$ \widetilde \alpha_K^S 
\overline {\bf U}_K + \Delta t \overline {\bf S}_K \in \overline G 
$, 
according to Lemmas \ref{lem:UcontrolS} and \ref{lem3}. It follows that 
\begin{equation}\label{eq:W4belongG}
{\bf W}_4 := \Delta t  \left( \widetilde \alpha_K^S 
\overline {\bf U}_K +  \overline {\bf S}_K \right) \in \overline G. 
\end{equation}
Subtracting ${\bf W}_4$ from ${\bf W}_3$ gives 
\begin{equation}\label{eq:W3W4}
{\bf W}_3 - {\bf W}_4 = ( 1- \Delta t  \widetilde \alpha_K^S ) \overline {\bf U}_K  - 
2 \frac{ \Delta t  }{ |K| } \widetilde \alpha_K^F  \sum_{ {\mathscr E} \in \partial K } 
\left[  |{\mathscr E}| \sum_{ \mu=1}^N \omega_\mu   
\left(
\frac{ p^{e,\star }_{{\mathscr E},\mu } }{ p^{e,{\tt int}(K)}_{{\mathscr E},\mu } } 
{\bf U}^{{\tt int}(K)}_{{\mathscr E},\mu } 
\right)
\right]. 
\end{equation}
Note that the exactness of the quadrature rule \eqref{eq:decomposition} for polynomials of degree up to $k$ leads to
\begin{equation}\label{KK}
\overline {\bf U}_K  = 
\sum_{{\mathscr E} \in \partial K} \sum_{\mu=1}^N \widehat \varpi_{{\mathscr E}}^{(\mu)} {\bf U}^{{\tt int}(K)}_{{\mathscr E},\mu } + \sum_{ q=1 }^{\widetilde Q} \widetilde \varpi_q   
{\bf U}_h^{{\tt int}(K)} ( \widetilde {\bf x}_K^{(q)} )
=: \sum_{{\mathscr E} \in \partial K} \sum_{\mu=1}^N \widehat \varpi_{{\mathscr E}}^{(\mu)} {\bf U}^{{\tt int}(K)}_{{\mathscr E},\mu } + {\bf W}_5,
\end{equation}
and obviously we have ${\bf W}_5 \in \overline G$. Substituting \eqref{KK} into \eqref{eq:W3W4} yields 
\begin{equation*}
{\bf W}_3 - {\bf W}_4 =  ( 1- \Delta t  \widetilde \alpha_K^S ) 
{\bf W}_5 + 
\sum_{{\mathscr E} \in \partial K} \sum_{\mu=1}^N \omega_\mu
\left[ \frac{ \widehat \varpi_{{\mathscr E}}^{(\mu)} }{ \omega_\mu }  -  
\Delta t \left( 
\frac{ 2 \widetilde \alpha_K^F |{\mathscr E}| p^{e,\star }_{{\mathscr E},\mu } }
{ |K| p^{e,{\tt int}(K)}_{{\mathscr E},\mu } }
+ \widetilde \alpha_K^S \frac{ \widehat \varpi_{{\mathscr E}}^{(\mu)} }{ \omega_\mu }
\right)
\right] {\bf U}^{{\tt int}(K)}_{{\mathscr E},\mu }, 
\end{equation*}
which belongs to $\overline G$, by Lemma \ref{lem3}, under the CFL condition \eqref{CFL2high2D}. 
Recall that we have shown in \eqref{eq:W4belongG} that ${\bf W}_4 \in \overline G$. 
It then follows that ${\bf W}_3 = ( {\bf W}_3 - {\bf W}_4) + {\bf W}_4 \in \overline G $. 
Recalling $ {\bf W}_1 \in G$ and ${\bf W}_2 \in G$ in 
\eqref{eq:2DproofW1} and \eqref{eq:2DproofW2}, and from equation
\eqref{eq:2Dsplit2} and Lemma \ref{lem3}, we finally conclude \eqref{2DPPhigh}. 
This completes the proof.  
\end{proof}


Theorem \ref{thm:2Dhigh} provides a sufficient condition \eqref{2DPPcondition} for the proposed high-order well-balanced
DG schemes \eqref{eq:2DDGweak} to be positivity-preserving, when 
an SSP-RK time discretization is used. The condition \eqref{2DPPcondition} can again be enforced by a 
simple positivity-preserving limiter similar to the 1D case; see equations \eqref{eq:limiter}--\eqref{eq:limitertheta} with the 1D point set ${\mathbb S}_j$ replaced by the 2D point set \eqref{2DPPset} accordingly.  
With the limiter applied at each stage of the SSP-RK time steps, the fully discrete DG schemes are positivity-preserving.

\subsubsection{Illustration of some details on Cartesian meshes}
Assume that the mesh is rectangular with cells $\{[x_{i-1/2},x_{i+1/2}]\times [y_{\ell-1/2},y_{\ell+1/2}] \}$ 
and spatial step-sizes $\Delta x_i=x_{i+1/2}-x_{i-1/2}$ and $\Delta y_\ell=y_{\ell+1/2}-y_{\ell-1/2}$ in $x$- and $y$-directions respectively, where $(x,y)$ denotes the 2D spatial coordinate variables. 
Let ${\mathbb S}_i^x=\{ x_i^{(\mu)}  \}_{\mu=1}^N$
and ${\mathbb S}_\ell^y=\{ y_\ell^{(\mu)}  \}_{\mu=1}^N$ 
denote the $N$-point Gauss quadrature nodes in the intervals 
$[x_{i-1/2},x_{i+1/2}]$ and 
$[y_{\ell-1/2},y_{\ell+1/2}]$ respectively. 
For the cell $K=[x_{i-1/2},x_{i+1/2}]\times [y_{\ell-1/2},y_{\ell+1/2}]$, 
the point sets $ {\mathbb S}_K^{(1)}$ and ${\mathbb S}_K^{(2)}$ in \eqref{2DPPset} are given by (cf.~\cite{zhang2010})
\begin{equation}\label{eq:RectS}
{\mathbb S}_K^{(1)} = \big(  \widehat{\mathbb S}_i^x \otimes 
{\mathbb S}_\ell^y \big) \cup \big(  {\mathbb S}_i^x \otimes 
\widehat{\mathbb S}_\ell^y \big),
\qquad
{\mathbb S}_K^{(2)}={\mathbb S}_i^x \otimes {\mathbb S}_\ell^y
\end{equation}
where 
$\widehat{\mathbb S}_i^x=\{ \widehat x_i^{(\nu)}  \}_{\nu=1}^{L}$
and $\widehat {\mathbb S}_\ell^y=\{\widehat y_\ell^{(\nu)}  \}_{\nu=1}^{L}$ 
denote the $L$-point (${L} \ge \frac{k+3}2$) Gauss--Lobatto quadrature nodes in the intervals 
$[x_{i-1/2},x_{i+1/2}]$ and 
$[y_{\ell-1/2},y_{\ell+1/2}]$ respectively.  
With ${\mathbb S}_K^{(1)}$ in \eqref{eq:RectS}, a special 2D quadrature \cite{zhang2010} satisfying \eqref{eq:decomposition} can be constructed:
\begin{equation} \label{eq:U2Dsplit}
\begin{split}
\frac{1}{|K|}\int_K u({\bf x}) d {\bf x}
&= \sum \limits_{\mu = 1}^{ N}   \frac{ \Delta x_i \widehat \omega_1 \omega_\mu }{ \Delta x_i + \Delta y_\ell }   \left(  
u\big( x_i^{(\mu)},y_{\ell-\frac12} \big) 
+ u\big( x_i^{(\mu)},y_{\ell+\frac12} \big) 
\right)
\\
&
+  \sum \limits_{\mu = 1}^{ N} \frac{ \Delta y_\ell \widehat \omega_1 \omega_\mu }{ \Delta x_i + \Delta y_\ell }  \left( u \big(x_{i-\frac12},y_\ell^{(\mu)}\big) +
u\big(x_{i+\frac12},y_\ell^{(\mu)}\big) \right)
\\
& +  \sum \limits_{\nu = 2}^{{L}-1} \sum \limits_{\mu = 1}^{N} 
\frac{ \widehat \omega_\nu \omega_\mu  }{\Delta x_i + \Delta y_\ell}
\left( \Delta x_i u\big(  x_i^{(\mu)},\widehat y_\ell^{(\nu)} \big) + \Delta y_\ell u\big(\widehat x_i^{(\nu)},y_\ell^{(\mu)}\big) \right),
\quad~ \forall u \in {\mathbb P}^k(K),
\end{split}
\end{equation}
where $\{\widehat w_\mu\}_{\mu=1}^{L}$ are the weights of the $L$-point Gauss--Lobatto quadrature. 
If labeling the bottom, right, top and left edges of $K$ as  
${\mathscr E }_1$, ${\mathscr E }_2$, ${\mathscr E }_3$ and ${\mathscr E }_4$, respectively, the equation \eqref{eq:U2Dsplit} implies, for $1\le \mu \le N$, that
$
\varpi_{ {\mathscr E }_j  }^{(\mu)} = \frac{ \Delta x_i \widehat \omega_1 \omega_\mu}{ \Delta x_i + \Delta y_\ell },~j=1,3;~
\varpi_{ {\mathscr E }_j  }^{(\mu)} = \frac{ \Delta y_\ell \widehat \omega_1 \omega_\mu}{ \Delta x_i + \Delta y_\ell },~j=2,4.
$ 
According to Theorem \ref{thm:2Dhigh}, the CFL condition \eqref{CFL2high2D} for our positivity-preserving DG schemes on Cartesian meshes is 
\begin{equation}\label{eq:CFL:2DCart}
\Delta t \left[ 2 \widetilde \alpha_K^F 
\frac{  p^{e,\star }_{{\mathscr E}_j,\mu } }
{ p^{e,{\tt int}(K)}_{{\mathscr E}_j,\mu } }
\left( \frac{1}{\Delta x_i} 
+ \frac{1}{\Delta y_\ell} \right)   + \widetilde \alpha_K^S \widehat \omega_1 \right]
\le \widehat \omega_1
,\quad \forall
K\in{\mathcal T}_h,~1\le j \le 4,
\end{equation}
where $\widehat  \omega_1= \frac{1}{L(L-1)}$. 
Assume the mesh is regular and define $h =\max_{i,\ell} \{ \Delta x_i, \Delta y_\ell \}$, then   
for smooth $p^e({\bf x})$, it holds 
$$
\frac{  p^{e,\star }_{{\mathscr E}_j,\mu } }
{ p^{e,{\tt int}(K)}_{{\mathscr E}_j,\mu } }
= \frac12 + \frac{
p_h^{e,{\tt ext}(K)}  ( {\bf x}_{{\mathscr E}_j }^{(\mu)}  ) 
}{2 p_h^{e,{\tt int}(K)}  ( {\bf x}_{{\mathscr E}_j }^{(\mu)}  )} 
= 1 + {\mathcal O}(h^{k+1}),
$$
whose effect in the CFL condition \eqref{eq:CFL:2DCart} can be ignored. 

\section{Numerical tests}\label{sec:examples}
This section presents several 1D and 2D examples to demonstrate the well-balanced and positivity-preserving properties 
of the proposed DG methods on uniform Cartesian meshes. 
Without loss of generality, we only show the numerical results obtained by our third-order ($k=2$) DG method with the explicit third-order SSP-RK time discretization \eqref{SSP1}. 
For the sake of comparison, we will also show the numerical results of 
the traditional non-well-balanced (denoted as ``non-WB'') DG schemes with the straightforward 
source term discretization and the original HLLC flux. 
Unless otherwise stated, we use a CFL number of $0.2$ and the ideal equation of state \eqref{eq:IEOS} with $\gamma=1.4$. 
In all the tests, the method is implemented by using C++ language with double precision.

\subsection{Example 1: One-dimensional polytropic equilibrium}\label{sec:ex2} 
This test is used to investigate the performance of the proposed schemes near the polytropic equilibrium states 
\cite{KM2014}. 
Under the gravitational field $\phi(x) = gx$, the stationary hydrostatic solutions are 
\begin{equation} \label{state6}
\rho^e (x) = \left( \rho_0^{\gamma -1} - \frac{1}{K_0} \frac{\gamma -1}{\gamma} g x \right)^{ \frac1{\gamma -1} }, \qquad u^e(x)=0, \qquad p^e(x)= K_0 \left(\rho^e(x)\right)^\gamma, 
\end{equation}
with $g=1$, $\gamma = 5/3$, $\rho_0 = p_0 = 1$, and $K_0 =  p_0 / \rho_0^{\gamma}$ on a computational domain $[0,2]$.
%

We first use this example to check the well-balancedness of our DG methods. The
initial data are taken as the stationary hydrostatic solutions \eqref{state6}. 
We simulate this problem up to $t = 4$ by using our third-order well-balanced DG scheme with different mesh points, and list the $l^1$-errors of numerical solutions in Table \ref{tab:Ex2}. These errors are evaluated between the numerical solutions and the projected stationary hydrostatic solutions.
It is clearly observed that the numerical errors are all at the level of round-off error, 
which verify the desired well-balanced property.

\begin{table}[htbp]
\centering
\caption{\small Example 1: 
	$l^1$-errors on different meshes of $M$ uniform cells.
}\label{tab:Ex2}
\begin{tabular}{cccc}
	\hline
	$M$ & errors in $\rho$ & errors in $m$ & errors in $E$ \\
	\hline
	50 & 1.0682e-14 & 1.0332e-14 & 4.4756e-16
	\\
	\hline
	100 & 3.6074e-14 & 4.5115e-14 & 6.5160e-15
	\\
	\hline 
	200 & 5.2993e-14 & 4.9258e-14 & 7.8335e-15
	\\
	\hline
\end{tabular}
\end{table}

Next, a small perturbation is imposed to the stationary hydrostatic state \eqref{state6}, so as to compare the performance of
well-balanced and non-WB DG schemes in simulating the evolution of such small perturbation.
More specifically, we add a periodic velocity perturbation
$$
u(x,t) = A \sin(4 \pi t),
$$
with $A = 10^{-6}$ to the system on the left boundary $x=0$.
The solutions are computed until $t=1.5$, before the waves propagate to the right boundary $x=2$. 
Figure \ref{fig:ex2a} displays the pressure perturbation and the velocity at
$t= 1.5$, computed by the proposed third-order well-balanced DG scheme on a mesh of $100$ uniform cells,  against the reference solutions computed on a much refined mesh of $1000$ cells.
For comparison, we also perform the third-order non-WB DG method 
and show its results in the same figure. 
As we can see, the results by the well-balanced
DG method agree well with the reference ones, while the results by the non-WB DG method 
do not match the reference ones especially in the region where $x>1.5$.
This demonstrates that the well-balanced methods are advantageous and more accurate for 
resolving small amplitude perturbations to 
equilibrium states.

\begin{figure}[htbp]
\centering
\begin{subfigure}[b]{0.48\textwidth}
	\begin{center}
		\includegraphics[width=0.98\textwidth]{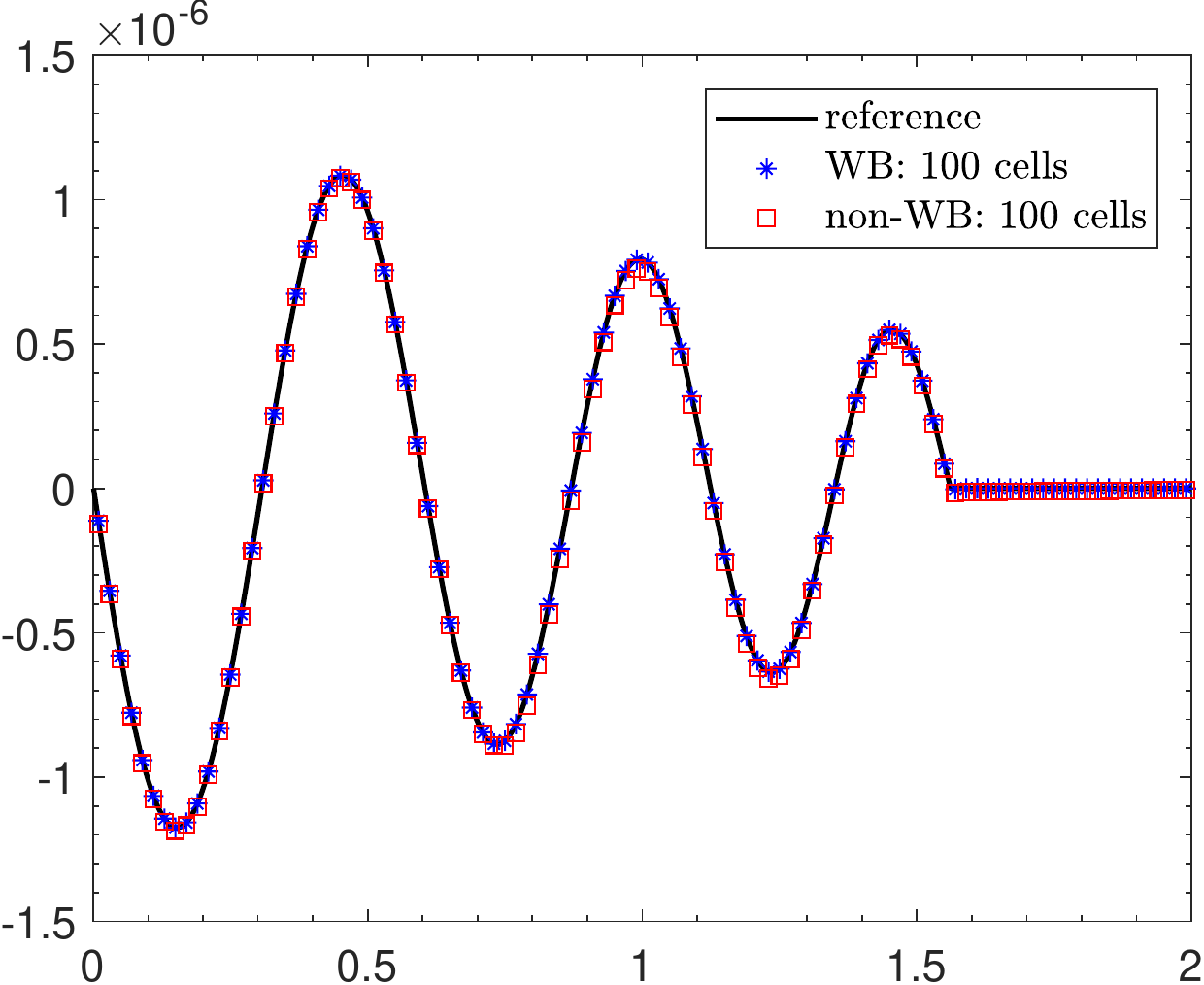}
	\end{center}
\end{subfigure}
\begin{subfigure}[b]{0.48\textwidth}
	\begin{center}
		\includegraphics[width=0.98\textwidth]{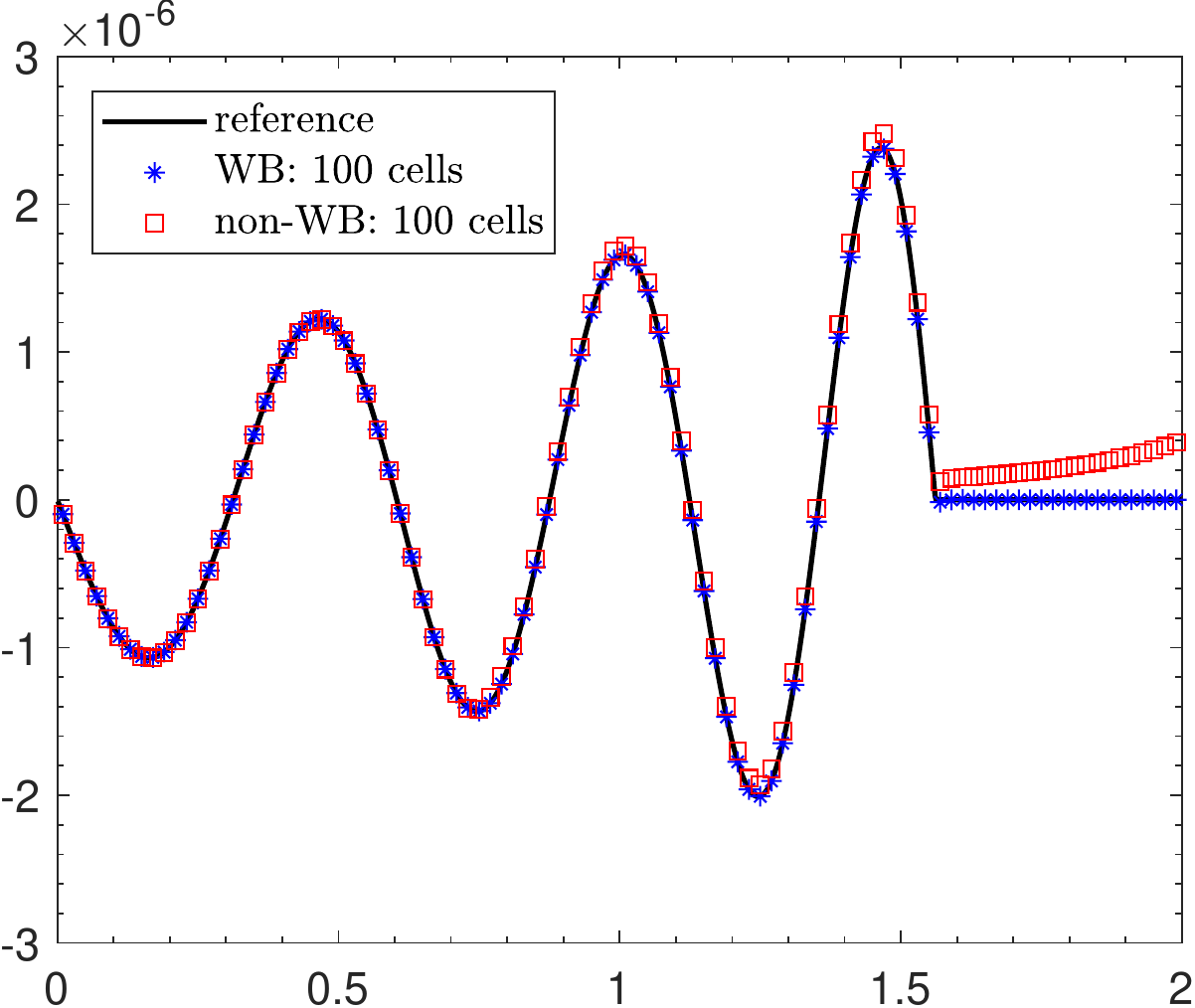}
	\end{center}
\end{subfigure}	
\caption{\small
	Example 1: Small amplitude waves with $A=10^{-6}$ traveling up the polytropic hydrostatic atmosphere. 
	The numerical solutions of well-balanced method (denoted by ``WB'') and non-WB method are obtained on the mesh of $100$ uniform cells. The reference solutions are computed by the well-balanced method using $1000$ mesh points. Left: pressure perturbation; Right: velocity.
}\label{fig:ex2a} 
\end{figure}

%

In the last test case of this example, we conduct the same simulation but with a large perturbation $A = 0.1$. 
We again evolve the simulation until $t=1.5$. 
Because the discontinuities are formed in the final solution, 
the WENO limiter \cite{Qiu2005} 
is implemented right before the positivity-preserving limiting procedure with the aid of the local
characteristic decomposition within a few  ``trouble'' cells detected adaptively. 
The numerical solutions by both the well-balanced and non-WB DG methods are shown in Figure \ref{fig:ex2},  against the reference solutions. One can see that both DG methods produce satisfactory results. 
This agrees with the normal expectation that the well-balanced methods perform similarly as non-WB methods in capturing solutions far away from steady states.

\begin{figure}[htbp]
\centering
\begin{subfigure}[b]{0.48\textwidth}
	\begin{center}
		\includegraphics[width=0.98\textwidth]{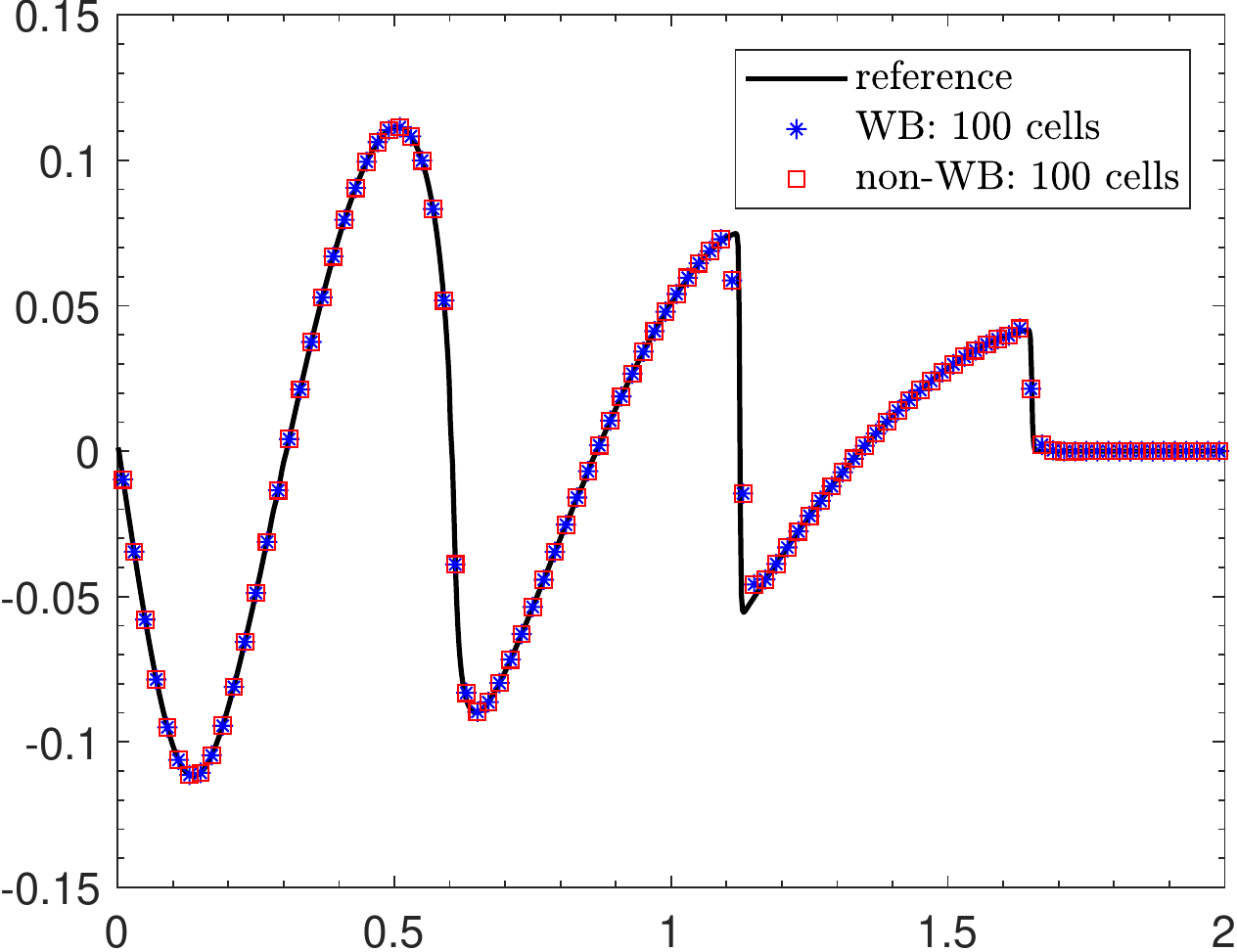}
	\end{center}
\end{subfigure}
\begin{subfigure}[b]{0.48\textwidth}
	\begin{center}
		\includegraphics[width=0.98\textwidth]{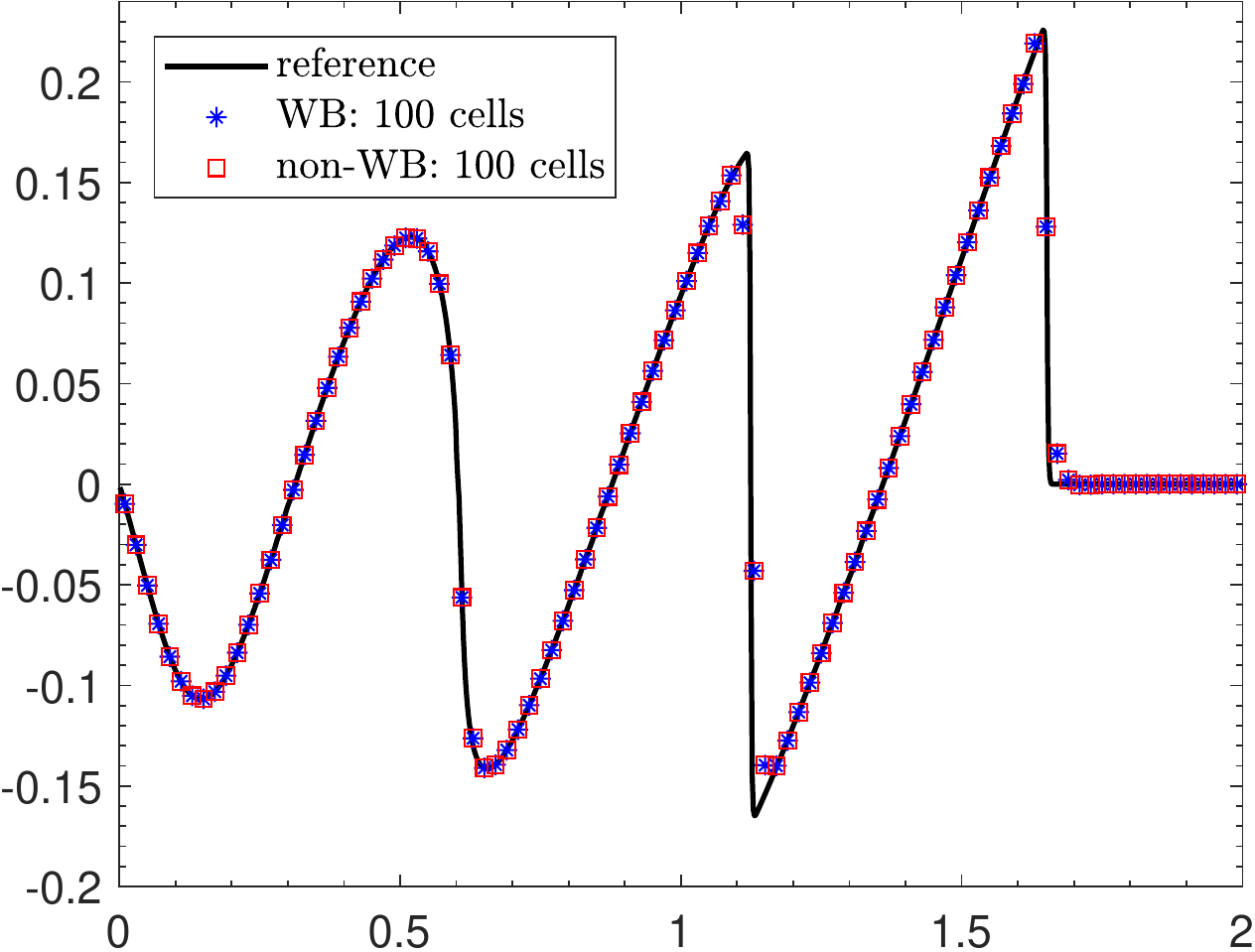}
	\end{center}
\end{subfigure}	
\caption{\small
	Same as Figure \ref{fig:ex2a} except for large amplitude waves with $A=0.1$ traveling up the polytropic hydrostatic atmosphere. Left: pressure perturbation; Right: velocity.
}\label{fig:ex2} 
\end{figure}

\subsection{Example 2: Rarefaction test with low density and low pressure}

To demonstrate the positivity-preserving property, we consider an extreme rarefaction test under a quadratic gravitational potential $\phi(x)=x^2/2$ centered around $x=0$. The computational domain is taken as $[-1,1]$, and the initial state is the same as a Riemann problem in \cite{zhang2010b}, given by  
$$
\rho(x,0)=7, \qquad p(x,0)=0.2, \qquad u(x,0)=\begin{cases}
-1, & \quad x<0,\\
1, & \quad x>0.
\end{cases}
$$ 
with outflow boundary conditions at $x=-1$ and $x=1$. 
This problem involves extremely low density and pressure, so that the positivity-preserving limiter should be employed. 
The CFL number is set as $0.15$, which is slightly smaller than $\widehat \omega_1 = \frac1{6}$. 
Figure~\ref{fig:ex4} gives 
the numerical results at $t=0.25$, obtained
by our positivity-preserving third-order well-balanced DG scheme, 
on a mesh with $800$ cells, compared with reference solutions obtained with much refined $128000$ cells. 
It is seen that the low density and low pressure wave structures are well captured by the proposed method. 
During the whole simulation, our scheme exhibits good robustness. We observe that it is necessary to enforce 
the condition \eqref{PPcondition}, otherwise the DG code will break down due to nonphysical solution.

\begin{figure}[htbp]
\centering
\begin{subfigure}[b]{0.3\textwidth}
	\begin{center}
		\includegraphics[width=1\textwidth]{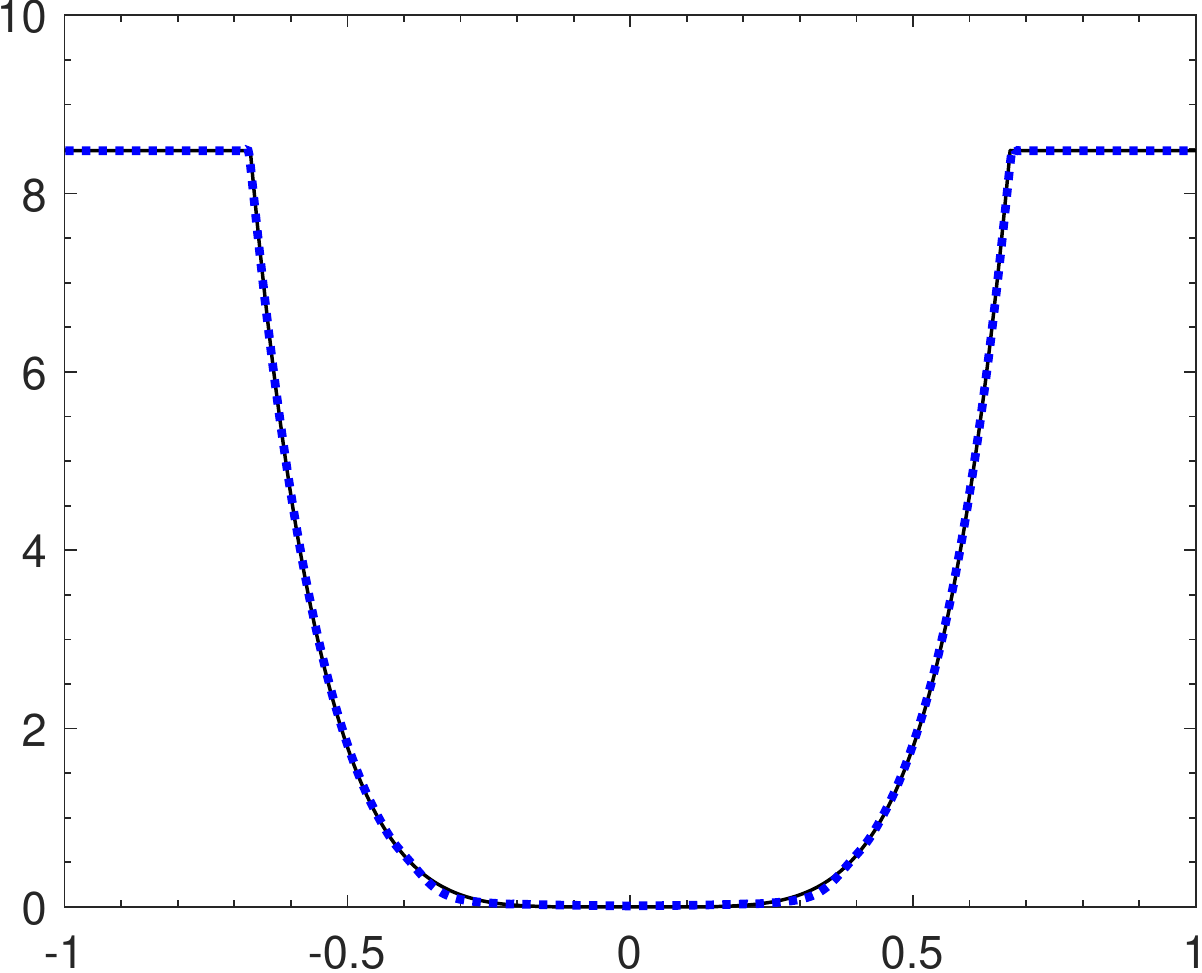}
	\end{center}
	\subcaption{$\rho$}
\end{subfigure}
\begin{subfigure}[b]{0.3\textwidth}
	\begin{center}
		\includegraphics[width=1\textwidth]{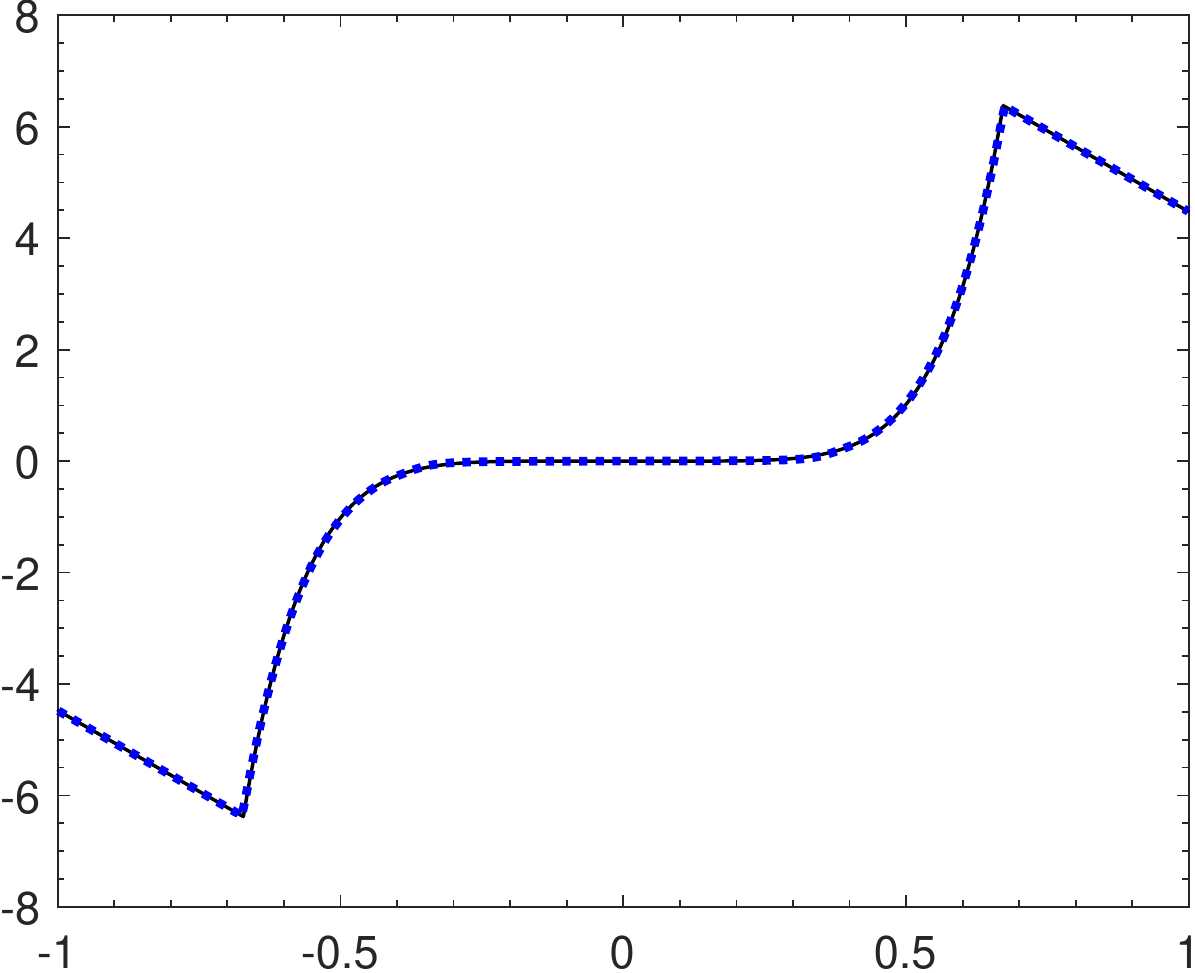}
	\end{center}
	\subcaption{$m$}
\end{subfigure}	
\begin{subfigure}[b]{0.3\textwidth}
	\begin{center}
		\includegraphics[width=1\textwidth]{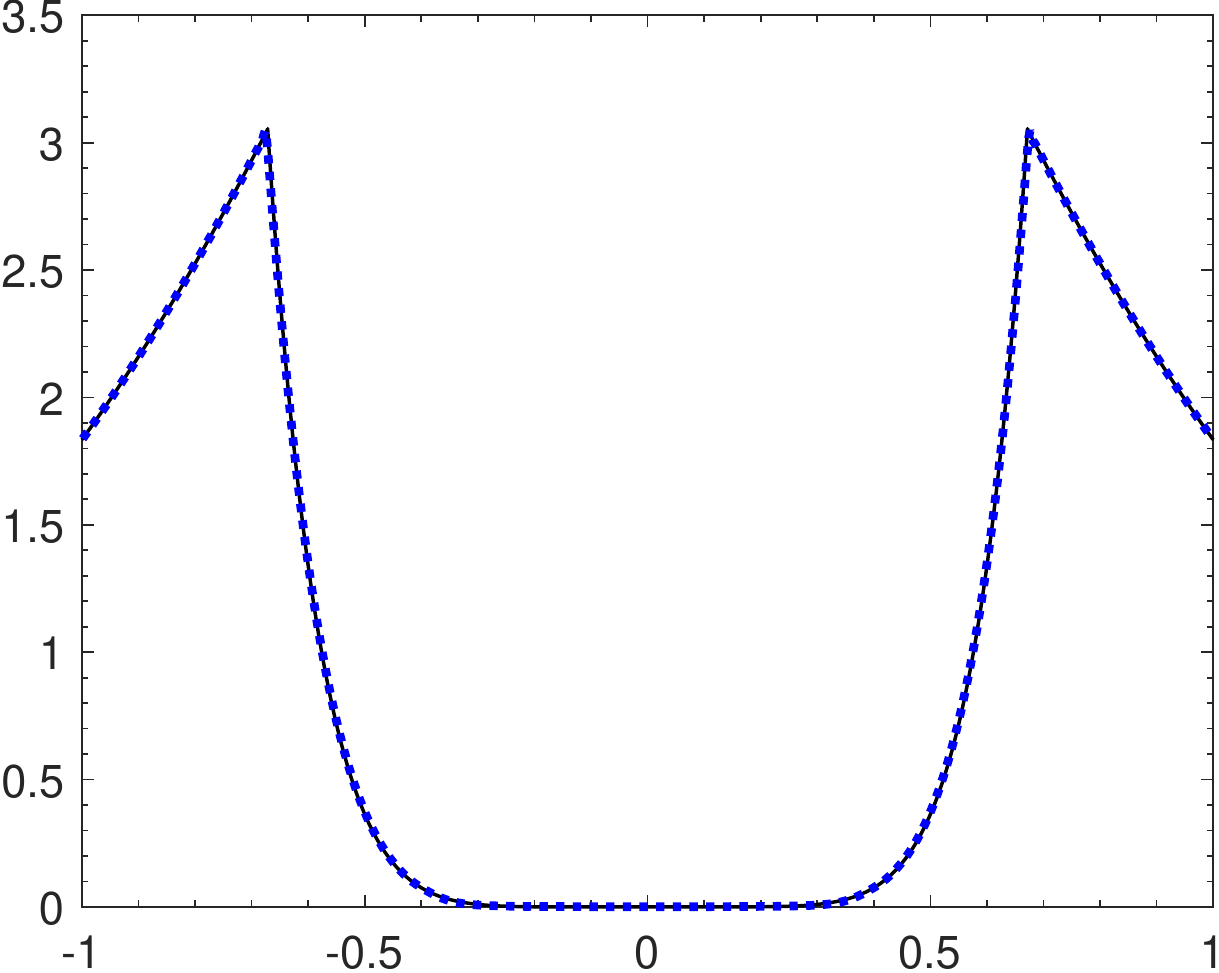}
	\end{center}
	\subcaption{$E$}
\end{subfigure}
\caption{\small
	Example 2: Density, momentum and energy for the rarefaction test 
	at $t=0.6$ obtained by the positivity-preserving well-balanced DG scheme with 800 cells 
	(dotted lines) and 128000 cells (solid lines). 
}\label{fig:ex4} 
\end{figure}

\subsection{Example 3: Leblanc problem in linear gravitational field}

In this test, we consider an extension of the standard 1D Leblanc shock tube problem to the gravitational case with $\phi(x)=gx$ and $g=1$. The initial condition of this problem is given by 
$$
(\rho, u, p) (x,0)=\begin{cases}
(2,~0,~10^9), & \quad x<5,\\
(10^{-3},~0,~1), & \quad x>5.
\end{cases}
$$ 
This problem is highly challenging due to the presence of the strong jumps in the initial density and pressure. 
The computational domain is taken as $[0,10]$ with reflection boundary conditions at $x=0$ and $x=10$. 
To fully resolve the wave structure, a fine mesh is required for such
test. 
In the computations, the CFL number is taken as $0.15$. 
As the exact solution contains strong discontinuities, the WENO limiter \cite{Qiu2005} 
is implemented right before the positivity-preserving limiting procedure with the aid of the local
characteristic decomposition within the adaptively detected ``trouble'' cells. 
Figure~\ref{fig:ex5} displays 
our numerical results at $t=0.25$, obtained
by the third-order positivity-preserving well-balanced DG scheme, 
on a mesh with $1600$ cells, compared with reference solutions obtained with much refined $6400$ cells. 
We see that the strong discontinuities are captured by the proposed method with high resolution.

\begin{figure}[htbp]
\centering
\begin{subfigure}[b]{0.32\textwidth}
	\begin{center}
		\includegraphics[width=1\textwidth]{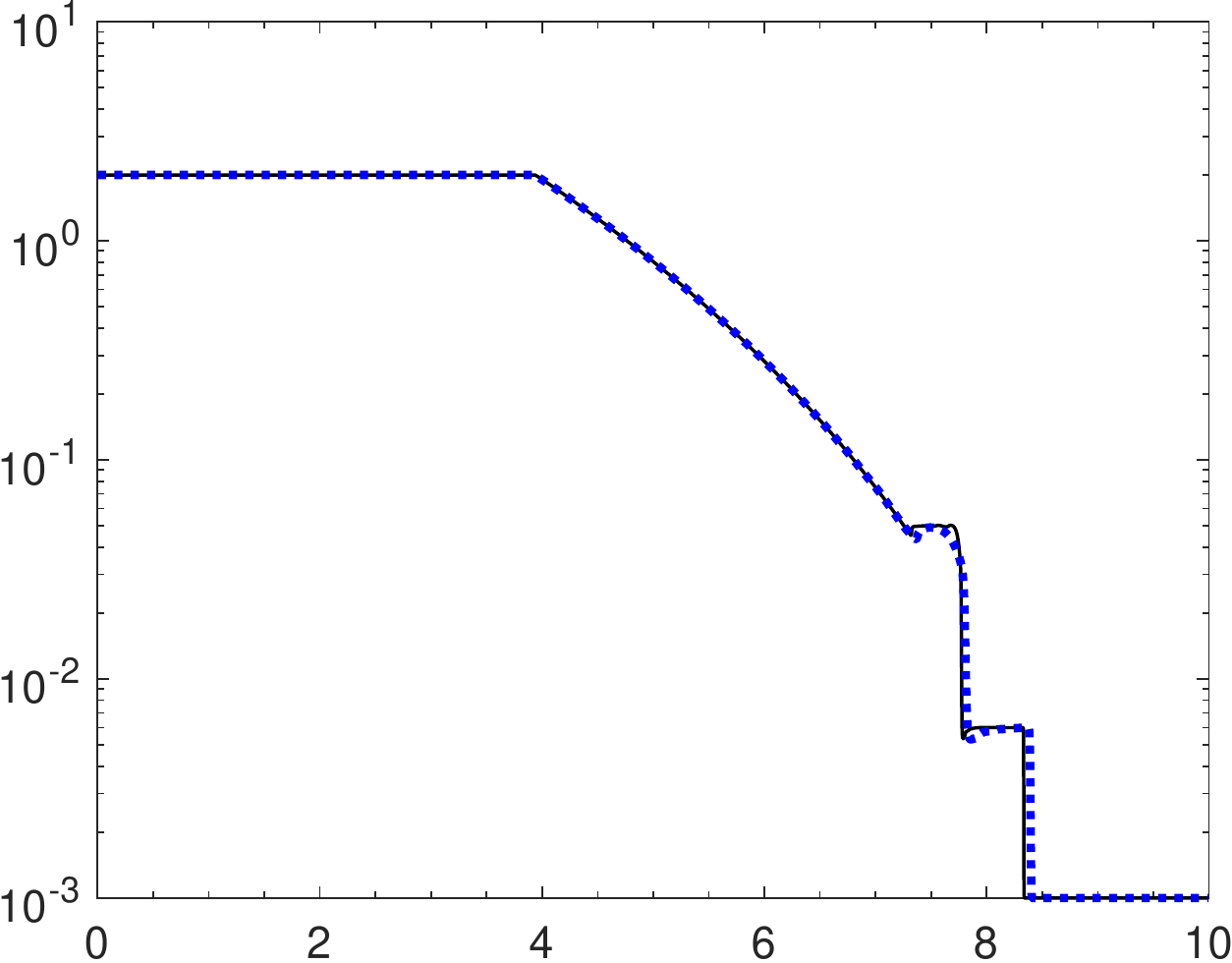}
	\end{center}
	\subcaption{$\rho$}
\end{subfigure}
\begin{subfigure}[b]{0.32\textwidth}
	\begin{center}
		\includegraphics[width=1\textwidth]{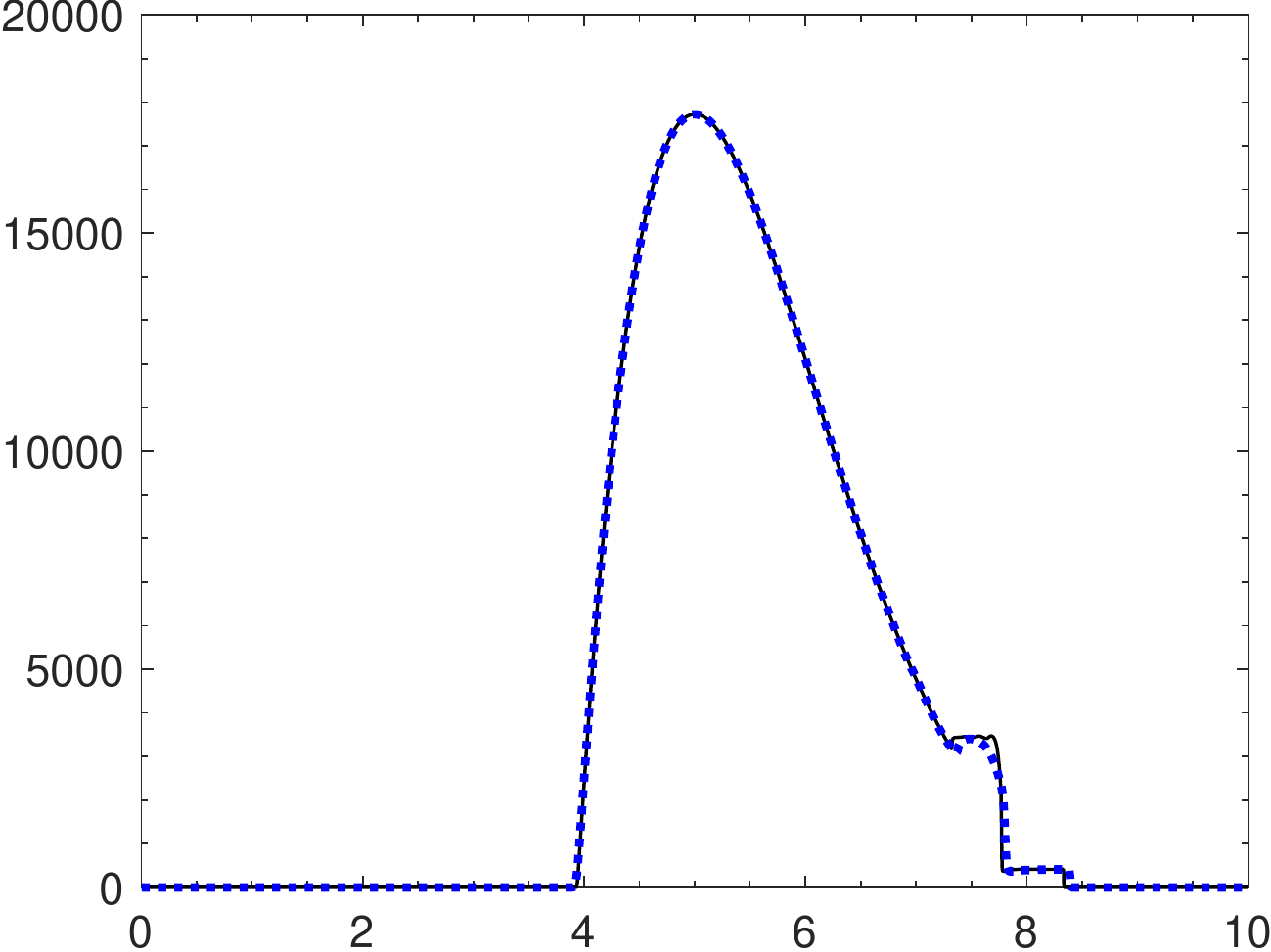}
	\end{center}
	\subcaption{$m$}
\end{subfigure}	
\begin{subfigure}[b]{0.32\textwidth}
	\begin{center}
		\includegraphics[width=1\textwidth]{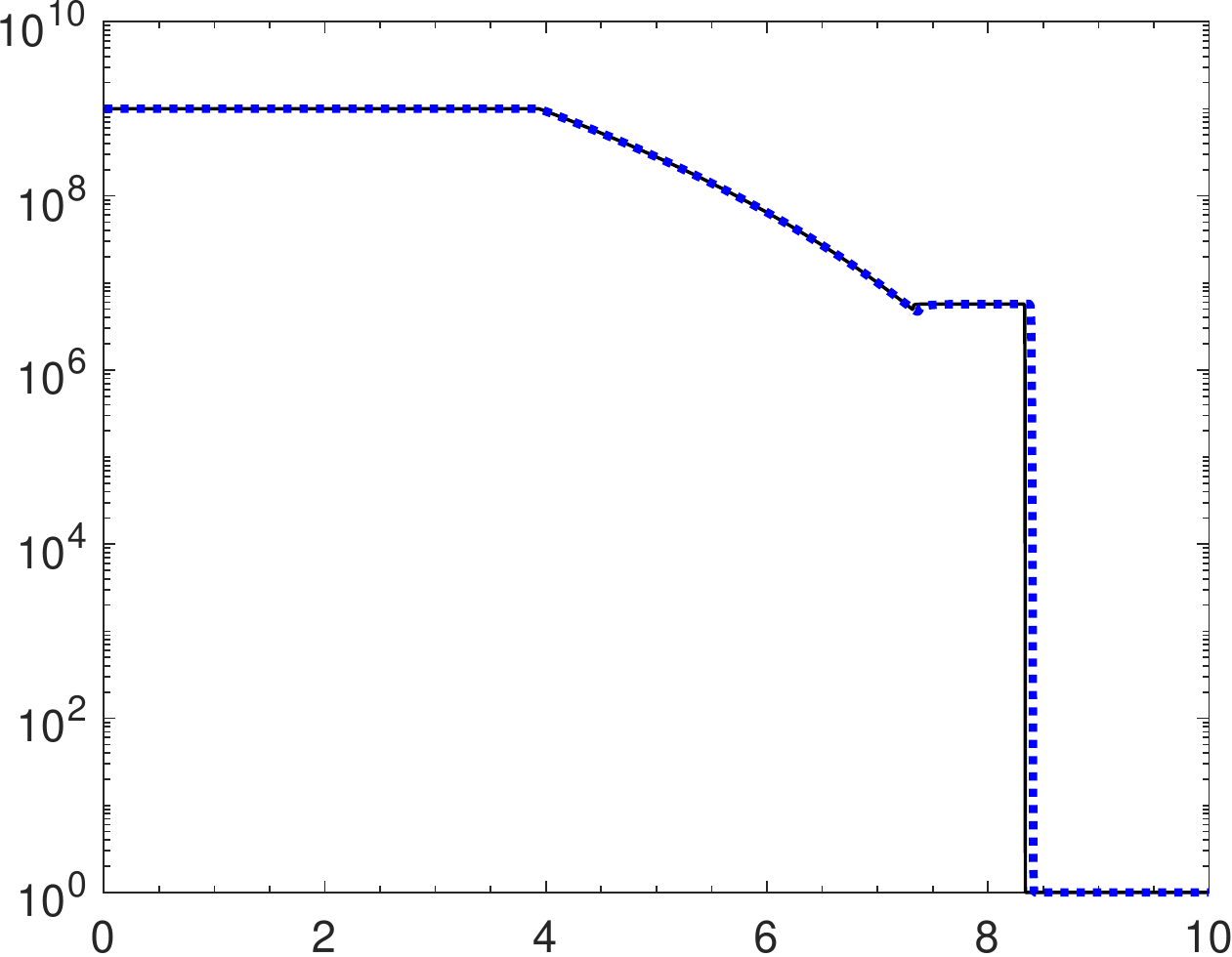}
	\end{center}
	\subcaption{$p$}
\end{subfigure}	
\caption{\small
	Example 3: The log plot of density (left), the velocity (middle) and the log plot of pressure (right) for the extended Leblanc problem
	at $t=0.00004$ obtained by the positivity-preserving well-balanced DG scheme with 1600 cells 
	(dotted lines) and 6400 cells (solid lines), respectively.
}\label{fig:ex5} 
\end{figure}

\subsection{Example 4: Two-dimensional accuracy test} \label{sec:2Dsmooth}  

In this example, we examine the accuracy of the proposed schemes on a two-dimensional smooth problem \cite{XS2013} with a linear gravitational field $\phi_x = \phi_y =1$ in the domain $\Omega = [0,2]^2$. 
The exact solution takes the following form 
\begin{align*}
& \rho(x,y,t) = 1 + 0.2 \sin ( \pi ( x+y - t ( u_0 + v_0 )  ) ),
\quad  {\bf u}(x,y,t) = (u_0,v_0),
\\
& p(x,y,t) = p_0 + t( u_0 + v_0 ) - x- y + 0.2 \cos( \pi ( x+y-t(u_0+v_0) ) ) / \pi,
\end{align*} 
where the parameters are taken as $u_0=v_0=1$ and $p_0=4.5$. 
The adiabatic index $\gamma$ is taken as $5/3$. 
The domain $\Omega$ is divided into $M \times M$ uniform cells, and 
the boundary condition is specified by the exact solution on $\partial \Omega$. 
Table~\ref{tab:acc2D} lists the $l^1$ errors at $t=0.1$ and the corresponding orders obtained by the proposed third-order well-balanced DG scheme at different grid resolutions. 
The results show that the expected convergence order is achieved.  
Our modification of the numerical flux and the non-trivial source term approximation  
do not affect the accuracy of the DG methods.

\begin{table}[htbp]
\centering
\caption{\small Example 4: 
	$l^1$-errors at $t=0.1$ in $\rho,{\bf m}=(m_1,m_2),E$, and corresponding convergence rates for
	the third-order well-balanced DG method at
	different grid resolutions.
} \label{tab:acc2D}
\begin{tabular}{c|c|c|c|c|c|c|c|c}
	\hline
	\multirow{2}{18pt}{Mesh}
	&\multicolumn{2}{c|}{$\rho$}&\multicolumn{2}{c|}{$m_1$}
	&\multicolumn{2}{c|}{$m_2$} &\multicolumn{2}{c}{$E$} 
	\\
	\cline{2-9}
	& error& order & error & order & error & order & error & order \\
	\hline
	$8 \times 8$& 4.20e-3& --          &4.29e-3& --   &  4.29e-3 & -- & 4.67e-3  &-- \\
	$16\times 16$& 5.25e-4& 3.00   & 5.42e-4&    2.98 & 5.42e-4 & 2.98 & 5.76e-4 & 3.02\\
	$32\times 32$& 6.62e-5& 2.99   & 6.86e-5&  2.98 & 6.86e-5 & 2.98  & 7.28e-5 & 2.98 \\
	$64\times 64$& 8.31e-6 &  2.99    & 8.61e-6&  2.99 & 8.61e-6 & 2.99 & 9.17e-6& 2.99 \\
	$128\times 128$& 1.04e-6 &  3.00 & 1.08e-6&    3.00 & 1.08e-6 & 3.00 & 1.15e-6 & 3.00 \\
	$256\times 256$&  1.30e-7 &  3.00  &1.35e-7&  3.00 & 1.35e-7 & 3.00 & 1.44e-7 & 3.00 \\
	$512\times 512$&  1.63e-8 &  3.00  &1.69e-8&  3.00 & 1.69e-8 & 3.00 & 1.80e-8 & 3.00 \\		
	\hline
\end{tabular}
\end{table}

\subsection{Example 5: Two-dimensional isothermal equilibrium}\label{sec:ex6}

This example is used to demonstrate the well-balanced property and the capability of the proposed methods in capturing the small perturbation of a 2D isothermal equilibrium solution \cite{XS2013}. We consider a linear gravitational field with $\phi_x=\phi_y = g$ and take $g=1$. The computational domain is taken as the unit square  $[0,1]^2$. 
The isothermal
equilibrium state under consideration takes the following form
\begin{equation}\label{eq:Ex7WB}
\begin{aligned}
\rho(x,y) = \rho_0 \exp \left( - \frac{ \rho_0 g }{p_0} (x+y) \right), \quad {\bf u}(x,y)={\bf 0}, \quad 
p(x,y) = p_0 \exp \left( - \frac{ \rho_0 g }{p_0} (x+y) \right),
\end{aligned} 
\end{equation}
with the parameters $\rho_0 = 1.21$ and $p_0=1$. 
%

We first validate the well-balanced property of the proposed DG method. To this end, we take the initial data as the equilibrium solution \eqref{eq:Ex7WB} and conduct the simulation up to $t=1$ on the three different uniform meshes. The $l^1$ errors in $\rho$, 
${\bf m}=(m_1,m_2)$ and $E$ are shown in Table \ref{tab:Ex7}. One can clearly see that the steady state solution is indeed maintained up to rounding error, which confirms the well-balancedness of the proposed DG method.  

We then investigate the capability of the proposed well-balanced method in capturing small perturbations of the hydrostatic equilibrium. Initially, a small Gaussian hump perturbation centered at $(0.3,0.3)$ is imposed in the pressure to the equilibrium  solution \eqref{eq:Ex7WB} as follows: 
$$
p(x,y,0) = p_0 \exp \left( - \frac{ \rho_0 g }{p_0} (x+y) \right) + \eta \exp \left( - \frac{100 \rho_0 g}{p_0} 
\big( (x-0.3)^2 + (y-0.3)^2  \big) \right),
$$
where $\eta$ is set as $0.001$. 
We evolve the solution up to $t=0.15$ on a mesh of $100 \times 100$ uniform cells with transmissive boundary conditions. 
The contour plots of the pressure perturbation and density perturbation are displayed in Figure~\ref{fig:ex7b}, obtained via the well-balanced and the non-WB DG schemes, respectively. It is observed that the non-WB DG method cannot capture such small perturbation well on the relatively coarse mesh, while the well-balanced one can resolve it accurately. 
\begin{table}[htbp]
\centering
\caption{\small 
	Example 5: $l^1$-errors for the steady state solution in Section~\ref{sec:ex6} at
	different grid resolutions.
}\label{tab:Ex7}
\begin{tabular}{ccccc}
	\hline
	Mesh & errors in $\rho$ & errors in $m_1$ & errors in $m_2$  & errors in $E$ \\
	\hline
	$50\times 50$ & 2.0615e-15 & 1.8301e-15 & 1.8527e-15 & 7.3921e-15
	\\
	\hline
	$100\times 100$ & 4.5131e-15 & 3.7141e-15 & 3.7649e-15 & 1.5167e-14
	\\
	\hline 
	$200\times 200$ & 9.6832e-15 & 7.3142e-15 & 7.3067e-15 & 3.0940e-14
	\\
	\hline
\end{tabular}
\end{table}

\begin{figure}[htbp]
\centering
\begin{subfigure}[b]{0.48\textwidth}
	\begin{center}
		\includegraphics[width=0.95\textwidth]{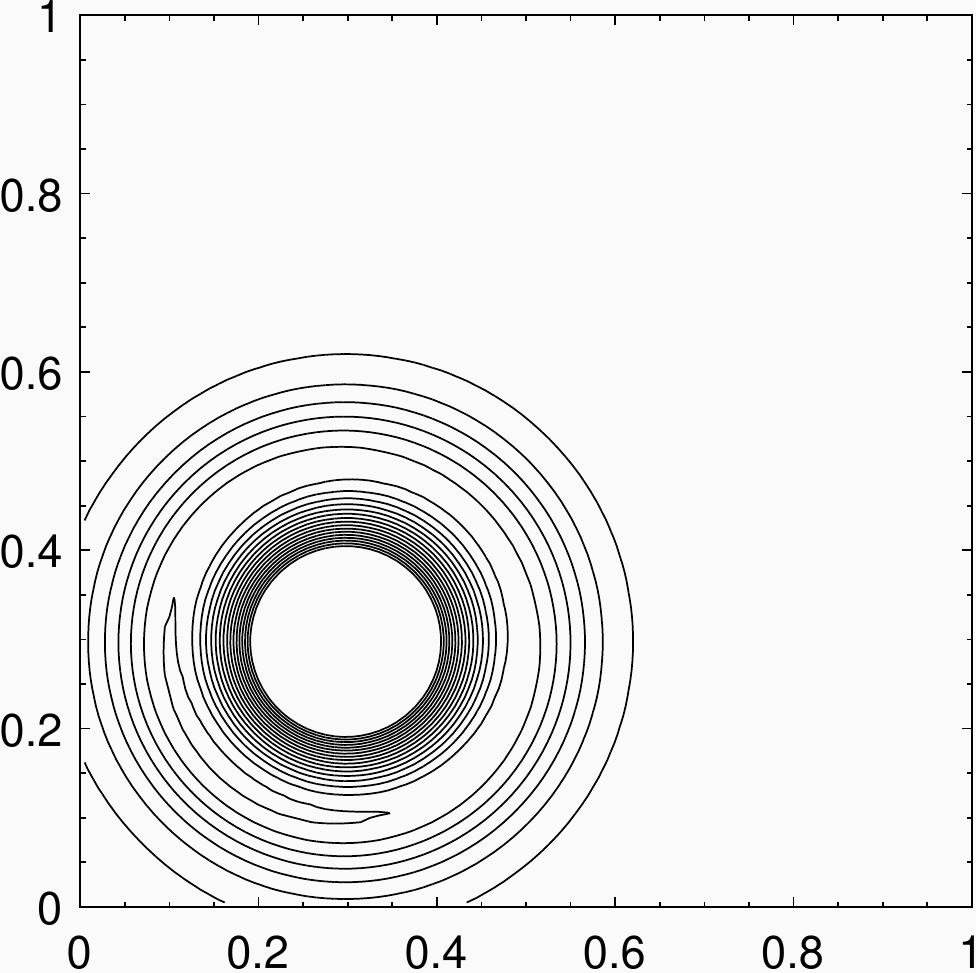}
	\end{center}
	\subcaption{WB method: pressure perturbation}
\end{subfigure}
\begin{subfigure}[b]{0.48\textwidth}
	\begin{center}
		\includegraphics[width=0.95\textwidth]{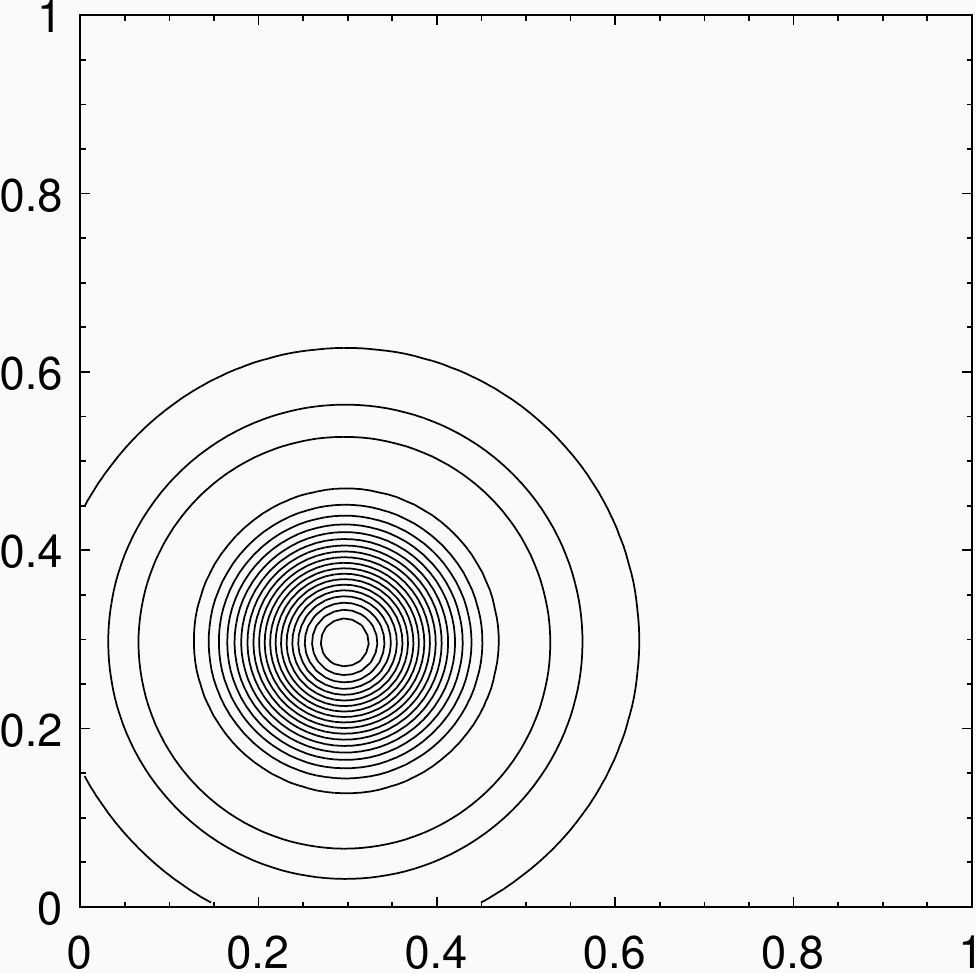}
	\end{center}
	\subcaption{WB method: density perturbation}
\end{subfigure}	
\\ \vspace{3mm}		
\begin{subfigure}[b]{0.48\textwidth}
	\begin{center}
		\includegraphics[width=0.95\textwidth]{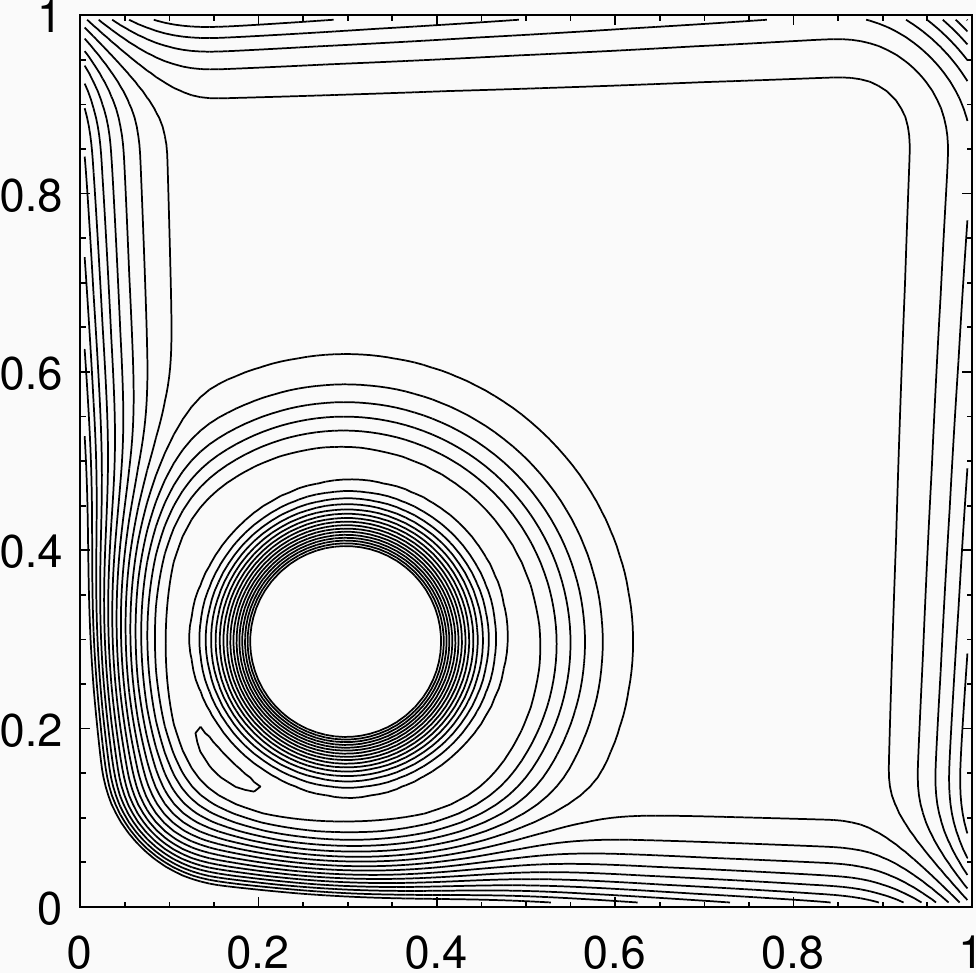}
	\end{center}
	\subcaption{Non-WB method: pressure perturbation}
\end{subfigure}
\begin{subfigure}[b]{0.48\textwidth}
	\begin{center}
		\includegraphics[width=0.95\textwidth]{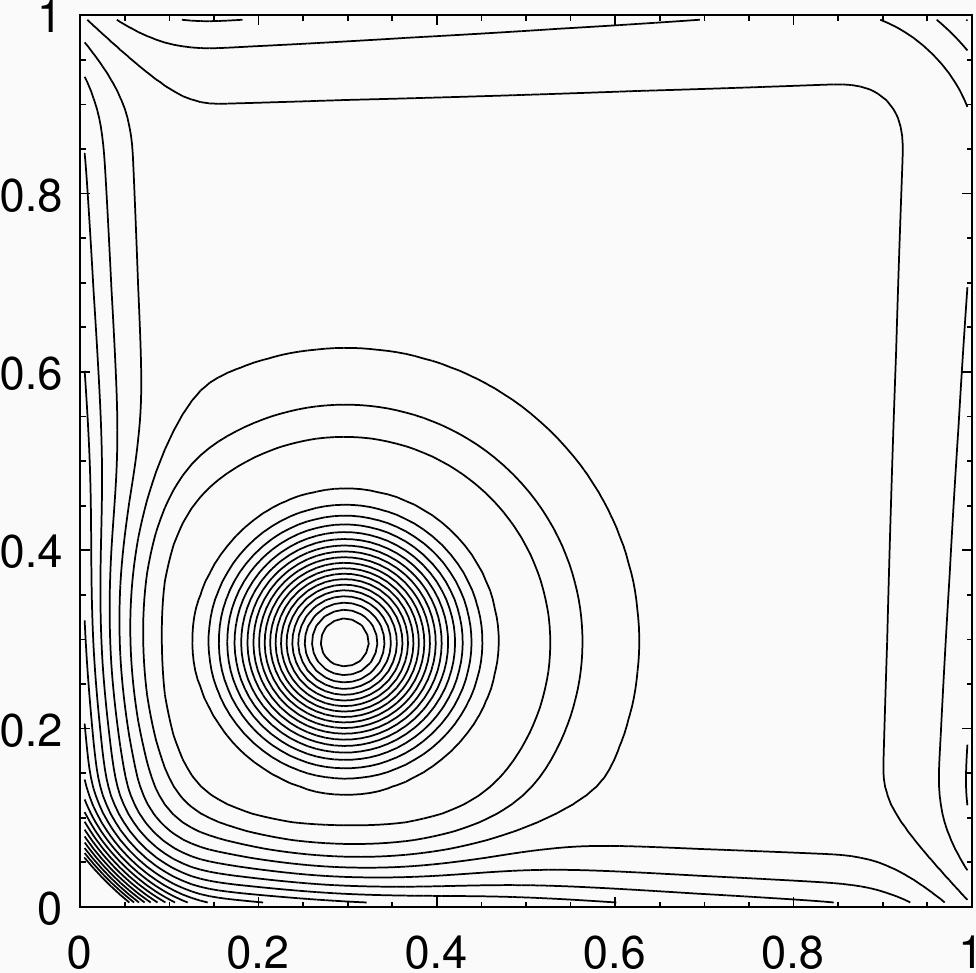}
	\end{center}
	\subcaption{Non-WB method: density perturbation}
\end{subfigure}	
\caption{\small
	Example 5: The contour plots of the pressure perturbation and the density perturbation of the hydrostatic solution at time $t=0.15$ obtained by the third-order WB and non-WB DG scheme with $100 \times 100$ cells. 20 equally spaced contour lines are displayed: from $-0.0003$ to $0.0003$ for pressure perturbation; from $-0.001$ to $0.0002$ 
	for density perturbation. 
}\label{fig:ex7b} 
\end{figure}


\subsection{Example 6: Two-dimensional polytropic equilibrium} \label{sec:ex81}

In this example, we verify the performance of the proposed methods on a two-dimensional polytropic test case \cite{KM2014} arising from astrophysics. We consider a static adiabatic gaseous sphere, which is held together by self-gravitation and can be constructed from the hydrostatic equilibrium 
$
\frac{ {\rm d} p} { {\rm d} r } = -\rho \frac{ {\rm d} \phi } { {\rm d} r }, 
$
with $\gamma =2$. One equilibrium solution of this model is given by 
\begin{equation}\label{eq:Ex8SS}
\rho ( r ) = \rho_c \frac{ \sin ( \alpha r ) } { \alpha r }, \quad u(r) = 0, \quad v(r) = 0, \quad p(r)= K_0 \rho(r) ^2, 
\end{equation}
under the gravitational field
\begin{equation}\label{eq:Ex7phi}
\phi ( r ) = - 2K_0 \rho_c \frac{ \sin( \alpha r ) } { \alpha r },
\end{equation}
where $\alpha = \sqrt{ {2 \pi g} /{ K_0} }$ with $K_0=g=\rho_c = 1$, and $r := \sqrt{ x^2 + y^2 }$ denotes the radial variable. 
The computational domain is taken as $[-0.5,0.5]^2$. 

%
We first demonstrate the well-balanced property of our DG scheme. The initial condition is taken as the equilibrium solution \eqref{eq:Ex8SS}, which should be exactly preserved. 
The computation is performed until $t=14.8$ on three different uniform meshes. 
The $l^1$ errors in the numerical solutions are presented in Table \ref{tab:Ex81}. 
It shows that the steady state is preserved up to rounding error, as expected from the well-balancedness of the proposed method.

\begin{table}[htbp]
\centering
\caption{\small 
	Example 6: $l^1$-errors for the steady state solution in Section~\ref{sec:ex81} 
	at
	different grid resolutions.
}\label{tab:Ex81}
\begin{tabular}{ccccc}
	\hline
	Mesh & errors in $\rho$ & errors in $m_1$ & errors in $m_2$  & errors in $E$ \\
	\hline
	$50\times 50$ & 3.9099e-14 & 1.0132e-13 & 1.0312e-13 & 8.5883e-15
	\\
	\hline
	$100\times 100$ & 7.4068e-14 & 1.8519e-13 & 1.8328e-13 & 1.7675e-14
	\\
	\hline 
	$200\times 200$ & 1.4237e-13 & 3.3540e-13 & 3.3567e-13 & 3.5853e-14
	\\
	\hline
\end{tabular}
\end{table}

%

We now impose a small perturbation to the initial pressure state 
$$
p(x,y,0) = K_0 \rho(r)^2 + \eta \exp( - 100 r^2 ), 
$$
and then compute the solution 
up to $t=0.2$ on a mesh of $200 \times 200$ uniform cells with transmissive boundary conditions. 
Figure~\ref{fig:ex8WB} shows 
the contour plots of the pressure perturbation and the velocity magnitude $\|{\bf u}\|$, obtained by using our well-balanced DG method and the non-WB DG method, respectively. 
We observe that the well-balanced DG scheme captures the small perturbation very well and preserve the axial symmetry, while the non-WB DG method cannot accurately resolve the small perturbation and maintain the axial symmetry on the relatively coarse mesh.

\begin{figure}[htbp]
\centering
\begin{subfigure}[b]{0.48\textwidth}
	\begin{center}
		\includegraphics[width=1\textwidth]{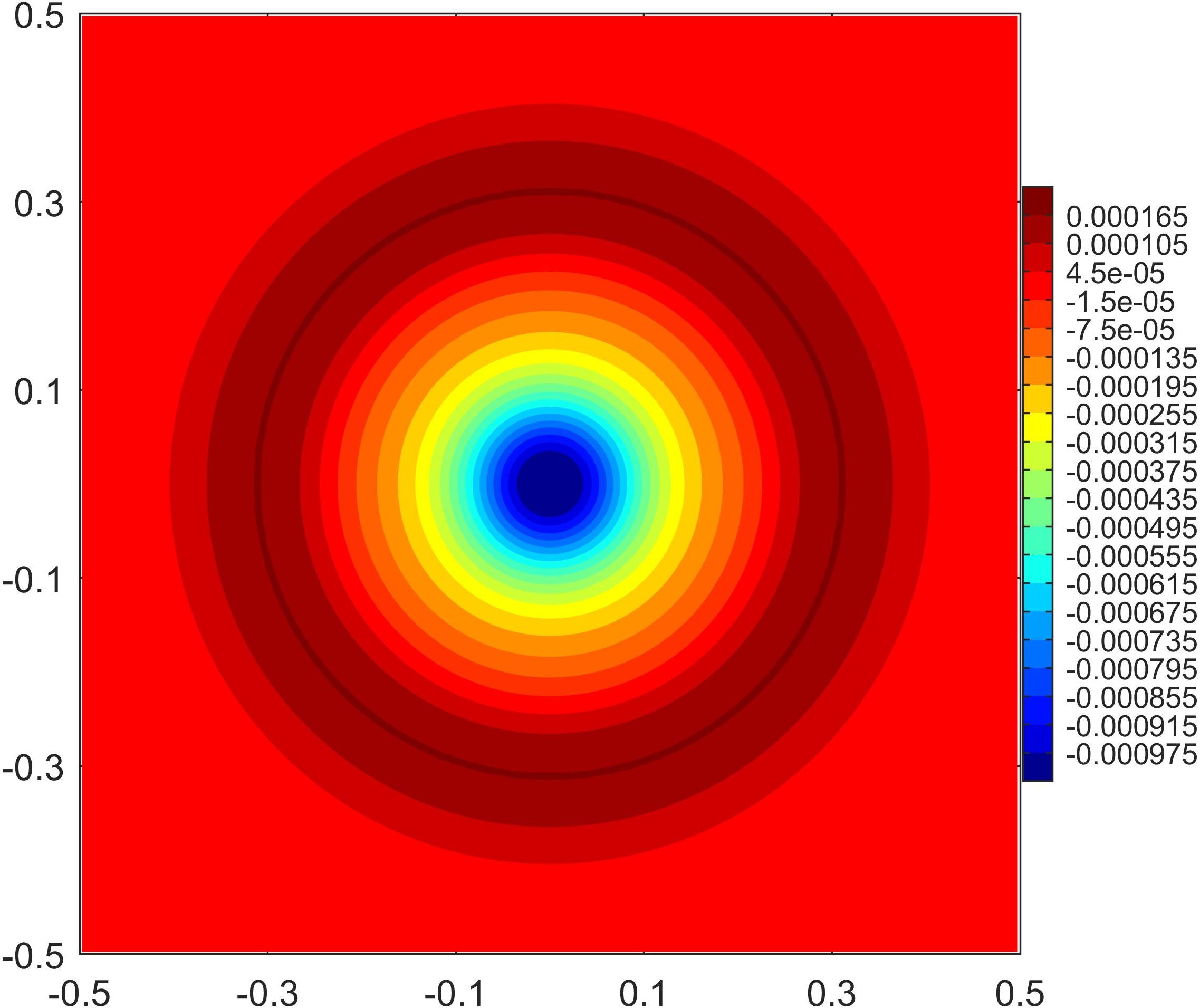}
	\end{center}
	\subcaption{WB method: pressure perturbation}
\end{subfigure}
\begin{subfigure}[b]{0.48\textwidth}
	\begin{center}
		\includegraphics[width=1\textwidth]{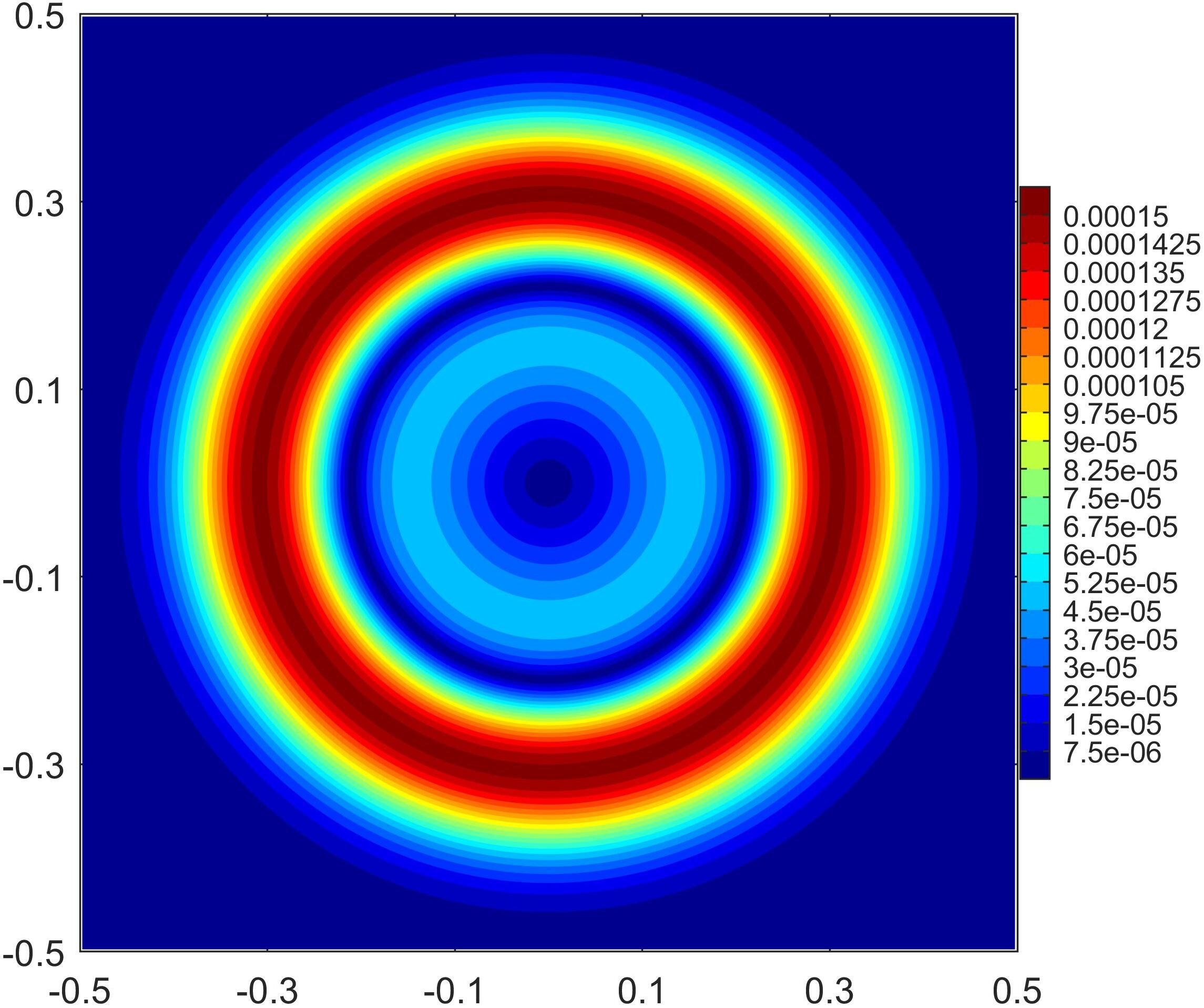}
	\end{center}
	\subcaption{WB method: velocity magnitude}
\end{subfigure}	
\\ \vspace{3mm}	
\begin{subfigure}[b]{0.48\textwidth}
	\begin{center}
		\includegraphics[width=1\textwidth]{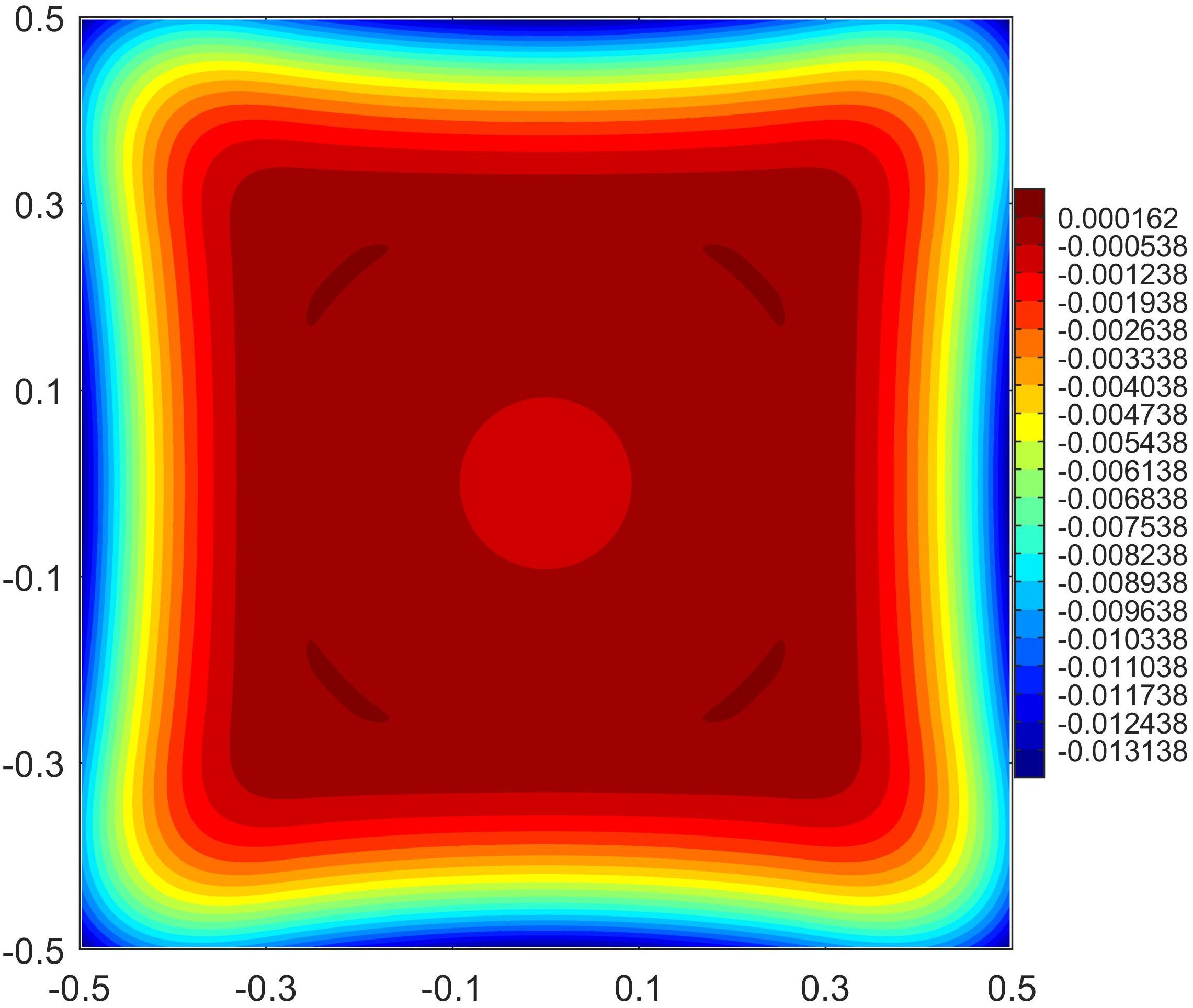}
	\end{center}
	\subcaption{Non-WB method: pressure perturbation}
\end{subfigure}
\begin{subfigure}[b]{0.48\textwidth}
	\begin{center}
		\includegraphics[width=1\textwidth]{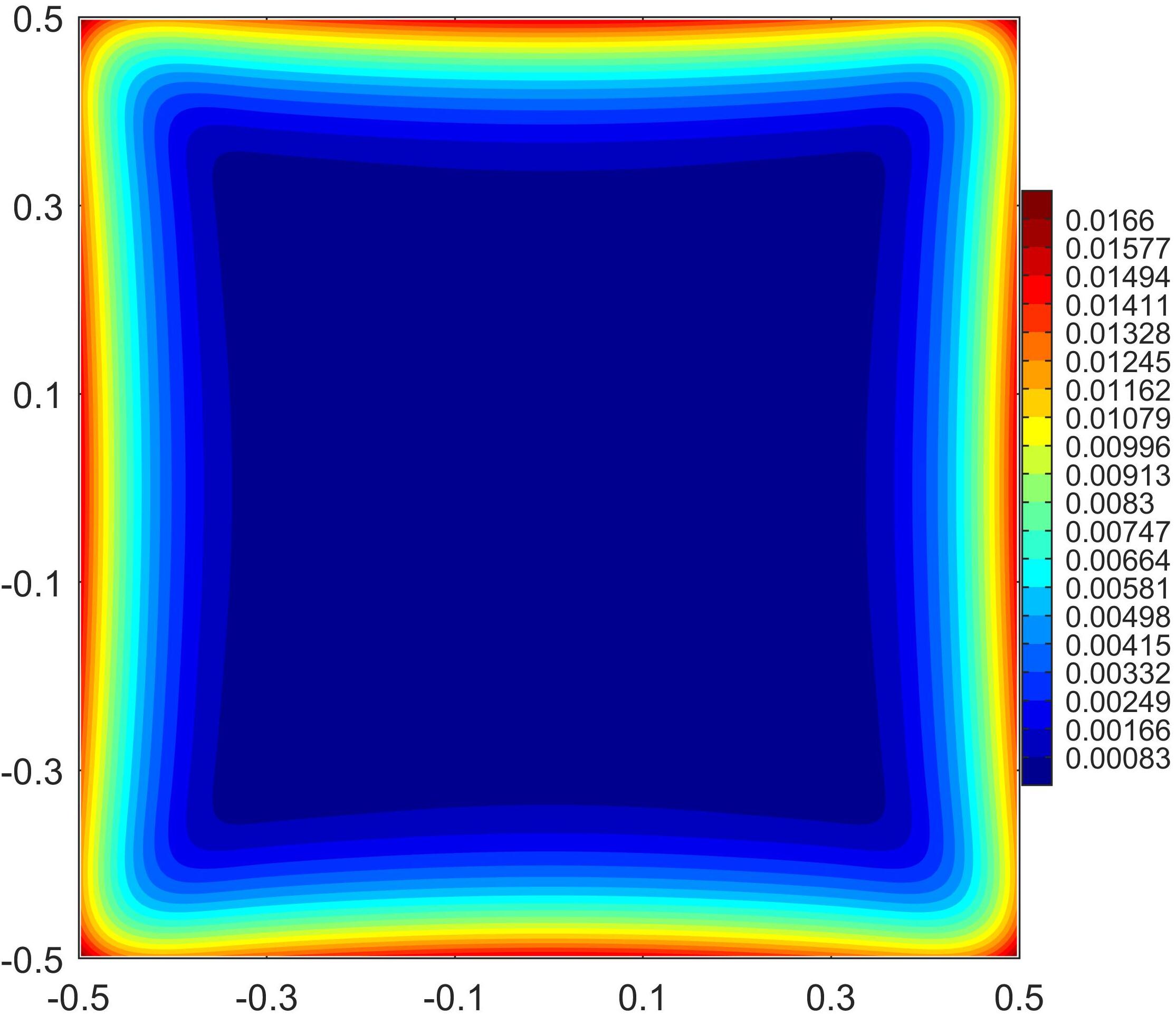}
	\end{center}
	\subcaption{Non-WB method: velocity magnitude}
\end{subfigure}	
\caption{\small
	Example 6: The contour plots of the pressure perturbation and the velocity magnitude at time $t=0.2$ obtained by using our well-balanced and non-WB DG schemes on $200 \times 200$ cells. 
}\label{fig:ex8WB} 
\end{figure}


\subsection{Example 7: Two-dimensional rarefaction test with low density and pressure}

This example is used to demonstrate the positivity-preserving property of the proposed DG method. 
The setup of this test is analogous to the one-dimensional 1-2-3 rarefaction test in \cite{EINFELDT1991273} and the two-dimensional 
rarefaction test in \cite{Thomann2019}.  
The initial condition is given by 
\begin{align*}
& \rho (x,y,0) = \exp( -\phi(x,y)/0.4  ), \qquad \quad p (x,y,0) = 0.4 \exp( -\phi(x,y)/0.4  ), 
\\
&u(x,y,0)= \begin{cases}
-2, & \quad x<0.5,\\
2, & \quad x>0.5,
\end{cases} \qquad \quad v(x,y,0)=0,
\end{align*}
with a quadratic gravitational potential $\phi(x,y) = \frac12 \left( (x-0.5)^2 + (y-0.5)^2 \right)$. 
The computational domain $\Omega = [0,1]^2$ is divided into $100\times 100$ uniform cells with transmissive boundary conditions. 
Figure~\ref{fig:ex9} presents the numerical solutions obtained by using the proposed positivity-preserving well-balanced DG scheme. 
The CFL number is set as $0.15$ in the computation. 
We observe that the density, pressure and energy come close to zero but remain positive throughout the simulation. 
It is noticed that the DG code would break down in the first time step, if the positivity-preserving limiting
technique is not employed. 

\begin{figure}[htbp]
	\centering
	\begin{subfigure}[b]{0.3\textwidth}
	\begin{center}
		\includegraphics[width=0.99\textwidth]{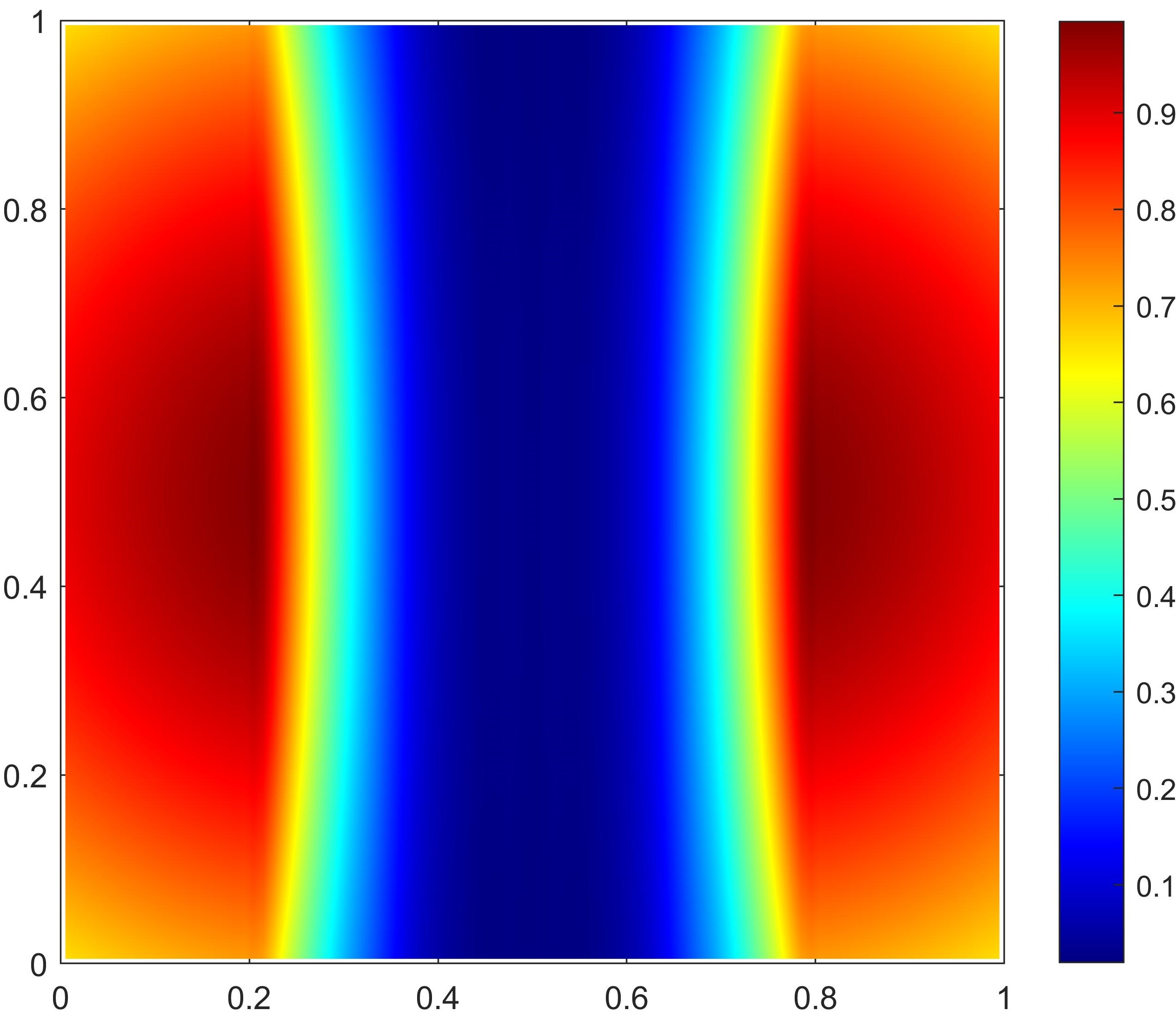}
	\end{center}
	\end{subfigure}
	\begin{subfigure}[b]{0.3\textwidth}
	\begin{center}
		\includegraphics[width=0.99\textwidth]{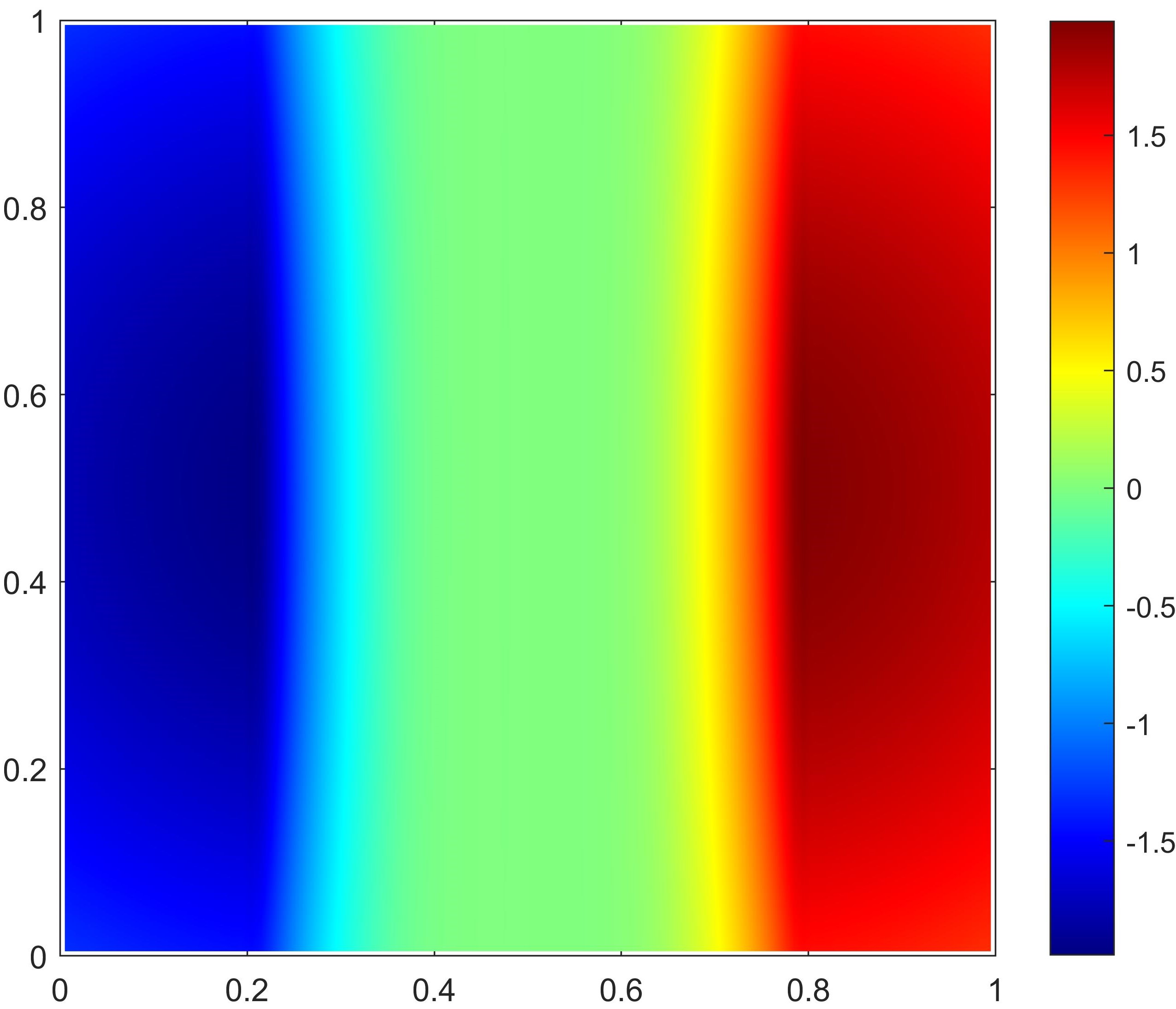}
	\end{center}
	\end{subfigure}	
	\begin{subfigure}[b]{0.3\textwidth}
	\begin{center}
		\includegraphics[width=0.99\textwidth]{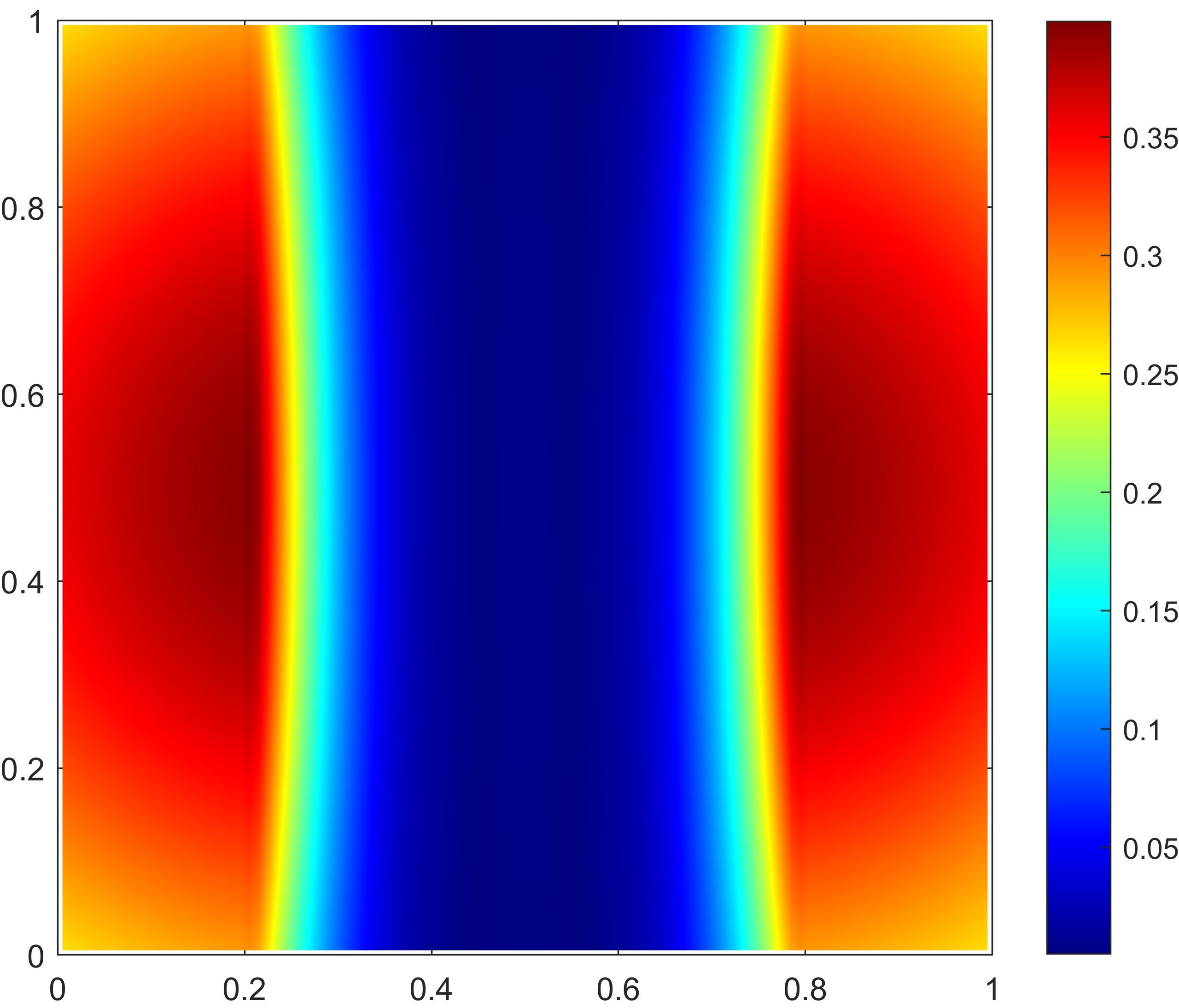}	
	\end{center}
	\end{subfigure}	
	\caption{\small
		Example 7: The schlieren images of $\rho$ (left),  $m_1$ (middle) and $p$ (right) at $t=0.1$  obtained by  
		using the positivity-preserving well-balanced DG scheme with $100 \times 100$ cells. 
	}\label{fig:ex9} 
\end{figure}


\subsection{Example 8: Two-dimensional blast problem}

To further verify the positivity-preserving property and the capability of the proposed DG method in resolving strong discontinuities, we consider a two-dimensional blast problem under the gravitational field \eqref{eq:Ex7phi}. 
The initial data is obtained by adding a huge jump to the pressure term of the equilibrium \eqref{eq:Ex8SS}, and the initial pressure is 
$$
p(x,y,0) = K_0 \rho(r)^2 + \begin{cases}
100, & \quad r<0.1,
\\
0, & \quad r\ge 0.1.
\end{cases}
$$
We set the parameters $K_0=g=1$ and $\gamma=2$ as those in Example 6, and $\rho_c = 0.01$ so that low pressure and low density appear in the solution. This, along with the presence the strong discontinuities, make this test challenging.   

Figure~\ref{fig:ex10WB} displays the contour plots of $\rho$ and $\log(p)$ at
$t = 0.005$ computed by the positivity-preserving third-order well-balanced DG method with $400 \times 400$ uniform cells. 
We also show the plot of $p$ along the line $y=x$, from which 
we can clearly observe a strong shock at $\sqrt{x^2+y^2} \approx 0.28$. 
In this test, the CFL number of $0.15$ is used, and the WENO limiter is implemented.
We observe that the discontinuities are well captured with high resolution, and the proposed DG method preserves 
the positivity of density and pressure as well as the axisymmetric structure of the solution.  
In this extreme test, it is necessary to use the positivity-preserving limiting technique, otherwise we observe that the DG
code would start to produce negative numerical pressure at $t \approx 0.00267$.

\begin{figure}[htbp]
\centering
\begin{subfigure}[b]{0.48\textwidth}
	\begin{center}
		\includegraphics[width=1\textwidth]{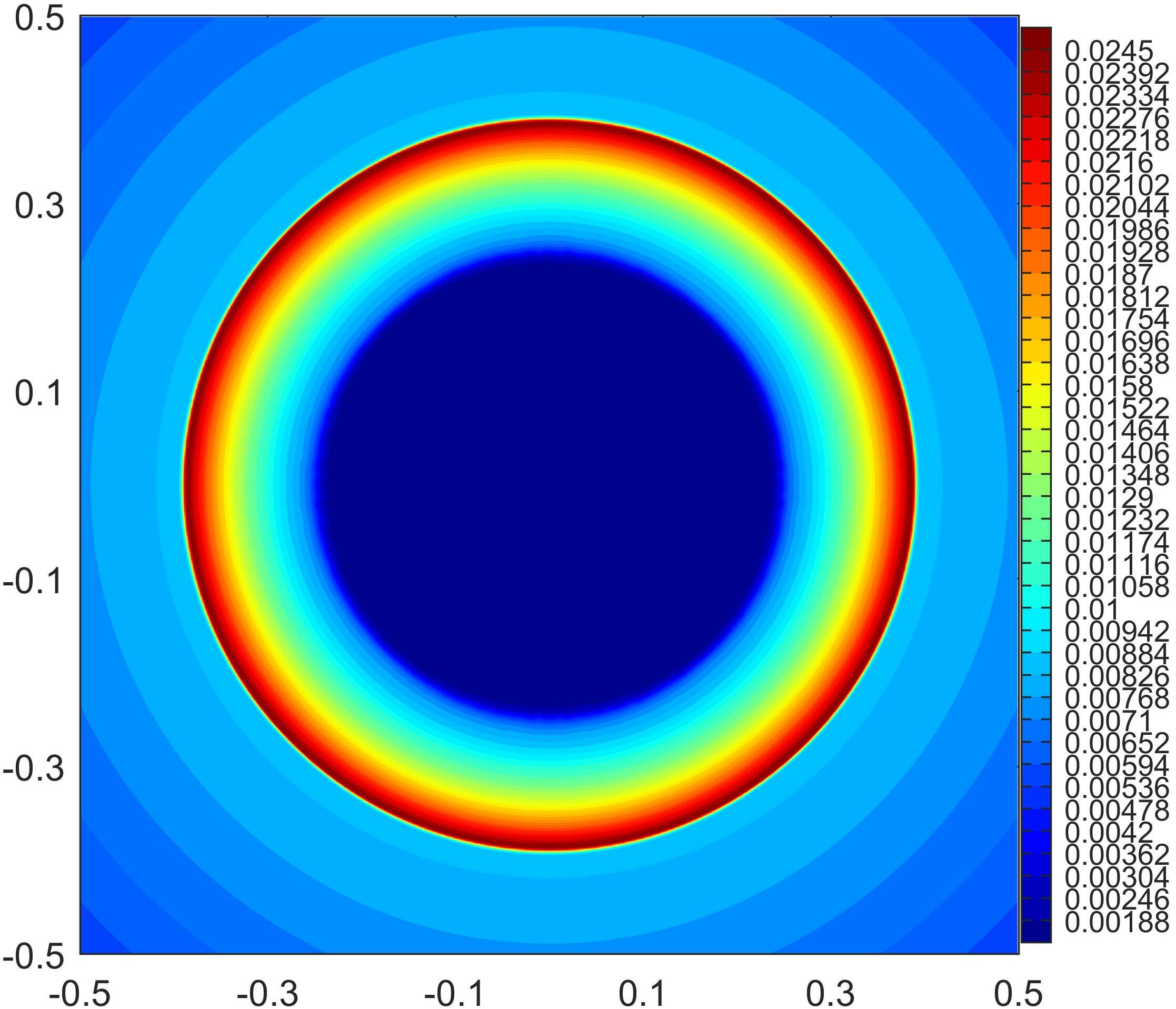}
	\end{center}
\end{subfigure}	
\begin{subfigure}[b]{0.48\textwidth}
	\begin{center}
		\includegraphics[width=1\textwidth]{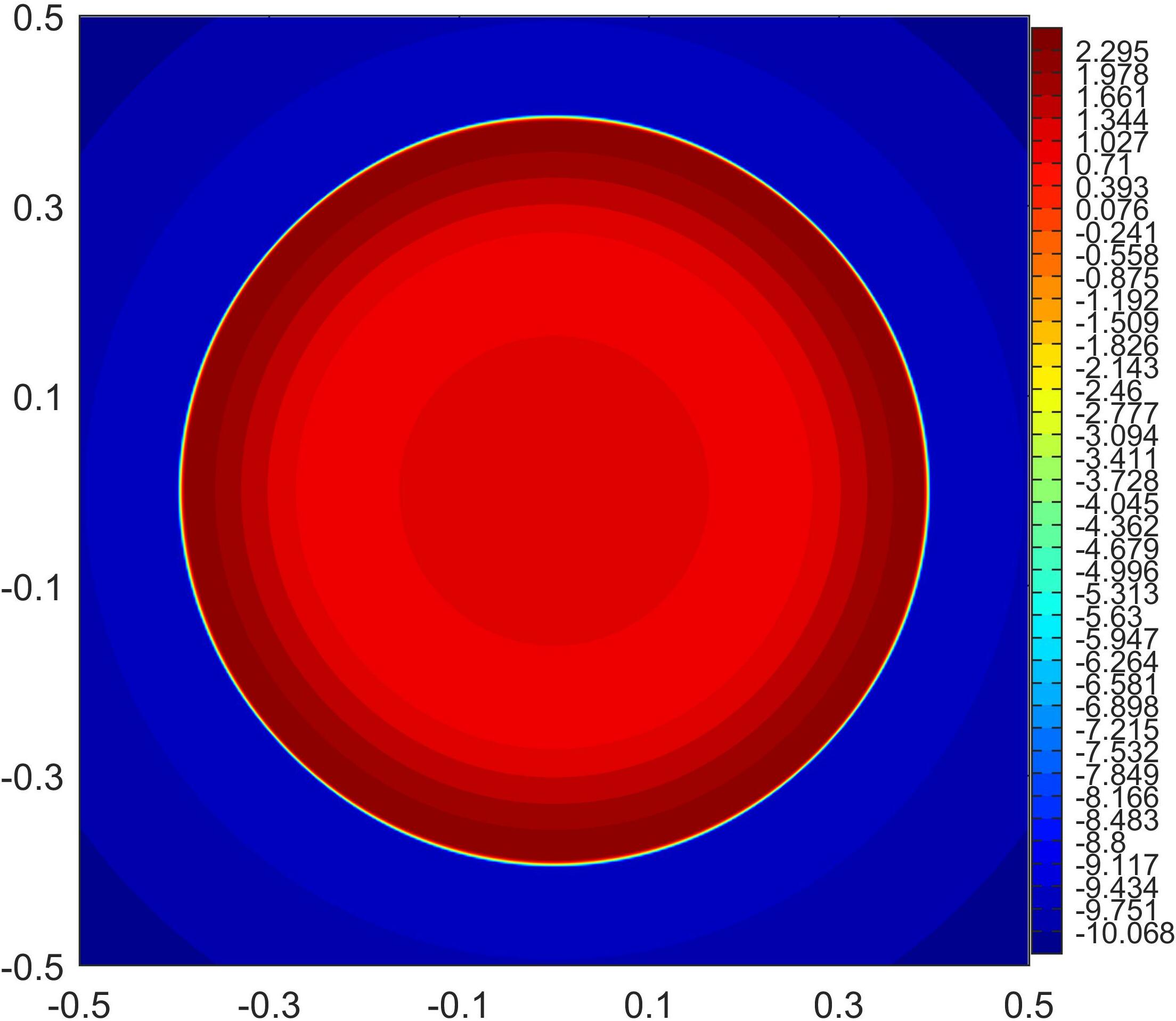}		
	\end{center}
\end{subfigure}	
\\ \vspace{3mm}	
\begin{subfigure}[b]{0.48\textwidth}
	\begin{center}
		\includegraphics[width=0.95\textwidth]{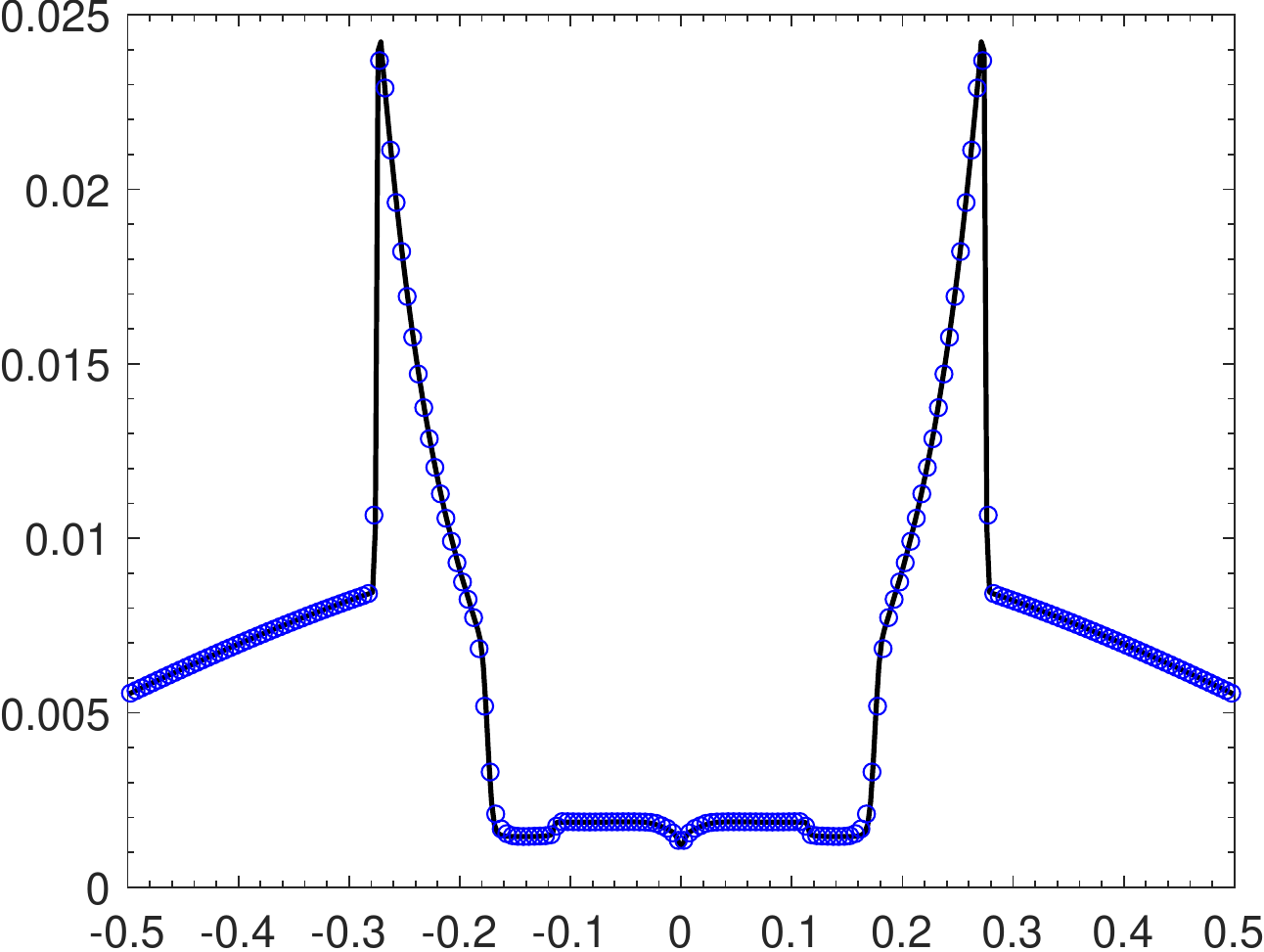}
	\end{center}
\end{subfigure}		
\begin{subfigure}[b]{0.48\textwidth}
	\begin{center}
		\includegraphics[width=0.95\textwidth]{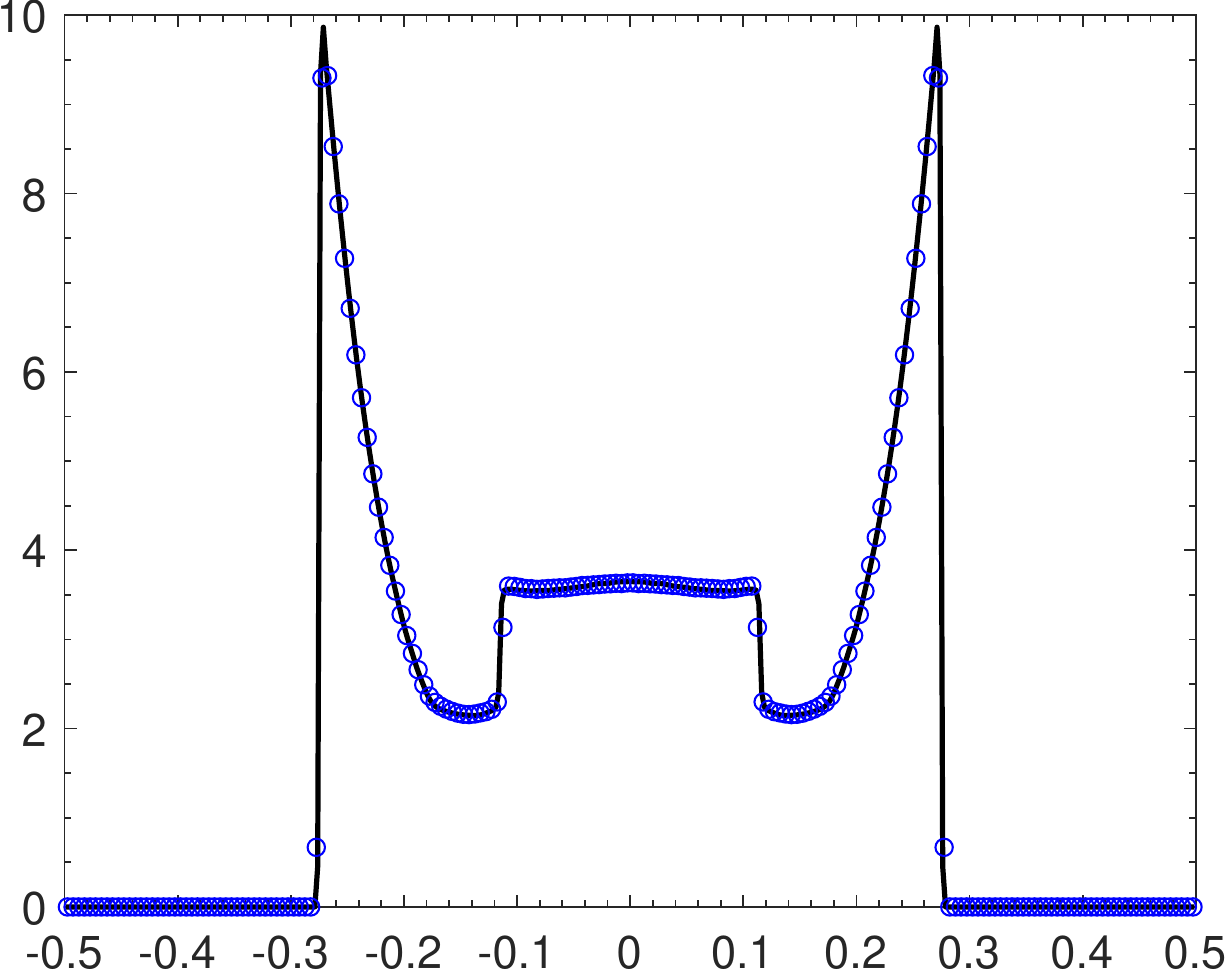}
	\end{center}
\end{subfigure}		
\caption{\small
	Example 8: The contour plots of the density $\rho$ (top-left) and the pressure logarithm ${\log}(p)$ (top-right) at $t=0.005$, and  
	the plots of $\rho$ (bottom-left) and $p$ (bottom-right) along the line $y=x$ within the scaled interval $[-0.5,0.5]$, obtained by  
	the positivity-preserving well-balanced DG scheme with $400 \times 400$ cells. 
}\label{fig:ex10WB} 
\end{figure}


\section{Conclusion} \label{sec:conclusion}
In this paper, we constructed a class of arbitrarily high-order accurate positivity-preserving well-balanced 
DG methods for the compressible Euler equations with gravitation. 
A novel well-balanced spatial discretization was specially designed with suitable source term treatments 
and a properly modified HLLC flux, while the desired positivity property was also achieved in the discretization at the same time. 
Based on some technical decompositions as well as several key properties of the admissible states and HLLC flux, 
rigorous positivity-preserving analyses were carried out in theory. 
It was proven that the resulting well-balanced DG schemes with SSP time discretization satisfy 
a weak positivity property, which implies that 
a simple existing limiter can effectively enforce the positivity-preserving property without losing conservation and high-order accuracy.
The proposed methods and analyses work for the Euler system with general EOS. 
Extensive 1D and 2D numerical tests were provided to demonstrate the 
accuracy, well-balancedness, positivity preservation, and high resolution of the proposed schemes. 
It is worth noting that the proposed numerical framework is also readily applicable for designing  
positivity-preserving well-balanced high-order accurate finite volume methods.

\appendix

\section{Extensions to general equation of state}\label{app:gEOS} 

In this appendix, we show that the proposed numerical methods and analyses, which are presented with the ideal EOS in the paper, are readily extensible to general EOS 
\begin{equation}\label{gEOS}
e={\mathcal E}(\rho,p), 
\end{equation}
which satisfies the following condition 
\begin{equation}\label{eq:assumpEOS}
\mbox{if}\quad \rho \ge 0,\quad \mbox{then}\quad p>0~\Leftrightarrow~e > 0.
\end{equation}
Such a condition is reasonable, holds for the ideal EOS \eqref{eq:IEOS}, and was also 
assumed in \cite{zhang2011} to construct positivity-preserving schemes for equations under general EOS. 

First of all, the functions $\alpha_{\bf n} ({\bf U})$,  $\alpha_{\max} ({\bf U})$
$\alpha_\pm ({\bf U})$, which are defined in the paper for the ideal EOS \eqref{eq:IEOS}, should be redefined 
for general EOS \eqref{gEOS} as follows:
\begin{equation}\label{eq:newAlpha}
\alpha_{\bf n} ({\bf U}):= | {\bf u} \cdot {\bf n} | + \hat c, \qquad 
\alpha_{\max} ({\bf U}) := |u| + \hat c, \qquad 
\alpha_\pm ({\bf U}) =  u \pm \hat c,
\end{equation}
where $\hat c= \max\{ \frac{p}{\rho \sqrt{2 e }}, c_s \}$ with $c_s$ denoting  the 
local sound speed. 
With the new definitions \eqref{eq:newAlpha}, all the related 
Lemmas \ref{lem:LFflux}, \ref{HLLC_std_euler:old00}, \ref{HLLC_std_euler:old}, \ref{HLLC_std_euler00}, \ref{HLLC_std_euler}, 
and \ref{HLLC_std_euler2D00}, remain valid for general EOS \eqref{gEOS}. 

\subsection{Well-balancedness}

The HLLC fluxes for general EOS \eqref{gEOS} also satisfy the contact property, i.e., Lemmas \ref{lem:HLLCcontact} and \ref{lem:HLLCcontact2D} can be extended as follows.   

\begin{lemma}[1D contact property]\label{lem:HLLCcontact:gEOS}
For any two states ${\bf U}_L =  \big( \rho_L, 0, \rho_L {\mathcal E}(\rho_L, p) \big)^\top $ 
and ${\bf U}_R = \big(  \rho_R, 0,  \rho_R {\mathcal E}(\rho_R,p) \big)^\top $, the 1D HLLC flux \eqref{HLLC} satisfies 
$$ 
{\bf F}^{hllc} ({\bf U}_L,{\bf U}_R) = ( 0, p, 0 )^\top.
$$
\end{lemma}

\begin{lemma}[2D contact property]\label{lem:HLLCcontact2D:general}
For any two states ${\bf U}_L =  \big( \rho_L, 0, 0, \rho_L {\mathcal E}(\rho_L, p) \big)^\top $ 
and ${\bf U}_R =  \big( \rho_R, 0, 0, \rho_R {\mathcal E}(\rho_R,p) \big)^\top $, the 2D HLLC flux satisfies 
$$ 
{\bf F}^{hllc}\left( {\bf U}_L, {\bf U}_R ; {\bf n}
\right) = ( 0, p {\bf n}^\top, 0 )^\top.
$$
\end{lemma}

Let 
${\bf U}^e({\bf x}) = \Big( \rho^e({\bf x}), {\bf 0}, \rho^e({\bf x}) {\mathcal E} \big( \rho^e({\bf x}),p^e({\bf x}) \big)  \Big)^\top$ 
denote the target stationary hydrostatic solution to be preserved. 
Define  
$${\bf U}_h^e({\bf x}) = {\bf P}_h {\bf U}^e({\bf x}) = \Big( \rho^e_h({\bf x}), {\bf 0}, \rho^e_h({\bf x}) e^e_h({\bf x})  \Big)^\top, \qquad \mbox{with} \quad e_h^e({\bf x}) = {\mathcal E} \big( \rho^e_h({\bf x}),p^e_h({\bf x}) \big) $$
as the standard projection of 
${\bf U}^e({\bf x})$ onto the DG space $[\mathbb V_h^k]^{d+2}$. Note that $p^e_h({\bf x})$, as defined above, may not belong to $\mathbb V_h^k$ for a general EOS.

In order to achieve the well-balancedness for general EOS \eqref{gEOS} by the contact property in 
Lemma \ref{lem:HLLCcontact:gEOS}, 
the 1D modified HLLC flux \eqref{eq:1Dflux} should be accordingly generalized to 
\begin{equation}\label{eq:1Dflux:gEOS}
\widehat{\bf F}_{j+\frac12} = 
{\bf F}^{hllc} \left( \frac{ e^{\star,-}_{j+\frac12}  }{ e_h^e( x_{j+\frac12}^- ) }   {\bf U}_{j+\frac12}^-, \frac{ e^{\star,+}_{j+\frac12}  }{ e_h^e( x_{j+\frac12}^+ ) }   {\bf U}_{j+\frac12}^+ \right),
\end{equation}
where  
$$e^{\star,\pm}_{j+\frac12} := {\mathcal E} \left( \rho^e_h( x_{j+\frac12}^\pm ), p^{e,\star}_{j+\frac12}   \right), \qquad p^{e,\star}_{j+\frac12} =  \frac12 \left( p_h^e( x_{j+\frac12}^- ) 
+ p_h^e( x_{j+\frac12}^+ ) \right).$$ 
Similarly, the 2D modified HLLC flux \eqref{2DWBflux} should be generalized to
\begin{equation}\label{2DWBflux:gEOS}
\widehat{\bf F}_{{\bf n}_{{\mathscr E},K}} = {\bf F}^{hllc} \left(
\frac{  e^{\star,{\tt int}(K)}_h }{ e^{e,{\tt int}(K)}_h }
{\bf U}_h^{{\tt int}(K)},~\frac{ e^{\star,{\tt ext}(K)}_h }{ e^{e,{\tt ext}(K)}_h } {\bf U}_h^{{\tt ext}(K)};~ {\bf n}_{{\mathscr E},K}
\right),
\end{equation}
where  
$$e^{\star,{\tt int}(K)}_h := {\mathcal E} \left( \rho^{e,{\tt int}(K)}_h, p^{e,\star}_h \right), \quad
e^{\star,{\tt ext}(K)}_h := {\mathcal E} \left( \rho^{e,{\tt ext}(K)}_h, p^{e,\star}_h \right), \quad p^{e,\star}_h := \frac12 \big( p^{e,{\tt int}(K)}_h + p^{e,{\tt ext}(K)}_h \big).$$  
We use the same source term discretizations as in the ideal EOS case proposed in the paper. 
Then, one can verify that the resulting DG schemes are well-balanced for the general EOS \eqref{gEOS}, namely, Theorems \ref{thm:1DWB} and \ref{thm:WB2D} remain valid.

\subsection{Positivity}
Following the proofs in the paper,
we can also prove the positivity-preserving property of the resulting DG schemes for the general EOS \eqref{gEOS}. 
Specifically, we have the following conclusions.  

\begin{itemize}
\item For the positivity of the 1D first-order scheme, Theorem \ref{thm:1D1st} also holds for the general EOS, under the CFL condition \eqref{eq:CFL1F} with $\widehat \alpha_j$ redefined as 
$$
\widehat \alpha_j := \widehat \alpha_j^{F}+ \widehat \alpha_j^{S}, \qquad \widehat \alpha_j^{F} =  2\frac
{ e^{\star,-}_{j+\frac12} + e^{\star,+}_{j-\frac12}  } { e_h^e(x_j)  }  \max_{ {\bf U} \in \{    \overline {\bf U}_{j-1} , \overline {\bf U}_j, \overline {\bf U}_{j+1} \}  } \alpha_{\max} ({\bf U}  ),\qquad \widehat \alpha_j^{S}=  \frac{ \left| p^{e,\star}_{j+\frac12} - p^{e,\star}_{j-\frac12} \right|  }{ \overline \rho_j^e \sqrt{2 \overline e_j }   }
$$
\item For the positivity of the 1D high-order schemes, Theorem \ref{thm:1Dhigh} also holds for the general EOS, under the CFL condition \eqref{CFL2high} with $\widetilde \alpha_j$ redefined as
\begin{align*}    
& \widetilde \alpha_j := \widetilde \alpha_j^F + \widetilde \alpha_j^S + \overline \alpha_j^S,
\qquad \widetilde \alpha_j^F := 2~ {\rm max} \left\{ 
\frac{ e^{\star,-}_{j + \frac12}   } {   e_h^e( x_{j + \frac12}^-  )   } , 
\frac{ e^{\star,+}_{j - \frac12}   } {   e_h^e( x_{j - \frac12}^+  )   } \right\}
\max_{ {\bf U} \in \{  {\bf U}_{j-\frac12}^\pm , {\bf U}_{j+\frac12}^\pm \}  } \alpha_{\max} ({\bf U}  )  ,
\\ 
& \widetilde \alpha_j^S := \widehat \omega_1 h_j \max_{1\le \mu \le N} \left\{  \frac{ \left| \big(p_h^e\big)_x (   x_j^{(\mu)}  ) \right|
}{ \rho_h^e( x_j^{(\mu)} ) \sqrt{2 e_h( x_j^{(\mu)} ) } } \right\},
\quad  \overline \alpha_j^S := \widehat \omega_1 \frac{  \left| 
	p_{j+\frac12}^{e,\star} - p_{j-\frac12}^{e,\star} 
	-h_j \sum \limits_{ \mu=1 }^N \omega_\mu (p_h^e)_x(x_j^{(\mu)})
	\right|  }{ \overline \rho_j^e \sqrt{2 \overline e_j } }.
\end{align*}
Notice that $ (p_h^e)_x\big|_{I_j} $ is generally not a polynomial for general EOS, so that the exactness of 
the Gauss quadrature rule is not applicable for the integral $\int_{I_j} (p_h^e)_x {\rm d} x$ to simplify  $\overline \alpha_j^S$.     
\item For the positivity of the 2D first-order scheme, Theorem \ref{thm:2D1st} holds for the general EOS, under the CFL condition 
$$
\Delta t \left(  2 \frac{\widehat \alpha_K^F}{|K|} \sum_{ {\mathscr E} \in \partial K } 
|{\mathscr E}| \frac{  {\mathcal E} ( \overline \rho^e_K,  p^{e,\star}_{ {\mathscr E}, K })   }{  \overline e^e_K  } 
+ \widehat \alpha_K^S
\right)  \le 1.
$$
\item For the positivity of the 2D high-order schemes, Theorem \ref{thm:2Dhigh} holds for the general EOS, under the CFL condition 
\begin{equation*}
\Delta t \left( 
\widetilde \alpha_K^F \frac{ 2  |{\mathscr E}| e^{\star, {\tt int}(K) }_{{\mathscr E},\mu } }
{ |K| e^{e,{\tt int}(K)}_{{\mathscr E},\mu } }
+ \widetilde \alpha_K^S \frac{ \widehat \varpi_{{\mathscr E}}^{(\mu)} }{ \omega_\mu }
\right) \le \frac{ \widehat \varpi_{{\mathscr E}}^{(\mu)} }{ \omega_\mu }, 
\qquad   1\le \mu \le N,~\forall{\mathscr E} \in \partial K,~\forall K \in {\mathcal T}_h,
\end{equation*}
where $e^{\star, {\tt int}(K) }_{{\mathscr E},\mu }: = {\mathcal E} \left( \rho^{e,{\tt int}(K)}_{{\mathscr E},\mu }, p^{e,\star }_{{\mathscr E},\mu } \right)$, and 
\begin{align*}
&\widetilde \alpha_K^F := \max\Big\{
\max_{ {\mathscr E} \in \partial K, 1\le \mu \le N } \alpha_{{\bf n}_{{\mathscr E},K} } ( {\bf U}^{{\tt int}(K)}_{{\mathscr E},\mu } ), 
\max_{ {\mathscr E} \in \partial K, 1\le \mu \le N } \alpha_{ {\bf n}_{{\mathscr E},K} } ( {\bf U}^{{\tt ext}(K)}_{{\mathscr E},\mu } )
\Big\}, 
\\
&
\widetilde \alpha_K^{S} = \max_{1\le q \le Q} \left\{  
\frac{ \left \|  {\bm \nabla} p_h^e (   {\bf x}_K^{(q)}  ) \right \| }  { 
	\rho_h^e ( {\bf x}_K^{(q )} )
	\sqrt{ 
		2 e_h (   {\bf x}_K^{(q)}  ) } } 
\right\} + 
\frac{ \left\|  {\bm a}_K^e \right\| }{ 
	|K| \overline \rho_K^e
	\sqrt{ 2 \overline e_K } },
\\
& {\bm a}_K^e :=  \sum_{ {\mathscr E} \in \partial K }  \left( |{\mathscr E}| \sum_{\mu=1}^N \omega_\mu p^{e,\star}_h ( {\bf x}_{\mathscr E}^{(\mu)} )  {\bf n}_{ {\mathscr E}, K } \right)
- |K| \sum_{ q=1}^Q \varpi_q  {\bm \nabla} p_h^e (   {\bf x}_K^{(q)}  ).
\end{align*}
\end{itemize}
The proofs of the above conclusions are similar to those for the ideal EOS in the paper and thus are omitted here.

\section{On positivity of well-balanced DG methods with modified LF flux}\label{app:DGLF} 
In the paper, we have shown that the proposed DG methods, with the modified HLLC flux \eqref{eq:1Dflux} 
and a special source term discretization, are well-balanced and positivity-preserving, when a simple limiter is applied. 
It was shown in \cite{LiXingWBDG2016,LiXingWBFV2016}  that 
two modified LF fluxes can also serve as effective bases of 
well-balanced DG schemes for isothermal and polytropic equilibria, respectively. 
For completeness of this work as well as comparison purpose, we have also carefully 
investigated the positivity of well-balanced DG methods with those modfied LF fluxes. 
We observe that the modification in the polytropic equilibrium cases changes some   
properties of the LF flux and 
makes the positivity-preserving property of those methods 
questionable.  
We can prove the positivity only when the isothermal equilibria with the ideal EOS are considered, for which the modified LF flux 
is 
\begin{equation}\label{eq:LLFflux} 
\widehat{\bf F}_{j+\frac12} = 
\frac12  \left[ 
{\bf F} ( {\bf U}_{j+\frac12}^- ) 
+ {\bf F} ( {\bf U}_{j+\frac12}^+ ) 
- \alpha_{j+\frac12}^{\tt LF} \rho^{e,{\rm max}}_{j+\frac12}
\left( \frac{ {\bf U}_{j+\frac12}^+ } {  \rho_h^e( x_{j+\frac12}^+) }
- \frac{ {\bf U}_{j+\frac12}^- }{  \rho_h^e( x_{j+\frac12}^-) } \right) 
\right], 
\end{equation} 
where $\rho^{e,{\rm max}}_{j+1/2} \ge \max\{ \rho_h^e( x_{j+1/2}^- ), \rho_h^e( x_{j+1/2}^+ ) \}$, and 
$\alpha_{j+\frac12}^{\tt LF}$ denotes the numerical viscosity parameter satisfying $
\alpha_{j+\frac12}^{\tt LF} \ge  \max_{ {\bf U} \in \{ {\bf U}_{j+1/2}^-, {\bf U}_{j+1/2}^+ \}  } \alpha_{\max} ({\bf U}  ).$
Without loss of generality, here we present only the 1D positivity-preserving conclusions, and the extensions to the multidimensional cases are straightforward. 
Note that the following conclusions hold for 
the flux \eqref{eq:LLFflux} with either local or global numerical viscosity parameter.

We first study the positivity of the 1D first-order scheme.

\begin{theorem}\label{thm:LF1st}
Assume the stationary hydrostatic solution $\{\rho^e,p^e\}$ belongs to isothermal equilibria 
and that the modified LF flux \eqref{eq:LLFflux} is used.  If 
the DG polynomial degree $k=0$ and $\overline {\bf U}_{j} \in G$ for all $j$, we have
\begin{equation*}
\overline {\bf U}_j + \Delta t {\bf L}_j ( {\bf U}_h  ) \in G, \quad \forall j,
\end{equation*}
under the CFL-type condition 
\begin{equation*}
\widehat \alpha_j \Delta t \le h_j,
\end{equation*}
with 
$$\widehat \alpha_j := 
\frac{ \alpha_{j - \frac12}^{\tt LF}  \rho_{j- \frac12}^{e,max} + \alpha_{j + \frac12}^{\tt LF} \rho_{j + \frac12}^{e,max} }{ 2 \overline \rho^e_j }
+  \frac{ \left| p^{e,\star}_{j+\frac12} - p^{e,\star}_{j-\frac12} \right|  }{  \overline \rho_j^e  \sqrt{2 \overline e_j }   }.$$
\end{theorem} 

The proof of Theorem \ref{thm:LF1st} is based on the following decomposition and Lemma \ref{lem3} to show the positivity property of the homogeneous case.
\begin{align*}
\overline {\bf U}_j -\frac{\Delta t}{h_j} \left( \widehat{\bf F}_{j+\frac12} - \widehat{\bf F}_{j-\frac12}   \right) 
& = 
\left[
1 - \frac{\Delta t}{h_j} 
\left(  \frac{ \alpha_{j - \frac12}^{\tt LF}  \rho_{j- \frac12}^{e,max} + \alpha_{j + \frac12}^{\tt LF} \rho_{j + \frac12}^{e,max} }{ 2 \overline \rho^e_j } \right)
\right] \overline{\bf U}_j 
\\
& 
+ \Delta t  \frac{ \alpha_{j + \frac12}^{\tt LF} \rho_{j + \frac12}^{e,max} }{ 2 h_j \overline \rho^e_{j+1} } 
{\bf W}_1 + \Delta t \frac{  \alpha_{j - \frac12}^{\tt LF} \rho_{j - \frac12}^{e,max} }{ 2 h_j \overline \rho^e_{j-1} } 
{\bf W}_2, 
\end{align*}
with 
$$
{\bf W}_1 := \overline {\bf U}_{j+1} - \frac{ \overline \rho^e_{j+1} }{ \alpha_{j + \frac12}^{\tt LF} \rho_{j + \frac12}^{e,max}  } {\bf F}  \big( \overline {\bf U}_{j+1} \big) ,  \qquad 
{\bf W}_2 := \overline {\bf U}_{j-1} + \frac{ \overline \rho^e_{j-1} }{ \alpha_{j - \frac12}^{\tt LF} \rho_{j - \frac12}^{e,max}  } {\bf F}  \big( \overline {\bf U}_{j-1} \big),
$$
both of which belong to $G$ according to 
Lemma \ref{lem:LFflux}. Then, the subsequent part of the proof 
exactly follows the proof of Theorem \ref{thm:1D1st} and thus is omitted.

We then show the weak positivity of 1D high-order schemes.

\begin{theorem}\label{thm:1DhighDGLF}
Assume that the modified LF flux \eqref{eq:LLFflux} is used and that the stationary hydrostatic solution $\{\rho^e,p^e\}$ belongs to isothermal equilibria with 
projections satisfying  
\begin{equation*}
\rho^e_h( x )>0, \quad p^e_h(x)>0, \quad     \forall x \in \mathbb S_j,~~\forall j. 
\end{equation*}
If the numerical solution ${\bf U}_h$ satisfies 
\begin{equation*}
{\bf U}_h ( x ) \in G,  \quad     \forall x \in \mathbb S_j,~~\forall j,
\end{equation*} 
we have
\begin{equation*}
\overline {\bf U}_j + \Delta t {\bf L}_j ( {\bf U}_h  ) \in G,~~\forall j,
\end{equation*}
under the CFL-type condition 
\begin{equation*}
\widetilde \alpha_j \Delta t \le \widehat \omega_1 h_j. 
\end{equation*}
Here 
$
\widetilde \alpha_j := \widetilde \alpha_j^{F} + \widetilde \alpha_j^S + \overline \alpha_j^S,
$
with  $\widetilde \alpha_j^S$ and $\overline \alpha_j^S$ defined in \eqref{eq:tildealphajs}, and $\widetilde \alpha_j^{F}$ redefined as 
$$
\qquad \widetilde \alpha_j^{F} := {\rm max} \left\{ 
\frac{ \rho^{e,{\rm max} }_{j + \frac12}   } {   \rho_h^e( x_{j + \frac12}^-  )   } , 
\frac{ \rho^{e, {\rm max} }_{j - \frac12}   } {   \rho_h^e( x_{j - \frac12}^+  )   } \right\}
\max_{ {\bf U} \in \{  {\bf U}_{j-\frac12}^\pm , {\bf U}_{j+\frac12}^\pm \}  } \alpha_{\max} ({\bf U}  ).
$$
\end{theorem}

\begin{proof}
Similar to the proof of Theorem \ref{thm:1Dhigh}, we define 
\begin{equation}\label{eq:proof2ddfg}
{\bf W}_2 :=  
\eta \widehat \omega_1  \big( {\bf U}_{j-\frac12}^+ +  {\bf U}_{j+\frac12}^- \big) 
-\frac{\Delta t}{h_j} \left( 
\widehat{\bf F}_{j+\frac12}  - \widehat{\bf F}_{j-\frac12}
\right),
\end{equation}
where $\eta$ is an arbitrary parameter in $(0,1]$. 
Plugging the numerical flux \eqref{eq:LLFflux} in \eqref{eq:proof2ddfg}, we can reformulate ${\bf W}_2$ into the following form 
\begin{equation}\label{eq:W2DGLF}
{\bf W}_2 = \left( 
\eta \widehat \omega_1 - \frac{ \Delta t } {h_j} \frac{ \alpha_{j - \frac12}^{\tt LF} \rho_{j- \frac12}^{e,{\rm max}}} 
{ \rho_h^e ( x_{j-\frac12}^+ ) } 
\right) {\bf U}_{j-\frac12}^+ 
+ \left( 
\eta \widehat \omega_1 - \frac{ \Delta t } {h_j} \frac{ \alpha_{j + \frac12}^{\tt LF} \rho_{j+ \frac12}^{e,{\rm max}}} 
{ \rho_h^e ( x_{j+ \frac12}^- ) } 
\right) {\bf U}_{j+\frac12}^-  + \frac{\Delta t}{2h_j}  \sum_{i=1}^4 {\bf W}_2^{(i)}, 
\end{equation}
with
\begin{align*}
& {\bf W}_2^{(1)} :=  \frac{ \alpha_{j - \frac12}^{\tt LF} \rho_{j- \frac12}^{e,{\rm max}} } 
{ \rho_h^e ( x_{j-\frac12}^- ) } 
\left( 
{\bf U}_{j-\frac12}^- + \frac{ \rho_h^e ( x_{j-\frac12}^- ) }{ \alpha_{j - \frac12}^{\tt LF} \rho_{j- \frac12}^{e,{\rm max}} }  {\bf F} \big(  {\bf U}_{j-\frac12}^- \big)
\right), 
\\
& {\bf W}_2^{(2)} :=  \frac{ \alpha_{j - \frac12}^{\tt LF} \rho_{j- \frac12}^{e,{\rm max}} } 
{ \rho_h^e ( x_{j-\frac12}^+ ) } 
\left( 
{\bf U}_{j-\frac12}^+ + \frac{ \rho_h^e ( x_{j-\frac12}^+ ) }{ \alpha_{j - \frac12}^{\tt LF} \rho_{j- \frac12}^{e,{\rm max}} }  {\bf F} \big(  {\bf U}_{j-\frac12}^+ \big)
\right), 
\\
& {\bf W}_2^{(3)} :=  \frac{ \alpha_{j + \frac12}^{\tt LF} \rho_{j+ \frac12}^{e,{\rm max}} } 
{ \rho_h^e ( x_{j+\frac12}^- ) } 
\left( 
{\bf U}_{j+\frac12}^- - \frac{ \rho_h^e ( x_{j+\frac12}^- ) }{ \alpha_{j + \frac12}^{\tt LF} \rho_{j+ \frac12}^{e,{\rm max}} }  {\bf F} \big(  {\bf U}_{j+\frac12}^- \big)
\right), 
\\
& {\bf W}_2^{(4)} :=  \frac{ \alpha_{j + \frac12}^{\tt LF} \rho_{j+ \frac12}^{e,{\rm max}} } 
{ \rho_h^e ( x_{j+\frac12}^+ ) } 
\left( 
{\bf U}_{j+\frac12}^+ - \frac{ \rho_h^e ( x_{j+\frac12}^+ ) }{ \alpha_{j + \frac12}^{\tt LF} \rho_{j+ \frac12}^{e,{\rm max}} }  {\bf F} \big(  {\bf U}_{j+\frac12}^+ \big)
\right).
\end{align*}
Lemmas \ref{lem:LFflux} and \ref{lem2} together imply that ${\bf W}_2^{(i)} \in  G$, $1\le i \le 4$. 
Using equation \eqref{eq:W2DGLF} and Lemma \ref{lem3}, we can conclude 
${\bf W}_2 \in  G$ if $\Delta t$ satisfies 
$$
\Delta t \widetilde \alpha_j^{F} \le \eta \widehat \omega_1 h_j.
$$
The subsequent part of this proof exactly follows the proof of Theorem \ref{thm:1Dhigh}. 
\end{proof}

\bibliographystyle{siamplain}
\bibliography{WBPP,WB2}

\end{document}